\documentclass[11pt]{amsart}
\usepackage{amsmath,amssymb,amsthm,amsfonts,setspace,hyperref,dsfont,yhmath}

\newcommand{\HH}{\mathbb{H}}
\newcommand{\KK}{\mathbb{K}}
\newcommand{\NN}{\mathbb{N}}
\newcommand{\RR}{\mathbb{R}}

\newcommand{\CC}{\mathbb{C}}
\newcommand{\QQ}{\mathbb{Q}}

\newcommand{\ZZ}{\mathbb{Z}}

\newcommand{\norm}[1]{\lVert#1\rVert}
\newcommand{\abs}[1]{\lvert#1\rvert}

\newtheorem{theorem}{Theorem}[section]
\newtheorem{corollary}[theorem]{Corollary}

\newtheorem{lemma}[theorem]{Lemma}
\newtheorem{proposition}[theorem]{Proposition}

\theoremstyle{definition}
\newtheorem{definition}[theorem]{Definition}
\newtheorem{remark}[theorem]{Remark}

\newtheorem{sect}{}

\newtheorem{sect2}{}

\newtheorem{sect15}{Example}
\newtheorem{sect6}{Case}
\newtheorem{sect3}{Case}

\numberwithin{equation}{section}

\begin{document}
\title[Uniform pointwise bounds]
{Uniform pointwise bounds for Matrix coefficients of unitary representations on semidirect products}

%\subjclass[2003]{}


\author[]{ Zhenqi Jenny Wang }
\address{Department of Mathematics\\
        Yale University\\
        New Haven, CT 06520,   USA}
\email{Zhenqi.Wang@yale.edu
}

\begin{abstract} Let $k$ be a local field of characteristic $0$, and let $G$ be a connected semisimple almost $k$-algebraic group. Suppose rank$_kG\geq 1$ and $\rho$ is an excellent representation of $G$ on a finite dimensional $k$-vector space $V$. We construct uniform pointwise bounds for the $K$-finite matrix coefficients restricted on $G$ of all unitary representations of the semi-direct product $G\ltimes_\rho V$ without non-trivial $V$-fixed vectors. These bounds turn out to be sharper than the bounds obtained from $G$
itself for some cases. As an application, we discuss a simple method of calculating Kazhdan constants for various compact subsets of the pair $(G\ltimes_\rho V,V)$.
\end{abstract}

%\date{today}

\maketitle

\section{Introduction and main results}\label{sec:9}
Let $k$ be a local field of char $k=0$. We say that $G$ is a (connected)\emph{ almost $k$-algebraic group} if $G$ is a (connected) $k$-Lie group with finite center for $k$ isomorphic to $\RR$ or $G$ is the group of $k$-rational points of a (connected) linear algebraic group $\tilde{G}$ over $k$ for $k$ non-archimedean or isomorphic to $\CC$. Unless stated otherwise, $G$ denotes a connected semisimple almost $k$-algebraic group with rank$_k(G)\geq 1$ and $\tilde{G}$ denotes its underlying algebraic group; that is, $G=\tilde{G}(k)$ for $k$ non-archimedean or isomorphic to $\CC$.

\subsection{Finite-dimensional representations of $G$} For a finite dimensional vector
space $V$ over $k$, a representation $\rho: G\rightarrow GL(V)$ is called \emph{normal} if $\rho$ is continuous for $k$ isomorphic to $\RR$; or if $\rho$ is a $k$-\emph{rational} map for $k$ non-archimedean or isomorphic to $\CC$.

There is a decomposition $G=G_cG_s$ (resp. $\tilde{G}=\tilde{G}_c\tilde{G}_s$) where $G_c$ (resp. $\tilde{G}_c$) is the product of compact factors (resp. $k$-anisotropic factors) of $G$ (resp. $\tilde{G}$) and $G_s$ (resp. $\tilde{G}_s$) is the product of non-compact factors (resp. $k$-isotropic factors) of $G$ (resp. $\tilde{G}$) when $k$ is isomorphic to $\RR$ (resp. non-archimedean or isomorphic to $\CC$).

Denote by $G_i$ (resp. $\tilde{G}_i$), $1\leq i\leq j$ the non-compact factors (resp. $k$-isotropic factors) of $G$ (resp. $\tilde{G}$). Also set $G_s=\tilde{G}_s(k)$ and $G_i=\tilde{G}_i(k)$ for $k$ non-archimedean or isomorphic to $\CC$. We call these $G_i$ the \emph{non-compact almost $k$-simple factors} of $G$.
\begin{definition} \label{def:2} A normal representation $\rho$ of $G$ on $V$ is called \emph{good} if the $\rho(G_s)$-fixed points in $V$ are $\{0\}$; and $\rho$ is called \emph{excellent} if $\rho(G_i)$-fixed points in $V$ are $\{0\}$ for each non-compact almost $k$-simple factor $G_i$ of $G$.
\end{definition}

In this paper we present an ``upper bound function" for $G$-matrix coefficients
for all unitary representations of $G\ltimes_\rho V$ without non-trivial $V$-fixed vectors if $\rho$ is an excellent representation of $G$ on $V$. Special cases of $SL(2,\KK)\ltimes \KK^2$
and $SL(2,\KK)\ltimes \KK^3$ are considered in \cite{oh} and \cite{tan} for a local field $\KK$. For these cases the following
conditions are satisfied: every orbit is locally closed (intersection of an open and a closed set) in the dual group $\widehat{V}$; and  for each $\chi\in\widehat{V}\backslash \{0 \}$ the stabilizer $S_\chi=\{g\in G\ltimes_\rho V:g\cdot\chi=\chi\}$ is amenable. The first one allows us to use the ``Mackey machine" and the latter one
implies that the $G$-matrix coefficients are bounded by the Harish-Chandra functions. Margulis also used this criterion in \cite[Theorem 2]{Margulis3} to prove Kazhdan's property $(T)$ of the pair $(O_3(\QQ_5)\ltimes \QQ_3^5, \QQ_3^5)$. In fact, if $G$ is a connected almost $k$-algebraic group and $\rho$ is a normal representation, then ``Mackey machine" applies to the semidirect product $G\ltimes_\rho V$ and hence we have complete descriptions of the
dual groups of $G\ltimes_\rho V$ (see \cite[Theorem 7.3.1]{Zimmer} or \cite[Chapter 5.4]{warner}). Therefore any irreducible representation $\pi$ of $G\ltimes_\rho V$
without non-trivial $V$-fixed vectors is induced from the ones on the stabilizers $S_\chi$, $\chi\in\widehat{V}\backslash \{0 \}$. However, for general cases, the  complexity of these stabilizers $S_\chi$ would require heavy analysis calculations.

Our work is an extension of the ideas of R. Howe and E. C. Tan \cite[Chap V, Theorem 3.3.1]{tan}. For $SL(2,\RR)\ltimes \RR^2$, they considered the system of imprimitivity based on $(SL(2,\RR),\widehat{\RR^2})$ instead of ``Mackey machine" to calculate upper bounds of $SO(2)$-finite matrix coefficients restricted on $SL(2,\RR)$. More precisely, the deformation of $SO(2)$-orbits under the natural dual action of the Cartan subgroup on $\widehat{\RR^2}$ gives enough information to get upper bounds of  $SO(2)$-finite matrix coefficients on $SL(2,\RR)$. In their proof the commutativity of $K$, a maximal compact subgroup in $G$, is essential. However, for general examples, the complexity of $G$-orbits and non-commutativity of $K$ do seem to require some new method for handling it; the same is true in an attempt at extending the results to non-archimedean fields.
\subsection{Main results} Let $G$ be a connected semisimple almost $k$-algebraic group with rank$_k(G)\geq 1$, $D$ a maximal $k$-split torus, $B$ a minimal parabolic subgroup containing $D$, $D^+$ the closed positive Weyl chamber of $D$ given by the choice of $B$, $K$
a good maximal compact subgroup such
that the Cartan decomposition $G=KD^+FK$ holds where $F$ is a finite subset of the centralizer of $D$ (Section \ref{sec:2}).

Denote by $\Phi$ be the set of
non-zero roots of $G$ relative to $D$ and by $\Phi^+$ the set of positive roots.
Let $V$ be a finite dimensional vector
space over $k$. Fix a normal representation $\rho$ of $G$ on $V$. Since char $k=0$, then by full reducibility of semisimple groups there is a natural decomposition of $V$ under $\rho$: $V=\oplus_{i=1}^N V_i$ such that for each $i$, $V_i$ is an irreducible representation of $G$. Denote by $\Phi_i$ the set of weights of $\rho$ on $V_i$ relative to $D$ and by $\lambda_i$ and $\varrho_i$ the highest weight and lowest weight respectively, compatible with the ordering of $\Phi$. Set $\Lambda(\Phi_i)=\frac{\mathfrak{q_i}}{2}(\lambda_i-\varrho_i)$, where $\mathfrak{q_i}=(\frac{1}{3})^{(\sharp\Phi_i-1)}$. Moreover, If $\dim V_{\lambda_i}=1$, $\mathfrak{q_i}=(\frac{1}{3})^{(\sharp\Phi_i-2)}$.
\begin{definition}\label{def:3} Set
\begin{align*}
p_{(G,V_i,\Phi_i)}&=\max_{1\leq j\leq n}\Big\{\frac{\text{the coefficient of $\omega_j$ in $\delta_B$}}{\text{the coefficient of $\omega_j$ in $\Lambda(\Phi_i)$}}\Big\}\qquad\text{ and }\notag\\
p_{(G,V,\Phi)}&=\max_{1\leq i\leq N}p_{(G,V_i,\Phi_i)}.
\end{align*}
where $\{\omega_1,\cdots,\omega_n\}$ is the set of simple roots of $\Phi^+$ and $\delta_B$ is the modular function of
$B$.
\end{definition}
For a unitary representation $\pi$ of $G\ltimes_\rho V$, a vector
$v$ in $\pi$ is called $K$-finite if the subspace spanned by $\pi(K)v$ is finite-dimensional. We use the term $K$-finite matrix coefficients of $\pi$ (on $G$) to refer
to its matrix coefficients with respect to $K$-finite vectors (restricted on $G$).
\begin{definition}
Let $\Psi$ be a positive function on $G$, invariant under left and right translations by $K$, and such that $\Psi(g)=\Psi(g^{-1})$, for any $g\in G$.  A unitary representation $\pi$ of $G$ is said to be \emph{$\bigl(K,\,r\Psi\bigl)$ bounded on $G$} where $r>0$, if for any $K$-finite
unit vectors $v$ and $w$, one has the estimate
\begin{align*}
\abs{\langle \pi(g)v,w\rangle}\leq r\dim \langle Kv\rangle^{1/2} \dim \langle Kw\rangle^{1/2}
\Psi(g)\qquad\text{for any }g\in G.
\end{align*}
 Here $\langle Kv\rangle$ denotes the subspace spanned by $Kv$ and similarly for $\langle Kw\rangle$.
\end{definition}
Note that $p_{(G,V,\Phi)}<\infty$ if $\rho$ is excellent. By the following theorem, $p_{(G,V,\Phi)}$ determines a uniform pointwise bound for all $K$-finite matrix coefficients of $G\ltimes_\rho V$ restricted on $G$.

\begin{theorem}\label{th:11}
Fix an excellent representation $\rho$ of $G$ on $V$. Let $m$ be an integer with $2m\geq p_{(G,V,\Phi)}$. Then for any  unitary representation $\Pi$ of $G\ltimes_\rho V$ without non-trivial $V$-fixed vectors, $\Pi$ is $\bigl(K,\,\Xi_G^{1/m}\bigl)$ bounded on $G$. Here $\Xi_G$ is the Harish-Chandra function of $G$ (see Section \ref{sec:12}).
\end{theorem}
Theorem \ref{th:11} does not provide the sharpest uniform pointwise bounds. The novelty of the above theorem lies in the simplicity of our method giving an upper
bound for $K$-finite matrix coefficients.

Proposition \ref{po:3} shows that if each almost $k$-simple factor of $G$ has Kazhdan's property $(T)$ then the upper bounds of $K$-finite matrix coefficients on $G$ can be obtained from semisimple part $G$ itself by using property $(T)$ or higher rank trick. Theorem \ref{th:11} provides upper bounds of $K$-finite matrix coefficients on $G$ even if the above condition fails on $G$. Moreover, using the $\rho(K)$-orbit-deformation method (in proving Theorem \ref{th:11}) one may obtain sharper upper bounds for $K$-finite matrix coefficients on $G$ than those offered by semisimple part itself
(see various examples in Section \ref{sec:19}). For instance, for $G=Sp(2n,\CC)$, $V=\CC^{2n}$ and $\rho$ the standard representation of $G$ on $V$, the best possible decay rate from $G$ itself for $K$-finite matrix coefficients is at most one half of the decay rate (provided by the combination of $\rho(K)$-orbit-deformation method and higher rank trick) (see Remark \ref{re:3}).

We also have the following interesting result, in which the group $G$ is defined by the symmetric or Hermitian form $Q$ on $L^{n+1}$: $\left( \begin{array}{ccc}
0 & 1 & 0\\
1 & 0 &0 \\
0 & 0 &-I_{n-1} \\
\end{array} \right)$ where $L=\RR$, $\CC$, or $\HH$ and $n\geq 2$. $G$ is $SO_0(1,n)$, $SU(1,n)$ or $Sp(1,n)$ respectively  for $L=\RR$, $\CC$ or $\HH$.
\begin{proposition}
Let $V=L^{n+1}$, $n\geq 2$ and $\rho$ be the standard representation of $G$ on $V$. Then for any  unitary representation $\pi$ of $G\ltimes_\rho V$ without non-trivial $V$-fixed vectors is $\bigl(SO(n),\,\Xi_{G}^{1/(n-1)}\bigl)$ bounded on $G$ for $L=\RR$, is $\bigl(S(U_1\times U_n),\,\Xi_G^{1/n}\bigl)$ bounded on $G$ for $L=\CC$ and is $\bigl(Sp(1)\times Sp(n),\,\Xi_G^{1/([\frac{1+2n}{3}]+1)}\bigl)$ bounded on $G$ for $L=\HH$.
\end{proposition}
It turns out that the uniform pointwise bounds for $K$-finite matrix coefficients in above proposition are in fact much sharper than the bounds obtained from $G$ itself: $SO_0(1,n)$ and $SU(1,n)$ don't have property $(T)$; for $L=\HH$ $[\frac{1+2n}{3}]+1<\frac{1}{2}(1+2n)$,  while  there exists an irreducible representation of $G$ such that it is not $\bigl(Sp(1)\times Sp(n),\,\Xi_G^{1/m}\bigl)$ bounded on $G$ for any $m<\frac{1}{2}(1+2n)$ (see Remark \ref{re:3}).

%Moreover, the method in proving Theorem \ref{th:11} shows that one can obtain sharper estimates than those from property $(T)$ for different cases by more detailed description of $\rho(K)$ orbits (see various examples in Section \ref{sec:19}). For instance, for $G=SL(3,\CC)$, $V=L^3$ and $\rho$ is the standard representation of $SL(3,\CC)$ on $\CC^3$; or $G=Sp(4,\CC)$, $V=\CC^4$ and $\rho$ is the standard representation of $Sp(4,\CC)$ on $\CC^4$, then any  unitary representation $\pi$ of $G\ltimes_\rho V$ without non-trivial $V$-fixed vectors is $\bigl(SU(3),\,\Xi_{G}\bigl)$ bounded on $G$ for $G=SL(3,\CC)$; or is $\bigl(Sp(2),\,\Xi_G^{\frac{1}{2}}\bigl)$ bounded on $G$ for $G=Sp(4,\CC)$ (see Example \ref{for:21} and \ref{for:37} in Section \ref{sec:19}); while by using property $(T)$ $\pi$ is $\bigl(SU(3),\,\Xi_{G}^{\frac{1}{2}}\bigl)$ bounded and $\bigl(Sp(2),\,\Xi_G^{\frac{1}{4}}\bigl)$ bounded on $G$ respectively.

It is proved in \cite[Proposition 2.3]{Valette1} and \cite[Proposition p.22]{Valette2} that the pair $(G\ltimes_\rho V,V)$ has Kazhdan's property $(T)$ if $\rho$ is good. Then Theorem \ref{th:11} gives, in the case of local fields with character $0$, the quantitative results. The pointwise bound $\Xi_G^{1/m}$ provides us with a simple and general method of calculating
Kazhdan constants (see Section \ref{sec:13} for definition) for various compact subsets of semisimple
$G$, in particular for any compact subset properly containing $K$.

It is well known that any almost simple $k$-algebraic group $G$ with rank$_kG=1$ does not have property $(T)$ if $k$ is non-archimedean or if Lie algebra $\mathfrak{g}$ of $G$ is isomorphic to $su(n,1)$ or $so(n,1)$ for $k$ archimedean. We have the following example including the case of  $G$ not having property $(T)$:
\begin{proposition}
Suppose rank$_kG=1$ and $\rho$ is irreducible and good on $V$. Denote by $G_1$ is the non-compact simple factor of $G$ and by $K_1$ a good maximal compact subgroup in $G_1$. Let $m$ be the smallest integer such that $2m\geq p'$ where $p'$ is defined in \eqref{for:53}. For any $h\in G_{1}$ such that $h\notin K_{1}$, $\frac{\sqrt{2\bigl(1-\Xi^{1/m}_{G_{1}}(h)\bigl)}}{\sqrt{2\bigl(1-\Xi^{1/m}_{G_{1}}(h)\bigl)}+3}$ is a Kazhdan constant for $((G\ltimes_\rho V,V),\{K_1,h\})$.
\end{proposition}
Let $G$ be a real semisimple connected Lie group  of $\RR$-rank $\ge 2$ without compact factors and with finite center and $\Gamma$
a cocompact torsion free lattice in $G$, which admits a representation $\rho:\Gamma\rightarrow
SL(N,\ZZ)$ without non-trivial invariant subspace with eigenvalue $1$. By Margulis' superrigidity theorem
\cite{Margulis},  the
representation $\rho$ of $\Gamma$ extends to a homomorphism
$G\rightarrow SL(n,\RR)$ (otherwise we consider a finite cover of $G$).  Combined with some results of A. Katok and R. Spatzier in  \cite{Spatzier1} and \cite{Spatzier}, Theorem \ref{th:11} in the case of archimedean fields yields an application in obtaining
tame estimates of cocycle equations on $s$ order Sobolev space $H^s(G\ltimes \RR^N/\Gamma\ltimes\ZZ^N)$ \cite{Zhenqi}.

\noindent{\bf Acknowledgements.} I would like to thank  Professor Roger Howe for
valuable suggestions and for his encouragement.

\section{Preliminaries on almost $k$-algebraic groups}
We denote by $\mathfrak{g}$ the Lie algebra of $G$ (resp. $\tilde{G}$) when $k$ is archimedean (resp. non-archimedean). Let $\mathfrak{g}_k$ be the $k$-Lie algebra. Throughout the paper $V$ will denote a finite-dimensional vector space over $k$.
\subsection{Unipotent, nilpotent elements and exponential map}\label{sec:3} Let $S$ be a connected $k$-linear algebraic group. Denote by $S^u$ the set of unipotent elements of $S$
and $\mathfrak{s}^{(n)}$ the subspace of nilpotent elements of $\mathfrak{s}$ where $\mathfrak{s}$ is the Lie algebra of $S$. Then  $x\in \mathfrak{s}^{(n)}$
implies that
\begin{align}\label{for:50}
    \exp x:=\sum_{i\geq 0}(i!)^{-1}x^i
\end{align}
belongs to $S^u$. Conversely, if $g\in S^u$, then the logarithm
\begin{align}\label{for:62}
    \ln g:=\sum_{i>0}(-i)^{-1}(1-g)^i
\end{align}
belongs to $\mathfrak{s}^{(n)}$. The set $S^u$ and $\mathfrak{s}^{(n)}$ are $k$-subvarieties in $S$ and $\mathfrak{s}$ respectively. The maps
$\exp :\mathfrak{s}^{(n)}\rightarrow S^u$ and $\ln: S^u\rightarrow \mathfrak{s}^{(n)}$ are inverses of each other, biregular and defined over $k$ (see \cite[Chapter 0.20]{Margulis}).
\begin{remark}
The condition char $k=0$ is necessary in guaranteeing the existence of exponential map (resp. logarithm map) of nilpotent elements (resp. unipotent elements) in $\mathfrak{s}$ (resp. group $S$).
\end{remark}
Under an algebraic group morphism the image of any unipotent element is unipotent and every unipotent $k$-algebraic group over a field of
characteristic $0$ is connected. Moreover, we have:
\begin{proposition}(\text{see }\cite{Margulis})\label{po:4}
If $\alpha:S\rightarrow S'$ is a $k$-group morphism, then for each $u\in S^u$ we have $\alpha(u)=\exp\big(d\alpha(\ln u)\big)$,
where $d\alpha:\mathfrak{s}\rightarrow \mathfrak{s}'$ is the differential of $\alpha$ and $\mathfrak{s}'$ is Lie algebra
of $S'$.
\end{proposition}
\subsection{Cartan decomposition}\label{sec:2} Let $D$ (resp. $\tilde{D}$) be a maximal $k$-split Cartan subgroup (resp. $k$-split torus) in $G$ (resp. $\tilde{G}$) and
$B$ (resp. $\tilde{B}$) a minimal parabolic $k$-subgroup of $G$ (resp. $\tilde{G}$) containing $D$ (resp. $\tilde{D}$) when $k$ is archimedean (resp. non-archimedean). For the non-archimedean case, set $D=\tilde{D}(k)$ and $B=\tilde{B}(k)$. Let $X(D)$ (resp. $X(\tilde{D})$) denote the set of characters of $D$ (resp. characters of $\tilde{D}$ over $k$) whose ordering is induced from $B$ (resp. $\tilde{B}$). Denote by $ X^+$ the set of positive
characters in $X(D)$ (resp. $X(\tilde{D})$) with respect to that ordering.

When $k$ is archimedean, we set
\begin{align*}
    k^0=\{x\in\RR\mid x\geq 0\}\quad\text{ and }\quad\tilde{k}=\{x\in\RR\mid x\geq 1\}.
\end{align*}
When $k$ is non-archimedean, we fix a uniformizer $q$ of $k$ such that $\abs{q}^{-1}$ is the cardinality
of the residue field of $k$, and set
\begin{align*}
    k^0=\{q^n\mid n\in\ZZ\}\quad\text{ and }\quad\tilde{k}=\{q^{-n}\mid n\in\NN\}.
\end{align*}
We set
\begin{gather*}
  D^0=\{d\in D\mid \chi(d)\in k^0\text{ for each }\chi\in X(D)\,\,(\text{resp. }X(\tilde{D}))\};\text{ and }\\
   D^+=\{d\in D\mid \chi(d)\in \tilde{k}\text{ for each }\chi\in X^+\}.
\end{gather*}
Equivalently $D^+=\{d\in D^0\mid \abs{\chi(d)}\geq 1\}$ for each $\chi\in X^+$. We call
$D^+$ a positive Weyl
chamber of $G$.

Let $Z$ (resp. $\tilde{Z}$) denote the centralizer of $D$ (resp. $\tilde{D}$) in $G$ (resp. $\tilde{G}$) for $k$ archimedean (resp. $k$ non-archimedean) and set $Z=Z(k)$ for $k$ non-archimedean. Since $X(Z)$ (resp. $X(\tilde{Z})$) can be considered
as a subset of $X(D)$ (resp. $X(\tilde{D})$) in a natural way, it has an induced ordering from this inclusion. Define
\begin{gather*}
  Z_+=\{z\in Z\mid \abs{\chi(z)}\geq 1\text{ for each }\chi\in X(Z)^+\};\text{ and }\\
   Z_0=\{z\in Z\mid \abs{\chi(z)}=1\text{ for each }\chi\in X(Z)^+\}.
\end{gather*}

For any subgroup $S$ of $G$, $N_G(S)$ denotes the normalizer of $S$, $C_G(S)$ denotes the
centralizer of $S$ and $Z(S)$ denotes the center of $S$.
\begin{lemma} \label{le:5} There exists a maximal compact subgroup $K$ of $G$ such that
\begin{enumerate}
  \item $N_G(D)\subseteq KD$,
  \item the Cartan decomposition $G = K(Z_+/Z_0)K$ holds, in the sense that for any $g\in G$, there are elements
$z\in Z_+$ (unique up to mod $Z_0$) such that $g\in KzK$.
\end{enumerate}
\end{lemma}
See \cite{Helgason} for archimedean case and see \cite{Borel} and \cite{Si} for non-archimedean
case. In general, the positive Weyl chamber $D^+$ has finite index in $Z_+/Z_0$. Hence
for some finite subset $F\subset C_{G}(D)$, $G=K(D^+F)K$, i.e., for any $g\in G$, there exist unique
elements $d\in D^+$ and $\omega\in F$ such that $g\in Kd\omega K$.

A maximal compact subgroup $K$ is
called a \emph{good} maximal compact subgroup of $G$ if it satisfies the properties listed in the
above lemma.
\begin{remark}
We have $G=KD^+K$ if $k$ is archimedean or; $G$ is split over $k$ or; if $G$ is quasi-split or split over an unramified extension over a non-archimedean
local field $k$.
\end{remark}

\subsection{Roots and weights relative to a $k$-split torus} \label{sec:8}
Let $\rho$ be a normal representation of $G$ on $V$. A character $\chi\in X(D)$ (resp. $\chi\in X(\tilde{D})$) is said to be a \emph{weight} of $D$ in the
representation $(\rho,V)$ for $k$ archimedean (resp. non-archimedean) if there exists a non-zero vector $v\in V$ such that
\begin{align}\label{for:101}
\rho(c)v=\chi(c)v\quad \text{ for all }c\in D;
\end{align}
and the corresponding \emph{weight space} $V_\chi$ of $\chi$ is described as all vectors in $V$ satisfying \eqref{for:101}.

Since char$(k)=0$, then by full reducibility of semisimple groups there is a decomposition of $V$ under $\rho$:
$V=\sum_{\mu\in\Phi_1} V_\mu$ where $V_\mu$ is the weight space of $\mu$ and $\Phi_1$ is the set of weights of $D$.
Notice that if char$(k)\neq 0$ then full reducibility fails.

Non-trivial characters of $X(D)$ (resp. $X(\tilde{D})$) in the adjoint representation of $G$  are said to be the \emph{roots} of $G$ for $k$ archimedean (resp. non-archimedean). Denote the set of roots by $\Phi$. For each $\omega\in\Phi$ let $\mathfrak{g}_\omega$ be the corresponding
root space, i.e.,
\begin{align*}
\mathfrak{g}_\omega=\{v\in\mathfrak{g}_k: \text{Ad}(d)(v)=\omega(d)v,\quad \forall d\in D \}.
\end{align*}

\begin{lemma}\label{le:4} Let $\rho$ be a normal representation of $G$, then the following hold:
\begin{enumerate}
  \item If $\ker(\rho)\bigcap G_s\subset Z(G)$ then every root of $G$ is a rational combination of weights of $G$ \label{for:35};
  \item the sum of all weights of $\rho$ is trivial\label{for:34}.
 \end{enumerate}
\end{lemma}
\begin{proof} \eqref{for:35} Let $\Phi$ be the set of roots. By full reducibility of $G$, there is a decomposition of $V$ under $\rho$: $V=\sum_{\mu\in\Phi_1} V_\mu$ where $V_\mu$ is the weight space of $\mu$ and $\Phi_1$ is the set of weights.

For each $\omega\in\Phi$ choose $0\neq u\in\mathfrak{g}_\omega$ and let $U_\omega$ denote the one-parameter subgroup $\exp(tu)$, $t\in k$. Since $k$ is infinite \eqref{for:50} and \eqref{for:62} yield that the subgroup $U_\omega$ is infinite. The assumption $\ker(\rho)\bigcap G_s\subset Z(G)$   implies that  there exists $\chi\in\Phi_1$ such that $\rho(U_\omega)$ is nontrivial on $V_\chi$ and  maps  $V_\chi$ into $\Sigma_{j\geq 0} V_{\chi+j\omega}$ \cite{humph}. Then it follows that $$j\omega=(\chi+j\omega)-\chi\qquad\text{where } 0\neq j\in\NN.$$
Then we finish the proof for \eqref{for:35}.

\eqref{for:34} Let $E=X(D)\otimes\RR$ (resp. $E=X(\tilde{D})\otimes\RR$), then $\Phi$ is a root system of $E$ and
the factor group $N_G(D)/Z_G(D)$ (resp. $N_{\tilde{G}}(\tilde{D})/Z_{\tilde{G}}(\tilde{D})$ ) coincides  with the weyl group of the root system $\Phi$ for $k$ archimedean (resp. non-archimedean) (see \cite{Helgason} and \cite{Margulis}). Moreover, there exists $W\subset N_{\tilde{G}}(\tilde{D})(k)$ be a complete set of representatives for $k$ non-archimedean \cite[Chapter 0.27]{Margulis}.

It is clear that the sum of all weights in $E$ is invariant under the action of the Weyl group. To prove \eqref{for:34}, it is sufficient to prove the following statement:  the only element in $E$ invariant under the Weyl group is $0$. Let $E_{in}$ be the subset of $E$ containing all vectors invariant under the Weyl group. It is obvious that $E_{in}$ is a subspace. Suppose $E_{in}\neq 0$.  Note that there exists
a positive definite inner product on $E$ invariant under the Weyl group \cite[Chapter 0.26]{Margulis}. Denote by $E^\bot_{in}$ the orthogonal complement of $E_{in}$
under the inner product. Then $E^\bot_{in}$ is also invariant under the Weyl group. For any $\omega\in\Phi$, we have a unique decomposition $\omega=\omega_1+\omega_2$ where $\omega_1\in E_{in}$ and $\omega_2\in E^\bot_{in}$. There exists $w$ in the Weyl group such that $w(\omega)=-\omega$, then we have
\begin{align*}
-\omega_1-\omega_2=-\omega= w(\omega)=w(\omega_1)+w(\omega_2)=\omega_1+w(\omega_2)
\end{align*}
By uniqueness of the decomposition, it follows that $\omega_1=0$, that is, $\omega\in E^\bot_{in}$. Then immediately we find that $\Phi\subset E^\bot_{in}$, which implies that $E\subset E^\bot_{in}$. Then $E_{in}=0$ which contradicts the assumption. Hence our claim is proved.
\end{proof}
\begin{remark}
The condition $\ker(\rho)\bigcap G_s\subset Z(G)$ is weaker than the condition $\rho$ is excellent. If $\rho$ is irreducible, then the two conditions are equivalent.
\end{remark}

Next, we will give a detailed description of good maximal compact subgroups for different $k$.
\section{Good maximal compact subgroups in $G$}\label{sec:4}
\subsection{Maximal compact subgroups in $G$ when $k=\CC$}\label{sec:5}  For each $\omega\in \Phi$ there exists $X_\omega\in\mathfrak{g}_\omega$ such that the following $\RR$-subspace
\begin{align}\label{for:92}
    \mathcal{K}&=\sum_{\omega\in\Phi^+}\RR(\textrm{i}[X_\omega,X_{-\omega}])+\sum_{\omega\in\Phi^+}\RR(X_\omega+X_{-\omega})\notag\\
    &+\sum_{\omega\in\Phi^+}\RR(\textrm{i}(X_\omega-X_{-\omega}))
\end{align}
is a compact $\RR$-subalgebra and the $\RR$-Lie group $K$ in $G$ with Lie algebra $\mathcal{K}$ is a maximal compact subgroup in $G$ \cite[Chapter III]{Helgason}. Let
\begin{align}\label{for:91}
\mathfrak{k}&=\{X_\omega+X_{-\omega}:\omega\in\Phi^+\},\quad \mathfrak{u}^+=\{X_\omega:\omega\in\Phi^+\}\quad\text{ and }\notag\\
\mathfrak{u}^-&=\{X_\omega:-\omega\in\Phi^+\}.
\end{align}
\subsection{Maximal compact subgroups in $G$ when $k=\RR$}\label{sec:6} In this part, we follow the notations and quote the conclusions from \cite[Chapter VI]{Helgason} with minor modifications. If $\mathfrak{g}=\mathfrak{k}_0+\mathfrak{p}_0$ is a Cartan decomposition of $\mathfrak{g}$ and
$\mathfrak{k}_0$ is the set of fixed points of the corresponding Cartan involution, then the Lie subgroup $K$ in $G$ with Lie algebra $\mathfrak{k}_0$ is a maximal compact subgroup of $G$. Let $\mathfrak{g}_{\CC}$ be the complexification of $\mathfrak{g}$, put $\mathfrak{u}=\mathfrak{k}_0+i\mathfrak{p}_0$ and let $\tau_1$ and $\tau_2$
denote the conjugations of $\mathfrak{g}_{\CC}$ with respect to $\mathfrak{g}$ and $\mathfrak{u}$. The automorphism of $\tau_1\tau_2$ of $\mathfrak{g}_{\CC}$ will be denoted by $\vartheta$.

Let $\mathcal{D}_{\mathfrak{p}_0}$ denote any maximal abelian
subspace of $\mathfrak{p}_0$ and let $\mathcal{D}_0$ be any maximal abelian subalgebra of $\mathfrak{g}$ containing $\mathcal{D}_{\mathfrak{p}_0}$. Obviously
$\mathcal{D}_{\mathfrak{p}_0}=\mathcal{D}_0\bigcap \mathfrak{p}_0$. We put $\mathcal{D}_{\mathfrak{k}_0}=\mathcal{D}_0\bigcap \mathfrak{k}_0$.
Let $\mathcal{D}$ denote the subspace of $\mathfrak{g}_{\CC}$ generated by $\mathcal{D}_0$. Then $\mathcal{D}$ is a Cartan subalgebra of $\mathfrak{g}_{\CC}$. Let $\mathcal{D}^*=\mathcal{D}_{\mathfrak{p}_0}+i\mathcal{D}_{\mathfrak{k}_0}$.  We denote by $\Psi(\mathcal{D},\,\mathfrak{g}_{\CC})$ the set of nontrivial roots of $(\mathfrak{g}_{\CC},\,\mathcal{D})$ and denote by $\mathfrak{g}_{\CC}^\omega$ the corresponding root space of $\omega$ in $\mathfrak{g}_{\CC}$ for each $\omega\in\Psi(\mathcal{D},\,\mathfrak{g}_{\CC})$. Since each root $\omega\in\Psi(\mathcal{D},\,\mathfrak{g}_{\CC})$ is real valued on $\mathcal{D}^*$ we get in this way an ordering of $\Psi(\mathcal{D},\,\mathfrak{g}_{\CC})$.  Let $\Psi^+$ denote the set of positive roots. Now for each $\omega\in\Psi(\mathcal{D},\,\mathfrak{g}_{\CC})$ the linear function $\omega^{\tau_1}$, $\omega^{\tau_2}$, and $\omega^{\vartheta}$
defined by
\begin{align*}
&\omega^{\tau_1}(h)=\overline{\omega(\tau_1h)},\quad \omega^{\tau_2}(h)=\overline{\omega(\tau_2h)},\quad \omega^{\vartheta}(h)=\omega(\vartheta h)
\end{align*}
for any $h\in\mathcal{D}$ are again members of $\Psi(\mathcal{D},\,\mathfrak{g}_{\CC})$. The root $\omega$ is trivial on $\mathcal{D}_{\mathfrak{p}_0}$ if and only if $\omega=\omega^\vartheta$. We divide the positive
roots in two classes as follows:
\begin{align*}
    P_+=\{\omega:\omega\in\Psi^+,\omega\neq \omega^\vartheta\}\,\text{ and }\, P_-=\{\omega:\omega\in\Psi^+,\omega=\omega^\vartheta\}.
\end{align*}
Define
\begin{align*}
P_+^1=\{\omega\in P_+:\omega\neq \omega^{\tau_1}\}\quad\text{ and }\quad P_+^2=\{\omega\in P_+:\omega=\omega^{\tau_1}\}.
\end{align*}
For $\omega\in \Psi(\mathcal{D},\,\mathfrak{g}_{\CC})$ let $\mathfrak{g}^\omega_{\RR}$ be the real vector space spanned by
\begin{align*}
\{x_\omega+\tau_1x_\omega:x_\omega\in \mathfrak{g}^\omega_{\CC}\}.
\end{align*}
It is clear that $\mathfrak{g}^\omega_{\RR}=\mathfrak{g}^{\psi}_{\RR}$ if $\psi=\omega^{\tau_1}$. Furthermore, $\dim \mathfrak{g}^\omega_{\RR}=2$ if $\omega\in P_+^1$ and $\dim \mathfrak{g}^\omega_{\RR}=1$ if $\omega\in P_+^2$.

The following follows from (the proof) of Theorem $3.4$ in \cite[Chapter VI]{Helgason}:
\begin{proposition}\label{po:5}
Let $\mathfrak{n}=\sum_{\omega\in P_+}\mathfrak{g}_{\CC}^\omega$, $\mathfrak{n}_0=\mathfrak{n}\bigcap\mathfrak{g}$, then
\begin{enumerate}
  \item $\mathfrak{g}=\mathfrak{k}_0+\mathcal{D}_{\mathfrak{p}_0}+\mathfrak{n}_0,\quad\text{direct vector space sum.}$

  \item For any $X_\omega\in\mathfrak{g}_{\RR}^\omega$, $X_\omega\in \mathfrak{k}_0$ if $\pm\omega\in P_-$; and $X_\omega+\tau_2X_\omega\in\mathfrak{k}_0$ if $\pm\omega\in P_+$.

  \item $\omega^{\tau_1}=\omega$ on $\mathcal{D}_{\mathfrak{p}_0}$; and $\omega^{\tau_2}=-\omega$ for any $\omega\in\Psi(\mathcal{D},\,\mathfrak{g}_{\CC})$.
\end{enumerate}
\end{proposition}
Select a basis $\{X^1_\omega\,,X^2_\omega\}$ (resp. $\{X^1_\omega\}$) of $\mathfrak{g}_{\RR}^\omega$ for each $\omega\in P^1_+$ (resp. $\omega\in P^2_+$). Set $\delta(\omega)=\{1,2\}$ (resp. $\delta(\omega)=\{1\}$) if $\omega\in P^1_+$ (resp. $\omega\in P^2_+$). Let
\begin{align}\label{for:97}
\mathfrak{k}&=\{X^i_\omega+\tau_2X^i_{\omega}:\quad\omega\in P_+,\,i\in \delta(\omega)\},\notag\\
\mathfrak{u}^+&=\{X^i_\omega:\quad\omega\in P_+^1,\,i\in \delta(\omega)\}, \notag\\
\mathfrak{u}^-&=\{\tau_2X^i_\omega:\quad\omega\in P_+^1,\,i\in \delta(\omega)\}.
\end{align}
%\begin{align*}
%\mathfrak{k}_1&=\{X^i_\omega+\tau_2X^i_{\omega}:\quad\omega\in P_+^1\,,i=1,2\},\notag\\
%\mathfrak{k}_2&=\{X_\omega+\tau_2X_{\omega}:\quad\omega\in P_+^2\},\notag\\
%\mathfrak{u}_1^+&=\{X^i_\omega:\quad\omega\in P_+^1\,,i=1,2\}, \notag\\
%\mathfrak{u}_2^+&=\{X_\omega:\quad\omega\in P_+^2\}, \notag\\
%\mathfrak{u}_1^-&=\{\tau_2X^i_\omega:\quad\omega\in P_+^1\,,i=1,2\}, \notag\\
%\mathfrak{u}_2^-&=\{\tau_2X_\omega:\quad\omega\in P_+^2\}.
%\end{align*}
Here we use the unified notations compatible with \eqref{for:91}.

Let $D$ be the connected Lie group in $G$ with Lie algebra $\mathcal{D}_{\mathfrak{p}_0}$. Denote by $\mathcal{D}_{\mathfrak{p}_0}(\omega)$ the restriction of $\omega$ to $\mathcal{D}_{\mathfrak{p}_0}$ for $\pm\omega\in P_+$ and  let $\Phi$ be the set of exponentials of non-trivial restricted roots, i.e.,
\begin{align*}
\Phi=\{\exp(\mathcal{D}_{\mathfrak{p}_0}(\omega)):\pm\omega\in P_+\}
\end{align*}
where $\exp(\mathcal{D}_{\mathfrak{p}_0}(\omega))(d)=\exp(\omega(\ln d))$ for any $d\in D$. The ordering in $\Psi(\mathcal{D},\,\mathfrak{g}_{\CC})$ induces a compatible ordering in $\Phi$.

\begin{remark}\label{re:5}
By using \eqref{for:92} and Proposition \ref{po:5}, for $k$ archimedean the following hold
\begin{enumerate}
  \item for any $v\in \mathfrak{k}$, $\exp(tv)$ ($t\in \RR$) is a one-parameter subgroup in $K$,
  \item the Lie algebra generated by $\{u:u\in\mathfrak{u}^+\}$ is $\mathfrak{g}^+=\sum_{\omega\in\Phi^+} \mathfrak{g}_\omega$ and
the Lie algebra generated by $\{u:u\in\mathfrak{u}^-\}$ is $\mathfrak{g}^-=\sum_{-\omega\in\Phi^+} \mathfrak{g}_\omega$.\label{for:58}
\item \label{for:1}define a map $\tau_0:\mathfrak{u}^+\rightarrow \mathfrak{u}^-$ as follows: $\tau_0(X_\omega)\rightarrow X_{-\omega}$ for $\omega\in\Phi^+$ for $k=\CC$; and
$\tau_0(u)\rightarrow \tau_2u$ for $u\in\mathfrak{u}^+$ for $k=\RR$. Then $\exp\bigl(t(u+\tau_0 u)\bigl)$ ($t\in \RR$) is a one-parameter subgroup in $K$ for any $u\in\mathfrak{u}^+$.
  \end{enumerate}
\end{remark}
Set $D^+=\{d\in D\mid \chi(d)\geq 1\}$ for each $\chi\in \Phi^+$ and call
$D^+$ a positive Weyl chamber of $G$.
\begin{remark}
For $K$ defined in Section \ref{sec:5} or Section \ref{sec:6}, we have $G=KD^+K$ (see \cite{Helgason}).
\end{remark}
\subsection{Maximal compact subgroups in $G$ if $k$ is non-archimedean}\label{sec:7}
Let $K$ be a good maximal compact subgroup in $G$ and let
$d(\,,\,)$ be a metric on $G$ induced from an absolute value on $k$. Since $K$ is compact and open (see \cite[Appendix]{Serre}), from \eqref{for:50} and \eqref{for:62} we see that
\begin{lemma}\label{le:6}
There exists $1>\sigma'>0$ such that for any $t\in k$ and any $x\in \tilde{G}^u\bigcap G$ with $d(x,e)\leq 1$, if $\abs{t}\leq \sigma'$  then $\exp(t\ln x)\in K$.
\end{lemma}
For each $\omega\in\Phi$, choose a basis $\{X^1_\omega,\cdots, X^{n}_\omega\}$ of $\mathfrak{g}_{\omega}$ where $n=\dim_k\mathfrak{g}_{\omega}$. Let
\begin{align}\label{for:96}
\mathfrak{u}^+&=\{X_\omega^i:1\leq i\leq \dim_k\mathfrak{g}_{\omega},\,\omega\in \Phi^+\}\qquad\text{ and }\notag\\
\mathfrak{u}^-&=\{X_\omega^i:1\leq i\leq \dim_k\mathfrak{g}_{\omega},\,-\omega\in \Phi^+\}.
\end{align}
\begin{remark}\label{re:11}
As an immediate consequence of Lemma \ref{le:6} we have the following statement: there exists $0<\sigma\in\NN$ such that $\exp(tu)\in K$ for all $t\in k$ with $\abs{t}\leq \abs{q}^{\sigma}$
and $u\in \mathfrak{u}^+\bigcup \mathfrak{u}^-$.
\end{remark}

%Let $U^+$ be the (unipotent) algebraic subgroup of $\tilde{G}$ whose Lie algebra is $\sum_{\omega\in\Phi^+} \mathfrak{g}_{\omega}$.
%The following facts about $G$ are well known (see \cite[Chapter 1.2]{Margulis}). The multiplication map induces an isomorphism of
%$k$-variety of $U^-C_{\tilde{G}}(\tilde{D})U^+$ onto a Zariski dense and open subset $E^+$ of $\tilde{G}$, and
%$E^+\bigcap G=U^-(k)C_{G}(D)U^{+}(k)$. A similar result is true for the subset
%$E^-$ in $\tilde{G}$ which is defined as $E^-=U^+C_{\tilde{G}}(\tilde{D})U^-$. Also,
%$E^-\bigcap G=U^+(k)C_{G}(D)U^{-}(k)$.
%\begin{remark}\label{re:8}
%Since $G$ is Zariski dense in $\tilde{G}$ (see \cite[Chapter 0.24]{Margulis}), from above argument it follows that the sets
%$E^+\bigcap G$ and $E^-\bigcap G$ are both Zariski dense in $\tilde{G}$.
%\end{remark}

\section{Preliminaries on matrix coefficients }
In this section we list some notations and well-known properties about matrix coefficients which will be used in this paper.
\begin{definition}
For a locally compact group $S$, a (continuous) unitary representation $\pi$
of $S$ is said to be strongly $L^p$ if there is a dense subset $W$ in the Hilbert space $\mathcal{H}$ attached
to $\pi$ such that for any $x$ and $y$ in $W$, the matrix coefficient $g\rightarrow \langle \pi(g)v,w\rangle$ lies in $L^p(S)$.
We say $\pi$ is strongly $L^{p+\epsilon}$ if it is strongly $L^q$ for any $q > p$.
\end{definition}
Since the matrix coefficients of a unitary representation with respect to unit vectors
are bounded by $1$, a strongly $L^q$ representation is also strongly $L^p$ for any $p\geq q$.
\subsection{The Harish-Chandra function $\Xi_G$}\label{sec:12} We denote by $\delta_B$ the modular function of
$B$; in particular for $a\in D^0$,
\begin{align}\label{for:114}
\delta_B(a)=\prod_{\alpha\in \Phi^+}\abs{\alpha(a)}^{m_\alpha}
\end{align}
where $m_\alpha$ denotes the multiplicity of $\alpha$. The Harish-Chandra function $\Xi_G$ is defined by
\begin{align*}
\Xi_G(g)=\int_K\delta_B(gk)^{-1/2}dk
\end{align*}
 As is well known, $\Xi_G$ is the diagonal matrix coefficient $g\rightarrow\langle\text{Ind}_B^G(1)(g)f_0,f_0\rangle$ where
$\text{Ind}_B^G(1)$ is the representation which is unitarily induced from the trivial representation
$1$ and $f_0$ is its unique (up to scalar) $K$-invariant unit vector. In fact, $f_0$ is given by
\begin{align*}
    f_0(kb)=\delta^{1/2}_B(b)\qquad\text{for }k\in K, \,\,b\in B.
\end{align*}
We list some well-known properties of $\Xi_G$ (see \cite{Ha1}, \cite{Si} and \cite{warner}):
\begin{proposition}\label{po:2}
\begin{enumerate}
  \item $\Xi_G$ is a continuous bi-$K$-invariant function of $G$ with values in $(0, 1]$.
  \item \label{for:41}For any $\epsilon>0$, there exist constants $c_1$ and $c_2(\epsilon)$ such that
\begin{align*}
c_1\delta_B^{-\frac{1}{2}}(b)\leq\Xi_G(b)\leq c_2(\epsilon)\delta_B^{-\frac{1}{2}+\epsilon}(b)\qquad\text{ for all }b\in B.
\end{align*}
Then the formal sum $\xi(\Phi):=\frac{1}{2}\sum_{\alpha\in\Phi^+}m_\alpha\alpha$ determines the decay rate of $\Xi_G$.
\item $\Xi_G$ is $L^{2+\epsilon}(G)$-integrable for any $\epsilon>0$\label{for:42}.
\end{enumerate}
\end{proposition}
For $k$ archimedean we can write the Haar measure $dg$ of $G$ in terms of the Cartan decomposition $KD^+K$: $dg = \Delta(b)dk_1dbdk_2$
where $\Delta(b)$ is a positive function on $D^+$ satisfying
\begin{align*}
d_1(t)\delta_B(b)\leq \Delta(b)\leq d_2\delta_B(b)
\end{align*}
for all $b\in D^+_t= \{g\in D^+: \abs{\alpha(g)}\geq t\text{ for all } \alpha\in\Phi^+\}$ and for some constants $d_1(t)$ and $d_2$ if $t>1$ (see \cite{howe} and \cite[Proposition
5.2.8]{Kn}).

For $k$ non-archimedean and for any bi-$K$-invariant function $f$ of $G$, we have
\begin{align*}
\int_{G}\abs{f(g)}^pdg=\sum_{d\omega\in D^+F}\text{Vol}(Kd\omega K)\abs{f(d)}^p\qquad\text{ for any }p>0.
\end{align*}
Moreover, there exist positive constants $c_1$ and $c_2$ such that $c_1\delta_B(d)\leq\text{Vol}(Kd\omega K)\leq c_2\delta_B(d)$ for all $d\omega\in D^+F$ (see \cite[Lemma 4.1.1]{Si}). Then we have the following result (detailed proof can be found in \cite[Lemma 7.3]{oh}):
\begin{lemma}\label{le:2}
Let $f$ be a bi-$K$-invariant continuous function on $G$. If $\int_{D^+}\abs{f(a)}^p\delta_B(a)da<\infty$ for some $p>0$ then $f\in L^p(G)$.
\end{lemma}

\subsection{Useful results about $K$-matrix coefficients}The following follows from (the proof) of \cite[Corollary in pg. 108]{cow1}:
\begin{theorem} \label{th:10} For a connected semisimple almost $k$-algebraic group $G$ and its unitary representation $\pi$,
if $\pi$ is strongly $L^{2p+\epsilon}$ for some positive integer $p$, then $\pi$ is $\bigl(K,\,\Xi_G^{1/p}\bigl)$ bounded on $G$.
\end{theorem}
Even though it is assumed that $G$ is an semisimple algebraic group in \cite[Theorem 2]{cow1}, the proof works for any semisimple almost $k$-algebraic group case
as well without any change.
\begin{definition}\label{def:1}
Let $S$ be a locally compact group and let $\pi_1$ and $\pi_2$ be unitary representations of $S$. We say that $\pi_1$
is \emph{weakly contained} in $\pi_2$ if and only if
%each diagonal matrix coefficient $s\rightarrow \langle\pi_1(s)v,v\rangle$ of $\pi_1$ is
%the limit, uniformly on compacta, of convex combinations of diagonal matrix coefficients of $\pi_2$ \cite{Fe}. Or equivalently,
each matrix coefficient $s\rightarrow \langle\pi_1(s)v,w\rangle$ of $\pi_1$ is the limit, uniformly on compacta, of sums of matrix coefficients of $\pi_2$
\begin{align*}
s\rightarrow \sum_{i=1}^n\langle\pi_2(s)v_i,w_i\rangle,
\end{align*}
subject to the condition that $\sum_{i=1}^n\norm{v_i}\,\norm{w_i}\leq \norm{v}\,\norm{w}$ (see \cite{cow1} and \cite{Ey}) and an equivalent definition given by diagonal matrix coefficient is in \cite{Fe}. Loosely speaking, an irreducible representation is weakly contained in the regular representation if it appears in the Plancherel formula.
\end{definition}
In particular, to an arbitrary representation $\pi$ of $S$, we may assign a set supp$\,\pi$ in the unitary dual $\widehat{S}$ of $S$ consisting of all $\psi\in \widehat{S}$ which are weakly contained in $\pi$. If $S$ is of type I, as we will assume, then supp$\,\pi$ is exactly the support of the projection-valued measure on $\widehat{S}$ defining $\pi$ (up to unitary equivalence). The following establishes equivalent definitions of a $\bigl(K,\,r\Psi\bigl)$ bounded representation \cite[Lemma 6.2]{howe}.
\begin{proposition}\label{th:2}
Let $S$ be a locally compact group with compact subgroup $K$ and $H$ be a subgroup of $S$. Then a representation of $S$ is $\bigl(K,\,r\Psi\bigl)$ bounded on $H$ if and only if all $\psi\in \text{supp}\,\pi$ are $\bigl(K,\,r\Psi\bigl)$ bounded on $H$.
\end{proposition}
We shall also make use of the following lemma which is an obvious consequence of \cite{cow1}, Theorem \ref{th:10} and above proposition (since a connected semisimple almost $k$-algebraic group is known to be of Type I (see \cite{Be} and \cite{warner})).
\begin{lemma}\label{le:7}Let $p$ be a positive integer. Suppose $\pi$ of $G$ is a direct integral of unitary representations. Then it is strongly
$L^{2p+\epsilon}$ if and only if almost all the integrands are.
%If the unitary representation $\pi$ of $G$ is strongly $L^{2p+\epsilon}$ then every matrix coefficient of $\pi$ is in $L^{2k+\epsilon}$.
  \end{lemma}

\section{The Fourier transform and projection-valued measure}
\subsection{The Fourier transform}
Let $\mathcal{N}$ be a locally compact abelian group with a Haar measure $d\mathfrak{n}$ and denote by $\widehat{\mathcal{N}}$ its dual group.  The Fourier transform of $L^1(\mathcal{N})$ is obtained by restriction:
\begin{align*}
    \widehat{f}(\chi)=\int_{\mathcal{N}}f(\mathfrak{n})\overline{\chi(\mathfrak{n})}d\mathfrak{n},\qquad f\in L^1(\mathcal{N}),
\end{align*}
the bar denoting complex conjugation. In particular, $\widehat{f}$ belongs to $C_0(\widehat{\mathcal{N}})$ for all $f\in L^1(\mathcal{N})$, where $C_0(\mathcal{N})$ is the space of complex-valued continuous functions vanishing at infinity \cite[pg. 93]{Folland1}. The space of functions $\mathcal{S}(\mathcal{N})$,
known as the \emph{Schwartz-Bruhat space} of $\mathcal{N}$ (rapidly decreasing functions on $\mathcal{N}$), is defined such that it has the property: the Fourier transform induces a topological isomorphism
$\mathcal{S}(\mathcal{N})\cong \mathcal{S}(\widehat{\mathcal{N}})$
(see \cite{Bruhat} and \cite{Osborne}). The definition by Fran\c{c}ois Bruhat is a direct limit of spaces of
infinitely differentiable, rapidly decreasing functions on quotient spaces of Lie type
of $\mathcal{N}$ and it generalizes the familiar Schwartz space of $\RR^n$ (see \cite{Folland}).
In particular, if  $\mathcal{N}$ is  isomorphic to a finite dimensional vector space over a local field $k$, in the archimedean case, $\mathcal{S}(\mathcal{N})$ is the space of Schwartz functions; for non-archimedean case $\mathcal{S}(\mathcal{N})$
is the space of locally constant functions with compact
support.
\begin{theorem} \label{th:1}For a suitable normalization of the dual Haar measure $d\widehat{\mathfrak{n}}$ on $\widehat{\mathcal{N}}$, we have:
\begin{enumerate}
  \item  The Fourier transform $f\rightarrow \widehat{f}$ from $L^1(\mathcal{N})\bigcap L^2(\mathcal{N})$ to $L^2(\widehat{\mathcal{N}})$ extends to an isometry from $L^2(\mathcal{N})$ onto $L^2(\widehat{\mathcal{N}})$.
  \item  If $f\in L^1(\mathcal{N})$ and $\widehat{f}\in L^1(\widehat{\mathcal{N}})$, then
for almost every $\mathfrak{n}\in \mathcal{N}$,
\begin{align*}
    f(\mathfrak{n})=\int_{\widehat{\mathcal{N}}}\chi(\mathfrak{n})\widehat{f}(\chi)d\chi.
\end{align*}
\item  Every $\mathfrak{n}\in\mathcal{N}$ defines a unitary character $\eta(\mathfrak{n})$ on $\widehat{\mathcal{N}}$ by the formula
\begin{align*}
\eta(\mathfrak{n})(\chi)=\chi(\mathfrak{n}),\qquad \forall \chi\in \widehat{\mathcal{N}}.
\end{align*}
 The canonical group homomorphism
$\eta:\mathcal{N}\rightarrow \widehat{\widehat{\mathcal{N}}} $ is an isomorphism of topological groups.
\end{enumerate}
\end{theorem}
$(1)$, $(2)$ and $(3)$ in above theorem are called Plancherel's Theorem, Fourier Inversion Theorem and Pontrjagin's Duality Theorem respectively.
So, we can and will always identify $\widehat{\widehat{\mathcal{N}}}$ with $\mathcal{N}$ and will take the normalized dual Haar measure $d\widehat{\mathfrak{n}}$ on $\widehat{\mathcal{N}}$ (relative to $d\mathfrak{n}$ on $\mathcal{N}$).

\subsection{Group algebra of locally compact groups} Let
 $S$ be a locally compact group, with a left invariant Haar measure $ds$. The \emph{convolution } $f_1\ast f_2$ of two functions $f_1,\,f_2\in L^1(S)$ is defined by
 \begin{align*}
f_1\ast f_2(h)=\int_Sf_1(s)f_2(s^{-1}h)ds.
 \end{align*}
 The group convolution algebra $L^1(S)$, equipped with the involution $f\rightarrow f^\ast$, where
 \begin{align*}
    f^\ast(s)=\delta_S(s^{-1})\overline{f}^\vee(s),\qquad \forall s\in S,
 \end{align*}
$\delta_S$ denoting the modular function of group $S$ and $^\vee$ denoting reflection ($f^\vee(s)=f(s^{-1})$ for all $s\in S$), is a Banach $^\ast$-algebra.

 Let $\pi$ be a unitary representations  of $S$ on a Hilbert space $\mathcal{H}$ with inner product $\langle\,,\,\rangle$. For $v\in \mathcal{H}$, a diagonal matrix coefficient $s\rightarrow\langle \pi(s)v,v\rangle$ on $S$ is a positive definite function, in the sense of Bochner. Vice versa, any positive definite function can be realized as a diagonal matrix coefficient by the Gelfand-Naimark-Segal construction (see \cite{Valette} and \cite{Dixmier}). The representation of $\pi$ extends to a $^\ast$-representation of $L^1(S)$: for any $f_1,\,f_2\in L^1(S)$
 \begin{align}\label{for:48}
 \pi(f_1\ast f_2)=\pi(f_1)\pi(f_2)\quad\text{ and }\quad\pi(f^\ast)=\pi(f)^\ast
 \end{align}
where $\pi(f)^\ast$ denotes the adjoint operator of $\pi(f^\ast)$ and $\pi(f)$ is the operator on $\mathcal{H}$ for which
\begin{align*}
\langle \pi(f)v,w\rangle=\int_{S}f(s)\langle \pi(s)v,w\rangle ds, \qquad \forall v,\,w\in \mathcal{H}
\end{align*}
for any $f\in L^1(S)$.

In particular, for the left regular representation $\Lambda$, $\Lambda(f)$ is the operator of left convolution by $f$ on $L^2(S)$: $\Lambda(f)g=f\ast g$ for any $g\in L^2(S)$.

\subsection{Projection-valued measure}\label{sec:1} Let $S$ be a locally compact group and $\mathcal{N}$ be an abelian closed
normal subgroup of $S$.
%We also assume that $\mathcal{N}$ is isomorphic to $k^n$ for some $n\in\NN$.
Let $\pi$ be a unitary representation of $S$ on a Hilbert space $\mathcal{H}$. For $\xi,\,\eta\in \mathcal{H}$, consider
the corresponding matrix coefficient of $\pi\mid_\mathcal{N}$:
\begin{align*}
\phi_{\xi,\eta}(\mathfrak{n})=\langle \pi(\mathfrak{n})\xi,\,\eta\rangle,\quad\text{ for any }\mathfrak{n}\in \mathcal{N}.
\end{align*}
By Bochner's Theorem, there exists a finite
complex regular Borel measure $\mu_{\xi,\eta}$ on $\widehat{\mathcal{N}}$ such that
\begin{align}\label{for:77}
\phi_{\xi,\eta}(\mathfrak{n})=\int_{\widehat{\mathcal{N}}}\chi(\mathfrak{n})d\mu_{\xi,\eta}(\chi).
\end{align}
Clearly, $d\mu_{\xi,\eta}(\widehat{\mathcal{N}})=\langle \xi,\eta\rangle$.  The representation $\pi\mid_\mathcal{N}$ extends to a $^\ast$-representation on $\mathcal{S}(\widehat{\mathcal{N}})$: for any $f\in \mathcal{S}(\widehat{\mathcal{N}})$, $\widehat{\pi}(f)$ is the operator on $\mathcal{H}$ for which
\begin{align*}
\bigl\langle \widehat{\pi}(f)\xi,\eta\bigl\rangle=\int_{\mathcal{N}}\bigl\langle\widehat{f}(\mathfrak{n})\pi(\mathfrak{n})\xi,\eta\bigl\rangle d\mathfrak{n}, \qquad \forall \xi,\,\eta\in \mathcal{H}.
\end{align*}
Then we have
\begin{align}\label{for:79}
\bigl\langle \widehat{\pi}(f)\xi,\eta\bigl\rangle&=\int_{\mathcal{N}}\bigl\langle\widehat{f}(\mathfrak{n})\pi(\mathfrak{n})\xi,\eta\bigl\rangle d\mathfrak{n}
\overset{\text{(1)}}{=}\int_{\mathcal{N}}\int_{\widehat{\mathcal{N}}}\widehat{f}(\mathfrak{n})\chi(\mathfrak{n})d\mu_{\xi,\eta}(\chi) d\mathfrak{n}\notag\\
&=\int_{\widehat{\mathcal{N}}}\int_{\mathcal{N}}\widehat{f}(\mathfrak{n})\chi(\mathfrak{n}) d\mathfrak{n}d\mu_{\xi,\eta}(\chi)
\overset{\text{(2)}}{=}\int_{\widehat{\mathcal{N}}}f(\chi)d\mu_{\xi,\eta}(\chi).
\end{align}
$(1)$ follows from \eqref{for:77} and $(2)$ holds by using Fourier Inversion Theorem and Plancherel's Theorem.

Then it follows that
\begin{align}\label{for:43}
\norm{\widehat{\pi}(f)}\leq \norm{f}_\infty,\qquad \forall f\in \mathcal{S}(\widehat{\mathcal{N}}).
\end{align}
Since the Fourier transform converts multiplication to convolution, that is:
\begin{align*}
    \widehat{f_1 \cdot f_2}&=\widehat{f_1}\ast\widehat{f_2},\qquad \forall f_1,\,f_2\in \mathcal{S}(\widehat{\mathcal{N}}),
    \end{align*}
it follows from \eqref{for:48} that
\begin{align}\label{for:24}
\widehat{\pi}(f_1\cdot f_2)=\pi(\widehat{f_1\cdot f_2})=\pi(\widehat{f_1}\ast\widehat{f_2})=\widehat{\pi}(f_1)\widehat{\pi}(f_2), \quad \forall f_1,\,f_2\in \mathcal{S}(\widehat{\mathcal{N}});
\end{align}
and the relation $(\overline{\widehat{f}})^\vee=\widehat{\overline{f}}$ yields
\begin{align}\label{for:6}
\widehat{\pi}(f)^\ast=\widehat{\pi}(\bar{f}), \qquad \forall f\in \mathcal{S}(\widehat{\mathcal{N}}).
\end{align}
This inequality \eqref{for:43} allows us to extend $\widehat{\pi}$ from $\mathcal{S}(\widehat{\mathcal{N}})$ to $L^\infty(\widehat{\mathcal{N}})$ by taking strong limits of operators and pointwise monotone increasing limits of non-negative functions (see \cite{Lang} for a detailed treatment). Hence $\widehat{\pi}$ is a homomorphism of $L^\infty(\widehat{\mathcal{N}})$ to bounded operators on $\mathcal{H}$.  Also, \eqref{for:24} and \eqref{for:6} extend to $f_1,\,f_2,\,f\in L^\infty(\widehat{\mathcal{N}})$.

Let $\mathcal{B}(\widehat{\mathcal{N}})$ be the $\sigma$-algebra of the Borel subsets of $\widehat{\mathcal{N}}$. For every $X\in \mathcal{B}(\widehat{\mathcal{N}})$, let $ch_X$ denote characteristic function of $X$. \eqref{for:24} and \eqref{for:6} imply that $\widehat{\pi}(ch_X)$ is idempotent and self-adjoint, which implies that it is an  orthogonal projection of $\mathcal{H}$. Write $\widehat{\pi}(ch_X)=P_X$. From \eqref{for:79}, we see that
\begin{align}\label{for:89}
\langle P_X\xi,\eta\rangle=\mu_{\xi,\eta}(X).
\end{align}
It is readily verified that $X\rightarrow P_X$ is a regular projection-valued measure on $\widehat{\mathcal{N}}$ associated to the unitary representation
$\pi\mid_{\mathcal{N}}$ of the abelian group $\mathcal{N}$. Moreover (see \cite[Theorem D.3.1]{Valette} or \cite[Chapter 5.4]{warner}),
%have
%\begin{align*}
%\langle \pi(\widehat{ch}_X)\xi,\eta\rangle&=\int_{\mathcal{N}}\langle\widehat{ch}_X(\mathfrak{n})\pi(\mathfrak{n})\xi,\eta\rangle d\mathfrak{n}
%=\int_{\mathcal{N}}\int_{\widehat{\mathcal{N}}}\widehat{ch}_X(\mathfrak{n})\overline{\chi(\mathfrak{n})}\mu_{\xi,\eta}(\chi) d\mathfrak{n}\\
%&=\int_{\widehat{\mathcal{N}}}\int_{\mathcal{N}}\widehat{ch}_X(\mathfrak{n})\overline{\chi(\mathfrak{n})} d\mathfrak{n}d\mu_{\xi,\eta}(\chi)
%=\int_{\widehat{\mathcal{N}}}ch_X(\chi)d\mu_{\xi,\eta}(\chi)\\
%&=\mu_{\xi,\eta}(X).
%\end{align*}
\begin{align*}
\pi(x)=\int_{\widehat{\mathcal{N}}}\chi(x)dP(\chi),\qquad \forall x\in \mathcal{N}.
\end{align*}
Since $\mathcal{N}$ is normal in $S$, we have a representation $\rho: S\rightarrow GL(\mathcal{N})$ defined by $\rho(s)\mathfrak{n}=s\cdot\mathfrak{n}= s\mathfrak{n}s^{-1}$. This means we have the following relation
\begin{align}\label{for:51}
\pi(s)\pi(\mathfrak{n})\pi(s)^{-1}=\pi\bigl(s\cdot\mathfrak{n}\bigl),\qquad \forall s\in S,\,\forall\mathfrak{n}\in \mathcal{N}.
\end{align}
Since $S$ acts continuously by automorphisms on $\mathcal{N}$, then it acts continuously on $\widehat{\mathcal{N}}$ by
\begin{align*}
(s\cdot\chi)(\mathfrak{n})=\chi\bigl(s^{-1}\cdot\mathfrak{n}\bigl),\qquad s\in S,\,\chi\in\widehat{\mathcal{N}},\,\,\mathfrak{n}\in \mathcal{N};
\end{align*}
and let $\star $ denote the associated action of $S$ on functions on $\widehat{\mathcal{N}}$. Thus
\begin{align*}
(s\star f)(\chi)=f(s^{-1}\cdot\chi), \qquad \forall \chi\in \widehat{\mathcal{N}}.
\end{align*}
Straightforward calculations show that the conjugacy relation \eqref{for:51} implies a similar relation for the operator $\widehat{\pi}$ for any $s\in S$ and any $f\in \mathcal{S}(\widehat{\mathcal{N}})$:
\begin{align*}
\bigl\langle \pi(s)\widehat{\pi}(f)\pi(s^{-1})\xi,\eta\bigl\rangle&=\int_{\mathcal{N}}\bigl\langle\widehat{f}(\mathfrak{n})\pi(s)\pi(\mathfrak{n})\pi(s^{-1})\xi,\eta\bigl\rangle d\mathfrak{n}\\
&=\int_{\mathcal{N}}\bigl\langle\widehat{f}(\mathfrak{n})\pi(s\mathfrak{n}s^{-1})\xi,\eta\bigl\rangle d\mathfrak{n}\notag\\
&\overset{\text{(1)}}{=}\int_{\mathcal{N}}\int_{\widehat{\mathcal{N}}}\widehat{f}(\mathfrak{n})\chi(s\cdot\mathfrak{n})d\mu_{\xi,\eta}(\chi) d\mathfrak{n}\\
&=\int_{\widehat{\mathcal{N}}}\int_{\mathcal{N}}\widehat{f}(\mathfrak{n})(s^{-1}\cdot\chi)(\mathfrak{n}) d\mathfrak{n}d\mu_{\xi,\eta}(\chi)\notag\\
&\overset{\text{(2)}}{=}\int_{\widehat{\mathcal{N}}}f(s^{-1}\cdot\chi)d\mu_{\xi,\eta}(\chi)\\
&\overset{\text{(3)}}{=}\bigl\langle \widehat{\pi}(s \star f)\xi,\eta\bigl\rangle.
\end{align*}
$(1)$ and $(3)$ follow from \eqref{for:77} and \eqref{for:79} respectively and $(2)$ holds by using Fourier Inversion Theorem and Plancherel's Theorem. Then it follows that
\begin{align}\label{for:90}
\pi(s)\widehat{\pi}(f)\pi(s^{-1})=\widehat{\pi}(s\star f),\qquad s\in S,\,\,f\in \mathcal{S}(\widehat{\mathcal{N}}).
\end{align}
The relation \eqref{for:90} persists in the strong limit to hold for $f\in L^\infty(\widehat{\mathcal{N}})$. In particular, the relation
\begin{align*}
\pi(s)P_X\pi(s)^{-1}=P_{s(X)},\qquad\text{ for }X\in\mathcal{B}(\widehat{\mathcal{N}}),\quad s\in S
\end{align*}
 holds for the projection-valued measure associated with $\pi$.

 The following establishes approximation relations between projected matrix coefficients, which is useful in this paper:
\begin{lemma}\label{le:1}
let $\pi_0$ be another unitary representation of $S$. Suppose $\pi$ is weekly contained in $\pi_0$. Then for any
$f\in\mathcal{S}(\widehat{\mathcal{N}})$ and for any vectors $\xi,\,\eta$ of $\pi$,
 each projected matrix coefficient
 \begin{align*}
 \phi(f):(s,t)\rightarrow\bigl\langle \widehat{\pi}(f)\bigl(\pi(s)\xi\bigl),\pi(t)\eta\bigl\rangle
 \end{align*}
 is the limit, uniformly on compacta, of sums of the projected matrix coefficients
 \begin{align*}
\psi(f):(s,t)\rightarrow\sum_{i=1}^n\bigl\langle \widehat{\pi_0}(f)\bigl(\pi_0(s)\xi_i\bigl),\pi_0(t)\eta_i\bigl\rangle
 \end{align*}
 for vectors $\xi_i,\,\eta_i$ of $\pi_0$, subject to $\sum_{i=1}^n\norm{\xi_i}\,\norm{\eta_i}\leq \norm{\xi}\,\norm{\eta}$.
\end{lemma}
\begin{proof}
%Since $d\mu^{\pi}_{\xi,\xi}$ is regular, there exist closed sets $\mathcal{F}_n$ and open sets $\mathcal{U}_n$ such that
%\begin{align*}
%\mathcal{F}_n \subset X\subset \mathcal{U}_n \quad\text{ and } \quad d\mu^{\pi}_{\xi,\xi}(\mathcal{U}_n\backslash \mathcal{F}_n)<\frac{1}{n}.
%\end{align*}
%There exist $0\leq f_n\leq 1$ inside $\mathcal{S}(\widehat{\mathcal{N}})$ such that $f_n=1$ on $\bigcup_{i\leq n}\mathcal{F}_i $ and $f_n=0$ on $\widehat{\mathcal{N}}\backslash(\bigcup_{i\leq n}\mathcal{U}_i)$. Using \eqref{for:76} and \eqref{for:79} for any $\eta\in\mathcal{H}_{\pi}$ we see that
%\begin{align}\label{for:82}
%&\big\lvert\bigl\langle \widehat{\pi}(ch_X-f_n)(\pi(s)\xi),\xi\bigl\rangle\big\rvert\notag\\
%&=\Big\lvert\int_{\widehat{\mathcal{N}}}\bigl(ch_X(\chi)-f_n(\chi)\bigl)d\mu^{\pi}_{\pi(s)\xi,\xi}(\chi)\Big\rvert\notag\\
%&\leq \frac{4}{n}\norm{\eta}^2.
%\end{align}
%It is clear that similar conclusion also holds for $\widehat{\pi_0}$.
For any $\epsilon>0$, there exists $M_1>0$ such that
\begin{align*}
\int_{\abs{\mathfrak{n}}\geq M_1}\abs{\widehat{f}(\mathfrak{n})}d\mathfrak{n}<\frac{\epsilon}{\norm{\xi}\,\norm{\eta}}.
\end{align*}
We assume neither of $\xi,\,\eta$ is $0$, otherwise the conclusion is obvious. Since $\pi$ is weekly contained in $\pi_0$, each matrix coefficient $(s,t)\rightarrow \bigl\langle\pi(s)\xi,\,\pi(t)\eta\bigl\rangle$ is the limit, uniformly on compacta, of sums of matrix coefficients
\begin{align*}
(s,t)\rightarrow \sum_{i=1}^n\bigl\langle\pi_0(s)\xi_i,\,\pi_0(t)\eta_i\bigl\rangle,
\end{align*}
subject to the condition that $\sum_{i=1}^n\norm{\xi_i}\,\norm{\eta_i}\leq \norm{\xi}\,\norm{\eta}$. Let $Y_1,\,Y_2$ be  compact sets in $S$. Then there exists $m(Y_1,Y_2,\epsilon)\in\NN$ such that for any $n\geq m$, any $\mathfrak{n}\in \mathcal{N}$ with $\abs{\mathfrak{n}}\leq M_1$ and any $y_1\in Y_1$, $y_2\in Y_2$ we have
\begin{align}\label{for:83}
E(n,\mathfrak{n},y_1,y_2)&=\Big\lvert\bigl\langle\pi(\mathfrak{n}\cdot y_1)\xi,\,\pi(y_2)\eta\bigl\rangle-\sum_{i=1}^n\bigl\langle \pi_0(\mathfrak{n}\cdot y)\xi_i,\,\pi_0(y_2)\eta_i\bigl\rangle\Big\rvert\notag\\
&\leq  \frac{\epsilon}{\int_{\mathcal{N}}\abs{\widehat{f}(\mathfrak{n})}d\mathfrak{n}}.
\end{align}
We also assume $\int_{\mathcal{N}}\abs{\widehat{f}(\mathfrak{n})}d\mathfrak{n}\neq 0$, otherwise the conclusion is obvious. Then it follows that for any $y_1\in Y_1$ and $y_2\in Y_2$
\begin{align*}
&\Big\lvert\bigl\langle \widehat{\pi}(f)(\pi(y_1)\xi),\,\pi(y_2)\eta\bigl\rangle-\sum_{i=1}^n\bigl\langle \widehat{\pi_0}(f)(\pi_0(y_1)\xi_i),\,\pi_0(y_2)\eta_i\bigl\rangle\Big\rvert\notag\\
&\leq
\int_{\mathcal{N}}\big\lvert\widehat{f}(\mathfrak{n})\big\rvert E(n,\mathfrak{n},y_1,y_2) d\mathfrak{n}\notag\\
&\leq \int_{\abs{\mathfrak{n}}\leq M_1}\big\lvert\widehat{f}(\mathfrak{n})\big\rvert E(n,\mathfrak{n},y_1,y_2) d\mathfrak{n}+\int_{\abs{\mathfrak{n}}\geq M_1}\big\lvert\widehat{f}(\mathfrak{n})\big\rvert \cdot\big\lvert\langle\pi(\mathfrak{n}\cdot y_1)\xi,\,\pi(y_2)\eta\rangle \big\rvert d\mathfrak{n}\\
&+\sum_{i=1}^n\int_{\abs{\mathfrak{n}}\geq M_1}\big\lvert\widehat{f}(\mathfrak{n})\big\rvert \cdot\big\lvert\langle\pi_0(\mathfrak{n}\cdot y_1)\xi_i,\,\pi_0(y_2)\eta_i\rangle \big\rvert d\mathfrak{n}\\
&\overset{\text{(1)}}{\leq}\frac{\epsilon}{\int_{\mathcal{N}}\abs{\widehat{f}(\mathfrak{n})}d\mathfrak{n}}\cdot\int_{\abs{\mathfrak{n}}\leq M_1}\abs{\widehat{f}(\mathfrak{n})}d\mathfrak{n}
+ \norm{\xi}\,\norm{\eta}\cdot\int_{\abs{\mathfrak{n}}\geq M_1}\abs{\widehat{f}(\mathfrak{n})}d\mathfrak{n}\\
&+ \sum_{i=1}^n\norm{\xi_i}\,\norm{\eta_i}\cdot\int_{\abs{\mathfrak{n}}\geq M_1}\abs{\widehat{f}(\mathfrak{n})}d\mathfrak{n}\\
&\leq\epsilon+\norm{\xi}\,\norm{\eta}\cdot\frac{\epsilon}{\norm{\xi}\,\norm{\eta}}+(\sum_{i=1}^n\norm{\xi_i}\,\norm{\eta_i})\cdot\frac{\epsilon}{\norm{\xi}\,\norm{\eta}}\notag\\
&=3\epsilon.
\end{align*}
$(1)$ holds by using Cauchy-Schwarz inequality. Hence the lemma is proved.
\end{proof}
\subsection{Identification between $V$ and $\widehat{V}$}\label{sec:22} In this paper, we consider the case $S=G\ltimes_\rho V$ and $\mathcal{N}=V$.
The dual group $\widehat{V}$ of $V$ can be identified with
$V$ as follows. Fix a unitary character $\zeta$ of the additive group of $k$ distinct from the unit character. The mapping
\begin{align*}
V\rightarrow \widehat{V},\qquad v\rightarrow \zeta_v
\end{align*}
is a topological group isomorphism, where $\zeta_v(x)$ is defined by $\zeta(\varsigma(v, x))$ and $\varsigma(v, x)=\sum v_ix_i$ for $v=\sum v_ie_i$ and
$x=\sum x_ie_i$ in $V$. Here $\{e_1,e_2,\cdots\}$ is a basis of $V$ (see. \cite[Ch II-5, Theorem 3]{weil}). Define the transpose $\tau$ of $\rho(g)$ by
\begin{align*}
\varsigma(v,\rho(g)^\tau w)=\varsigma(\rho(g)v,w),\qquad\forall v,w\in V,
\end{align*}
and denote $(\rho(g)^{-1})^\tau$ by $\rho^\ast(g)$.

Under this
identification, the dual action of $G$ on $\widehat{V}$ corresponds
to the $G$ action $\rho^\ast$ on $V$:
\begin{align}\label{for:36}
(g\cdot\zeta_v)(\mathfrak{n})=\zeta_v\bigl(\rho(g)^{-1}\mathfrak{n}\bigl)=\zeta_{\rho^\ast(g) v}(\mathfrak{n}),\qquad g\in G\text{ and }\,v,\,\mathfrak{n}\in\mathcal{N}.
\end{align}
Also, the space $L^\infty(V)$ can be identified with the space $L^\infty(\widehat{V})$ and the associated action of $G$ on $L^\infty(V)$ is
\begin{align}\label{for:12}
(g\star f)(v)&=f(g^{-1}\cdot \zeta_v)=f(\zeta_{\rho^\ast(g^{-1}) v}),\qquad \forall f\in L^\infty(V),\,\forall v\in\mathcal{N}\notag\\
&=f\bigl(\rho^\ast(g^{-1}) v\bigl).
\end{align}
We will use this identification hereafter.

%\begin{theorem}\label{th:13}
%Let $H$ be a connected reductive algebraic group over a non-archimedean field $k$ and $V$ a finite dimensional vector
%space over $k$. For a rational homomorphism $\rho: H\rightarrow GL(V)$
%suppose $\rho$ is irreducible and $\rho\mid_{H_i}$ is non-trivial for any $k$-isotropic factor $H_i$ of $H$.
%Let $\pi$ be a unitary representation of $H\ltimes_\rho V$ on a Hilbert space $\mathcal{H}$ which contains no fixed vectors for $V$.  Then for any $K$-finite unit
%vectors $v$ and $w$ in $\mathcal{H}$,
%\begin{align*}
%\abs{\langle \pi(g)v,w\rangle}\leq
%([K:K\cap \omega K\omega^{-1}]\dim \langle Kv\rangle \dim \langle Kv\rangle)^{\frac{1}{2}}\Xi_G^{\frac{1}{2p}}(g),
%\end{align*}
%for any $g=k_1d\omega k_2\in K(D^+F)K=H$. Here $p$ is defined in Corollary \ref{cor:4}.

%\end{theorem}

\section{Analysis of $K$-orbit}\label{sec:21}
\subsection{Useful notations}\label{sec:10}
We try as much as possible to develop a unified system of notations. We will use notations from this section throughout   subsequent sections. So the reader
should consult this section  if an unfamiliar symbol appears. In the whole Section \ref{sec:21}  we assume $\rho$ is an irreducible representation of $G$ on $V$.
\begin{sect}\label{sec:18}
Let $\Phi$ be the set of
non-zero roots of $G$ relative to $D$ and fix an ordering on $\Phi$ as described in Section \ref{sec:4} for different cases of $k$. Denote by $\Phi^+$ the set of positive roots and by $\Delta$ the set of simple roots in $\Phi^+$. Let $\Phi_1$ be the set of weights of $\rho$ relative to $D$ and let $\lambda$ and $\varrho$ be the highest and lowest weight respectively with respect to the given ordering on $\Phi$.
\end{sect}
\begin{sect} \label{for:84}Let $K$ be a good maximal compact subgroup in $G$. Fix a
$\rho(K)$-invariant norm $\abs{\,\cdot\,}$ on $V$. We introduce a norm on $End(V)$:
\begin{align*}
 \norm{x}\stackrel{\text{def}}{=}\sup\{\abs{x(\nu)}\cdot\abs{\nu}^{-1}, \,0\neq \nu\in V\},\quad\text{ for any }x\in End(V).
\end{align*}
In what follows, $C$ will denote any constant that depends only on $G$ and the given representation $\rho$. $\mathfrak{o}(x)$ ($x\in X$) will denote any map (defined on a set $X$) taking values in $V$ with a uniform bound which depends only on $G$ and the given representation $\rho$.
 \end{sect}

\begin{sect}\label{for:142} Let $\mathfrak{u}^+$ and $\mathfrak{u}^-$ be the sets defined in \eqref{for:91}, \eqref{for:97} or \eqref{for:96} under different cases of $k$. Write $V=\bigoplus V_\phi$, $\phi\in\Phi_1$ as the weight space decomposition of $V$. We use $\pi_{\phi}$ to denote the projection from $V$ to the weight space $V_\phi$. Denote by $d\rho$ the induced representation
of $\rho$ on $\mathfrak{g}_k$.

The following fact about $d\rho$ is well known: for any $0\neq v\in V_\phi$ and $\phi\neq \lambda$ (resp. $\phi\neq \varrho$),  there exist  $\mathfrak{c}\in \mathfrak{u}^+$ (resp. $\mathfrak{c}\in \mathfrak{u}^-$) such that $d\rho(\mathfrak{c})v\neq 0$. For a field of characteristic $0$, it is a direct consequence of
Birkhoff-Witt theorem and irreducibility of $d\rho$.
\end{sect}
\begin{sect}\label{for:11}
Define a map $r: \mathfrak{u}^+\bigcup\mathfrak{u}^-\rightarrow \Phi$ by taking $r(\mathfrak{c})$ to $\omega$ if $\mathfrak{c}\in\mathfrak{g}_\omega$.
For any finite set $A$ we use $\sharp A$ to denote the cardinality of $A$. For a finite dimensional vector space $M$ over $k$ and its $k$-subspace $M_1$, a $k$-subspace $M_2$ of $M$ is called \emph{a complement subspace} of $M_1$ in $M$ if $M=M_1\oplus_k M_2$.
\end{sect}

\begin{sect}\label{sec:17}
Suppose $\ker(\rho)\bigcap G_s\subset Z(G)$. We denote by $\psi(\ln d)=\ln\abs{\psi(d)}$ for any $d\in D$ and any $\psi\in \Phi_1$. Since $\sum_{\psi\in\Phi_1}\psi(\ln d)=0$
(see \eqref{for:34} of Lemma \ref{le:4}), for any $a\in D^+$ we have
\begin{align}\label{for:32}
 \lambda(\ln a)\geq \frac{\sum_{\psi\in\Phi_1}\abs{\psi(\ln a)}}{2(N_{\Phi_1})}\quad\text{ and }\quad\varrho(\ln a^{-1})\geq \frac{\sum_{\psi\in\Phi_1}\abs{\psi(\ln a)}}{2(N_{\Phi_1})}
\end{align}
where $N_{\Phi_1}$ is the number of non-zero weights of $\rho$ in $\Phi_1$.

\eqref{for:35} of Lemma \ref{le:4} and \eqref{for:32} imply that there exists $m>0$ such that for any $a\in D^+$ one has
\begin{align*}
m\lambda(\ln a)>\sum_{\psi\in \Delta}\abs{\psi(\ln a)}\quad\text{ and }\quad m\lambda(\ln a)>\sum_{\psi\in \Delta}\abs{\psi(\ln a)}.
\end{align*}
Suppose $\lambda=\sum_{\omega\in \Delta} m_\omega\omega$ where $m_\omega\in\QQ$.
Hence we see that all $m_\omega>0$. Similar result holds for $\varrho$, that is, $\varrho=-\sum_{\omega\in \Delta} m'_\omega\omega$ where $m'_\omega>0$ for each $\omega\in\Delta$.
\end{sect}
\begin{sect}
 \label{for:64} Any element in the root lattice has a unique expression $\sum c_\omega\omega$ where $c_\omega\in\ZZ$ and $\omega\in\Delta$. Any weight $\phi$ other than $\lambda$ has the form $\lambda-\sum c_\omega\omega$ ($\omega\in\Delta$ and $c_\omega\in\NN$ \cite{humph}.
 The express is unique and we call the number $\sum c_\omega$ the length of $\phi$ and denote it by $L(\phi)$.
  %and call all $\omega$'s showing up in the express the generators of $\phi$ and denote the set of generators by Gen$(\phi)$.\foot{do we need generator?}\\
  Define a function $l$ on simple roots: $l(\omega)=a_\omega$ where $a_\omega>1$ for each $\omega\in\Delta$. Naturally,  $l$ can be extended to a function $l_1$ on all weights: $l_1(\phi)=c_\omega l(\omega)$ if $\phi=\lambda-\sum c_\omega\omega$.

$l_1$ is well defined by uniqueness of the expression on root lattice.  We can choose $l$ sufficiently close to $1$ and rationally independent on $\Delta$ such that for any $\phi_1\,,\phi_2\in\Phi_1$:
\begin{itemize}
  \item if $\phi_1\neq \phi_2$ then $l_1(\phi_1)\neq l_1(\phi_2)$;
  \item if $L(\phi_1)<L(\phi_2)$ then $l_1(\phi_1)<l_1(\phi_2)$.
\end{itemize}
 $l_1$ also gives an ordering to all weights:
\begin{align*}
    \Phi_1=\{\lambda,\phi_1,\cdots,\phi_{(\sharp{\Phi_1}-1)}=\varrho:\text{ where }l_1(\phi_i)<l_1(\phi_j)\text{ if }i<j\}.
\end{align*}
\end{sect}
\begin{sect} \label{sec:16}
%Suppose $k$ is archimedean.
For any $\phi\in\Phi_1\backslash\lambda$, let $A_\phi=\{\mathfrak{c}\in
\mathfrak{u}^+:\ker(d\rho(\mathfrak{c}))\bigcap V_\phi\not=V_\phi\}$. Result in \ref{for:142} implies that $A_\phi\neq \emptyset$. Define
\begin{align*}
\mathcal{A}_\phi&=\bigl\{(\mathfrak{c}_1,\cdots,\mathfrak{c}_j):\mathfrak{c}_i\in
A_\phi, 1\leq i\leq j, \text{ and }\bigcap_{i=1}^j
\ker(d\rho(\mathfrak{c}_i))\bigcap V_\phi=\{0\}\bigl\}.
\end{align*}
Next, we will show that $\mathcal{A}_\phi\neq \emptyset$.
Choose an element from $A_\phi$ and denote it by $\mathfrak{c}_1$.
Let $V_{\phi}^1$ be a complement subspace of
$\ker(d\rho(\mathfrak{c}_1))\bigcap V_\phi$ in $V_\phi$. If
$V_{\phi}^1=V_\phi$, that is, $\ker(d\rho(\mathfrak{c}_1))\bigcap V_\phi=\{0\}$, then $(\mathfrak{c}_1)\in \mathcal{A}_\phi$. If
$\ker(d\rho(\mathfrak{c}_1))\bigcap V_\phi\neq\{0\}$, there exists $\mathfrak{c}_2\in
A_\phi$ such that $d\rho(\mathfrak{c}_2)$ acts non-trivially on $\ker(d\rho(\mathfrak{c}_1))\bigcap V_\phi$. Again result in \ref{for:142}
guarantees the existence of $\mathfrak{c}_2$. Then $\bigcap_{j=1}^2\ker(d\rho(\mathfrak{c}_i))\bigcap
V_\phi\neq \ker(d\rho(\mathfrak{c}_1))\bigcap V_\phi$.  Let
$V_{\phi}^2$ be a complement subspace of
$\bigcap_{j=1}^2\ker(d\rho(\mathfrak{c}_i))\bigcap V_\phi$ in
$\ker(d\rho(\mathfrak{c}_1))\bigcap V_\phi$. If
$V_{\phi}^2=\ker(d\rho(\mathfrak{c}_1))\bigcap V_\phi$, that is, $\bigcap_{j=1}^2\ker(d\rho(\mathfrak{c}_i))\bigcap V_\phi=\{0\}$, we have
$V_\phi=V_{\phi}^1\bigoplus V_{\phi}^2$ and $(\mathfrak{c}_1,\mathfrak{c}_2)\in \mathcal{A}_\phi$. If
$\bigcap_{j=1}^2\ker(d\rho(\mathfrak{c}_i))\bigcap V_\phi\neq\{0\}$ we get
\begin{align*}
V_\phi=V_{\phi}^1\bigoplus V_{\phi}^2
\bigoplus\bigl(\bigcap_{l=1}^2\ker(d\rho(\mathfrak{c}_l))\bigcap
V_\phi\bigl)
\end{align*}
and then we repeat this process on
$\bigcap_{j=1}^{2}\ker(d\rho(\mathfrak{c}_j))\bigcap V_\phi$. Then finally, we get an integer $j\leq \dim_k V_\phi$, subspaces $V_\phi^l$ of $V_\phi$ and $\mathfrak{c}_l\in A_\phi$, $1\leq l\leq j$ such that:
\begin{enumerate}
\item \label{for:52} $V_{\phi}^l\bigcap\ker(d\rho(\mathfrak{c}_l))=\{0\}$, $1\leq l\leq j$;\\
\item \label{for:2}$V_{\phi}^i\subset\ker(d\rho(\mathfrak{c}_l))$, $i\geq 2$ and $1\leq l\leq
i-1$;\\
  \item $\bigcap_{l=1}^j\ker(d\rho(\mathfrak{c}_l))\bigcap
V_\phi=\{0\}$, which implies that $V_\phi=\bigoplus_{i=1}^{j} V_{\phi}^i$.
\end{enumerate}
Hence the claim $\mathcal{A}_\phi\neq \emptyset$ is proved.

The subspaces $V_\phi^l$, $1\leq l\leq j$ are called the \emph{ordered associated subspaces} of $a=(\mathfrak{c}_1,\cdots,\mathfrak{c}_j)$. Set $\sharp a=j$. For each $\phi\in\Phi_1\backslash\lambda$, fix an element $b_\phi\in \mathcal{A}_\phi$ and set $\sharp \phi=\sharp b_\phi$. Let
\begin{align}\label{for:49}
\ell_0=\sum_{\phi\in\Phi_1\backslash\lambda}\sharp\phi.
\end{align}
Then we have a decomposition of $V$:
\begin{align}\label{for:70}
V=V_\lambda\oplus V_{\phi_1}^1\oplus\cdots V_{\phi_1}^{\sharp\phi_1}\oplus\cdots\oplus V_{\varrho}^1\oplus\cdots V_{\varrho}^{\sharp\varrho},
\end{align}
where $V_{\phi_i}^j$, $1\leq i\leq \sharp\Phi_1-1$ and $1\leq j\leq \sharp(\phi_i)$ are ordered associated subspaces of $b_{\phi_i}$. We call the above decomposition an \emph{effective decomposition} of $V$. Fix an effective decomposition of $V$. We call these corresponding elements $b_\phi=(\mathfrak{c}^1_\phi,\cdots,\mathfrak{c}^{\sharp\phi}_\phi)$ \emph{the associated  elements of $V$}.  Let
\begin{align}\label{for:88}
    C_0=\min_{\phi\in\Phi_1\backslash\lambda}\min_{1\leq i\leq \sharp\phi}\inf\{\abs{d\rho(\mathfrak{c}^i_\phi)v}\abs{v}^{-1},\,\,0\neq v\in V_\phi^i\}.
\end{align}
Property \eqref{for:52} above yields that $C_0>0$.

For any $v\in V$, denote by $v_\phi$ the $v$'s component in $V_\phi$ and if $\phi\neq\lambda$ denote by $v_\phi^i$ the $v$'s component in the $i$-th subspace of $V_\phi$. With respect to the effective decomposition of $V$, we can define a new norm $\abs{\cdot}'$ on $V$:
\begin{align*}
\abs{v}'=\max_{
            \phi\in\Phi_1\backslash\lambda}\max_{1\leq i\leq \sharp \phi}\{\abs{v_\phi^i},\,\abs{v_\lambda}\}.
\end{align*}
Since all norms on $V$ are equivalent, we have
\begin{align}\label{for:80}
C^{-1}\abs{\cdot}\leq \abs{\cdot}'\leq C\abs{\cdot}.
\end{align}
\end{sect}
\begin{sect} \label{sec:14}Let
\begin{align*}
    C'_0=\max_{\phi\in\Phi_1\backslash\lambda}\max_{1\leq i\leq \sharp\phi}\norm{d\rho(\mathfrak{c}^i_\phi)}.
\end{align*}
Let $M_1=\max_\phi\sharp b_\phi$ where $b_\phi$ are the associated  elements of $V$. Let $[\cdot]$ denote the floor function. Define $M_1$ integer numbers inductively: $n_1=1$ and $n_i=\biggl[\frac{\sum_{j=1}^{i-1}2n_{j} C'_0+2}{C_0}\biggl]+1$ if $2\leq i\leq M_1$,  where $C_0$ is defined in \eqref{for:88}. Let
$C_1=\max_{1\leq i\leq M_1}n_i$.
\end{sect}
\begin{sect}\label{for:105} Let $\mathcal{W}=\{v\in V:\pi_{\lambda}(v)=0\text{ and }\frac{3}{4}\leq\abs{v}\leq\frac{5}{4}\}$ and let
$\mathcal{W}_{c_0}=\{v+\mathcal{W}:v\in V_\lambda \text{ with }\abs{v}\leq c_0\}$.
%When $k$ is archimedean,
For any small enough $c_0>0$  we have a cover of $\mathcal{W}$: $\mathcal{W}=\bigcup_{j,i}\mathcal{W}^{j,i,c_0}$, $1\leq j\leq \sharp\Phi_1-1$ and $1\leq i\leq \sharp\phi_j$ where $\mathcal{W}^{j,i,c_0}$ are defined as follows:
\begin{align*}
\mathcal{W}^{j,i,c_0}=&\{v\in V:\pi_{\lambda}(v)=0,\,\frac{1}{2}<\abs{v}<\frac{3}{2}\text{ and }\abs{v^l_{\phi_m}}< 2\theta^l_{\phi_m}\text{ for all }m<j,\\
 &1\leq l\leq \sharp\phi_m; \,\abs{v^l_{\phi_j}}<2\theta^l_{\phi_j}\text{ if }l\leq i-1\text{ and }\abs{v^i_{\phi_j}}>\theta^i_{\phi_j}\}
\end{align*}
where $\theta^l_{\phi_m}=c_0^{1/3^m}n_l$. For simplicity, denote by $\theta_{\phi_n}^0=\theta_{\phi_{n-1}}^{\sharp\phi_{n-1}}$ and $\theta_{\phi_1}^0=0$.

 We call the above cover \emph{the effective cover of $\mathcal{W}$ determined by $c_0$}. By \eqref{for:80}, for each $v\in \mathcal{W}$, there exists $l,\,m\in\NN$ such that $\abs{v^l_{\phi_m}}\geq\frac{3}{4C}$, and then it is easy to check $v\in \mathcal{W}^{j,i,c_0}$ for some $(j,\,i)$ if $c_0$ is small enough. Hence $\bigcup_{j,\,i}\mathcal{W}^{j,i,c_0}$ covers $\mathcal{W}_{c_0}$. Set
\begin{align*}
\mathcal{W}_{c_0}^{j,i}=\{v+\mathcal{W}^{j,i,c_0}:v\in V_\lambda \text{ with }\abs{v}<2c_0\}.
\end{align*}
Then we have an open cover of $\mathcal{W}_{c_0}$: $\mathcal{W}_{c_0}\subset\bigcup_{j,i} \mathcal{W}_{c_0}^{j,i}$, which is called \emph{the effective open cover of $\mathcal{W}_{c_0}$ determined by $c_0$}.

%When $k$ is non-archimedean, for small enough $c_0>0$, we have a decomposition of $\mathcal{W}$: $\mathcal{W}=\bigcup_{i}\mathcal{W}^{i}$ where $\mathcal{W}^{i}$ are %defined as follows:
%\begin{align*}
%\mathcal{W}^{i}=&\{v\in \mathcal{W}:\abs{v_j}\leq \theta_j\text{ for all }j<i, \quad\abs{v_i}>\theta_i\}.
%\end{align*}
%where $\theta_j=c_0^{q_j}$.

%We call the above decomposition \emph{the effective decomposition of $\mathcal{W}$ determined by $c_0$}. Similar to earlier augments, it is easy to check
%that $\mathcal{W}^{i}$, $1\leq i\leq \Phi_1(n)$ is a decomposition of $\mathcal{W}$. Also, the corresponding decomposition of $\mathcal{W}_{c_0}$: %$\mathcal{W}_{c_0}=\bigcup_{i} \mathcal{W}_{c_0}^{i}$ where $\mathcal{W}_{c_0}^{i}=\{v+\mathcal{W}^{i}:v\in V_\lambda \text{ with }\abs{v}\leq c_0\}$ is called \emph{the %effective decomposition of $\mathcal{W}_{c_0}$ determined by $c_0$}.\\
\end{sect}
Denote by
\begin{align}\label{for:124}
\text{Cone}_1(c_0,s)&=\{v\in V:\abs{\pi_{\lambda}(v)}\leq c_0\text{ and }\abs{v}\geq s\}\text{ and }\notag\\
\text{Cone}_2(c_0,s)&=\{v\in V:\abs{\pi_{\varrho}(v)}\leq c_0\text{ and }\abs{v}\geq s\}.
\end{align}
An important step in proving Theorem \ref{th:11} is:
\begin{proposition}\label{po:7}
There exists $\ell_1\in\NN$ such that for any $s>0$, if $c_0$ is small enough, there
  is an open cover of $\text{Cone}_1(c_0,s)$:
 \begin{align*}
\text{Cone}_1(c_0,s)&\subset\bigcup_{1\leq i\leq \ell_1}\mathcal{E}_i
\end{align*}
such that for each $\mathcal{E}_i$, there are at least $[C^{-1}(c_0s^{-1})^{\mathfrak{q}}]$ different elements $\tau^i_j$ in $K$ such that
\begin{align*}
\rho(\tau^i_j)\mathcal{E}_i&\bigcap\rho(\tau^i_\ell)\mathcal{E}_i=\emptyset\qquad\text{ if }\ell\neq j
\end{align*}
where $\mathfrak{q}=-(\frac{1}{3})^{(\sharp\Phi_1-1)}$ and $[\cdot]$ denotes the floor function. Moreover, If $\dim V_\lambda=1$, $\mathfrak{q}=-(\frac{1}{3})^{(\sharp\Phi_1-2)}$.
\end{proposition}
Denote by $B'$ the opposite
parabolic subgroup to $B$ with the common Levi subgroup $C_G(D)$. The ordering on $\Phi$ determined by $B'$ interchanges the sets of positive roots and negative roots determined by $B$, which means $\varrho$ is the highest weight of $\Phi_1$ relative to the new ordering. Hence the above proposition applies to $\text{Cone}_2(a_0,s)$. As an immediate consequence, we deduce the following corollary which is also essential in proving Theorem \ref{th:11}:
\begin{corollary}\label{cor:9}
There exists $\ell_2\in\NN$ such that for any $s>0$, if $c_0$ is small enough, there
  is an open cover of $\text{Cone}_2(c_0,s)$:
 \begin{align*}
\text{Cone}_2(c_0,s)&\subset\bigcup_{1\leq i\leq \ell_2}\mathcal{U}_i
\end{align*}
such that for each $\mathcal{U}_i$, there are at least $[C^{-1}(c_0s^{-1})^{\mathfrak{q}}]$ different elements $\upsilon^i_j$ in $K$ such that
\begin{align*}
\rho(\upsilon^i_j)\mathcal{U}_i&\bigcap\rho(\upsilon^i_\ell)\mathcal{U}_i=\emptyset\qquad\text{ if }\ell\neq j
\end{align*}
where $\mathfrak{q}=-(\frac{1}{3})^{(\sharp\Phi_1-1)}$.  Moreover, If $\dim V_\lambda=1$, $\mathfrak{q}=-(\frac{1}{3})^{(\sharp\Phi_1-2)}$.
\end{corollary}
Before proceeding further with the proof of Proposition \ref{po:7}, we prove certain facts about $\rho(K)$-orbits of $\mathcal{W}_{c_0}$, and these facts ultimately lead to the proof of Proposition \ref{po:7}.
\subsection{Analysis of $K$-orbits of $\mathcal{W}_{c_0}$} \label{sec:15}
Fix a weight $\phi_\ell\in \Phi_1\backslash \lambda$. Suppose the associated element of $V$ for the  subspace $V_{\phi_\ell}$ is $a=(\mathfrak{c}_1,\cdots,\mathfrak{c}_{\sharp\phi_\ell})\in \mathcal{A}_{\phi_\ell}$ (\ref{sec:16} of Section \ref{sec:10}). Let $\alpha_i=r(\mathfrak{c}_i)$ (\ref{for:11} of Section \ref{sec:10}) and let
\begin{enumerate}
  \item $k_i=\mathfrak{c}_i+\tau_0\mathfrak{c}_i$, $1\leq i\leq \sharp\phi_\ell$  and $t\in\RR$ with $\abs{t}\leq 1$ when $k$ is archimedean (see \eqref{for:1} of Remark \ref{re:5});
  \item $k_i=\mathfrak{c}_i$, $1\leq i\leq \sharp\phi_\ell$ and $t\in k$ with $\abs{t}\leq \abs{q}^{\sigma}$  when $k$ is non-archimedean (see Remark \ref{re:11}).
\end{enumerate}
Fix $1\leq i\leq \sharp\phi_\ell$. We consider
\begin{align}\label{for:68}
X_{t}=\exp(td\rho(k_i))\quad\text{ and }\quad X'_{t}=\exp(tk_i).
\end{align}
Clearly,  $X'_{t}\in K$. Since
\begin{align}\label{for:110}
\rho(X'_{t})&=\rho\bigl(\exp(tk_i)\bigl)=\exp(td\rho(k_i))=X_{t},
\end{align}
$X_{t}\in \rho(K)$. \eqref{for:110} is clear for $k$ archimedean; and see Proposition \ref{po:4} for $k$ non-archimedean.

As a first step towards the proof of Proposition \ref{po:7}, we prove:
\begin{lemma}\label{le:13}
For any $v\in \mathcal{W}_{c_0}^{\ell,i}$ (\ref{for:105} of Section \ref{sec:10}), we have
\begin{align*}
\bigl\lvert\pi_{\phi+\alpha_i}(X_{t}v)\bigl\rvert\geq&\left\{\begin{aligned} &\abs{t}c_0^{1/3^{\ell}}-C\abs{t}^{2}-2Cc_0^{1/3^{(\ell-1)}},\quad & \mathcal{W}_{c_0}^{\ell,i}&\neq \mathcal{W}_{c_0}^{\sharp\Phi_1-1,\sharp\varrho}\\
&\frac{1}{8}C_0\abs{t}-C\abs{t}^{2}-2Cc_0^{1/3^{(\sharp\Phi_1-2)}},\quad& \mathcal{W}_{c_0}^{\ell,i}&= \mathcal{W}_{c_0}^{\sharp\Phi_1-1,\sharp\varrho}.
 \end{aligned}
 \right.
\end{align*}
Here $\pi_{\phi+\alpha_i}$ is defined in \ref{for:142} of Section \ref{sec:10}.
\end{lemma}
\begin{proof} The proof contains two parts: we will give detailed description of $X_{t}$-orbits in $\mathcal{W}^{\ell,i,c_0}$  and $\mathcal{W}_{c_0}^{\ell,i}$ (\ref{for:105} of Section \ref{sec:10}) respectively.
\begin{sect2}\noindent{\small{\bf Analysis of $X_{t}$-orbits in $\mathcal{W}^{\ell,i,c_0}$}}.\\
For any $\nu\in \mathcal{W}^{\ell,i,c_0}$ write
\begin{align}\label{for:81}
\nu=\overbrace{\sum_{\omega<\phi_\ell}\nu_\omega}^{\nu_1}+\overbrace{\sum_{j<i}\nu^j_{\phi_\ell}}^{\nu_4}+\nu^i_{\phi_\ell}+\overbrace{\sum_{i<j}\nu^j_{\phi_\ell}}^{\nu_2}+
\overbrace{\sum_{\phi_\ell<\omega}\nu_{\omega}}^{\nu_3}.
\end{align}
We will consider $X_{t}$ orbits for $\nu_2$, $\nu_3$, $\nu_4$ and $\nu^i_{\phi_\ell}$ respectively (see notations in \ref{sec:16} of Section \ref{sec:10}). Using \eqref{for:110} we see that $X_{t}$ has the form
\begin{align}
&X_{t}=\exp\bigl(td\rho(k_i)\bigl)=\sum_{j\geq 0}(j!)^{-1}t^{j}d\rho(k_i)^j.
\end{align}
Notice that for any $\chi\in\Phi_1$,  if $\pi_{\phi+\alpha_i}\bigl(d\rho (k_i)^{j}v\bigl)\neq 0$ where $v\in V_\chi$, then $\phi+\alpha_i-\chi= l\alpha_i$, $-j\leq l\leq j$. Hence we have
\begin{align}\label{for:4}
   \pi_{\phi+\alpha_i}\bigl( X_{t}\nu_3\bigl)&=t^{2}\mathfrak{o}(\nu_3,t),\qquad\text{(\ref{for:84} of Section \ref{sec:10}).}
   \end{align}
 For any $v\in V_\phi$ we find that
\begin{align}\label{for:78}
   \pi_{\phi+\alpha_i}\bigl( X_{t}v\bigl)&=\sum_{0\leq j\leq 1}(j!)^{-1}t^{j}\pi_{\phi+\alpha_i}\bigl(d\rho(k_i)^jv\bigl)\notag\\
   &+\sum_{j\geq 2}(j!)^{-1}t^{j}\pi_{\phi+\alpha_i}\bigl(d\rho(k_i)^jv\bigl)\notag\\
   &=t\pi_{\phi+\alpha_i}\bigl(d\rho(k_i)v\bigl)+t^2\abs{v}\mathfrak{o}(\abs{v}^{-1}v,t)\notag\\
   &=t\pi_{\phi+\alpha_i}\bigl(d\rho(\mathfrak{c}_i)v\bigl)+t^2\abs{v}\mathfrak{o}(\abs{v}^{-1}v,t).
\end{align}
Using \eqref{for:2} in \ref{sec:16} of Section \ref{sec:10} we see that $d\rho(\mathfrak{c}_i)(v_{\phi_\ell}^j)=0$ if $j>i$. In particular, $d\rho(\mathfrak{c}_i)(\nu_2)=0$. It follows from \eqref{for:4} and \eqref{for:78} that
\begin{align}\label{for:130}
&\pi_{\phi+\alpha_i}\bigl( X_{t}(\nu-\nu_1)\bigl)\notag\\
&=t\pi_{\phi+\alpha_i}\bigl(d\rho(\mathfrak{c}_i)\nu^i_{\phi_\ell}\bigl)+t\pi_{\phi+\alpha_i}\bigl(d\rho(\mathfrak{c}_i)\nu_4\bigl)+t^2\mathfrak{o}(\nu,t).
\end{align}
\end{sect2}
\begin{sect2}\noindent{\small{\bf Analysis of $X_{t}$-orbits in $\mathcal{W}_{c_0}^{\ell,i}$}}.\\
For any $\nu\in \mathcal{W}^{\ell,i,c_0}$,  any $\mu\in V_\lambda$ with $\abs{\mu}< 2c_0$, we have
\begin{align}\label{for:86}
 &\bigl\lvert\pi_{\phi+\alpha_i}(X_{t}(\nu+\mu))\bigl\rvert\notag\\
 &=\Bigl\lvert\pi_{\phi+\alpha_i}(X_{t}(\nu-\nu_1))+\pi_{\phi+\alpha_i}(X_{t}(\nu_1))+\pi_{\phi+\alpha_i}(X_{t}(\mu))\Bigl\rvert,
\end{align}
where $\nu_1$ is as defined in \eqref{for:81} for $\nu$.

Since $X_{t}\in \rho(K)$ and $\rho(K)$ preserves the norm $\abs{\cdot}$ on $V$ (see \ref{for:84} in Section \ref{sec:10}) we have
\begin{align}\label{for:5}
\bigl\lvert\pi_{\phi+\alpha_i}(X_{t}(\mu))\bigl\rvert\leq \bigl\lvert X_{t}(\mu)\bigl\rvert=\abs{\mu}<2c_0.
\end{align}
From definitions of $\theta^l_{\phi_m}$ (see \ref{for:105} in Section \ref{sec:10}) we see that
 \begin{align}\label{for:30}
\bigl\lvert\pi_{\phi+\alpha_i}(X_{t}(\nu_1))\bigl\rvert\leq \bigl\lvert X_{t}(\nu_1)\bigl\rvert=\abs{\nu_1}\leq Cc_0^{1/3^{(\ell-1)}}.
 \end{align}
 Then it follows from \eqref{for:130}, \eqref{for:86}, \eqref{for:5} and \eqref{for:30} that
\begin{align}\label{for:55}
 &\bigl\lvert\pi_{\phi+\alpha_i}(X_{t}(\nu+\mu))\bigl\lvert\notag\\
 &\geq\bigl\lvert\pi_{\phi+\alpha_i}(X_{t}(\nu-\nu_1))\bigl\lvert-\bigl\lvert\pi_{\phi+\alpha_i}(X_{t}(\nu_1))\bigl\lvert-
\bigl\lvert\pi_{\phi+\alpha_i}(X_{t}(\mu))\bigl\lvert\notag\\
 &\geq \Bigl\lvert t\pi_{\phi+\alpha_i}\Bigl(d\rho(\mathfrak{c}_i)\nu^i_{\phi_\ell}\bigl)\bigl\lvert-\Bigl\lvert t\pi_{\phi+\alpha_i}\bigl(d\rho(\mathfrak{c}_i)\nu_4\bigl)\Bigl\lvert-
 C\abs{t}^{2}-
 Cc_0^{1/3^{(\ell-1)}}-2c_0
 \notag\\
 &\geq \abs{t}\bigl(C_0\abs{\nu^i_{\phi_\ell}}-C_0'\abs{\nu_4}\bigl)-C\abs{t}^{2}-2Cc_0^{1/3^{(\ell-1)}}.
 \end{align}
($C_0$ is in \eqref{for:88} and $C_0'$ and $n_i$ are in \ref{sec:14} of Section \ref{sec:10}).

\begin{sect6}\label{sec:23}$\mathcal{W}_{c_0}^{\ell,i}\neq \mathcal{W}_{c_0}^{\sharp\Phi_1-1,\sharp\varrho}$\\
From \eqref{for:55} one has
\begin{align*}
 &\bigl\lvert\pi_{\phi+\alpha_i}(X_{t}(\nu+\mu))\bigl\lvert\\
 &\geq c_0^{1/3^{\ell}}\abs{t}\bigl(C_0n_i-\sum_{j=1}^{i-1}2C_0'n_j\bigl)-C\abs{t}^{2}-2Cc_0^{1/3^{(\ell-1)}}.
 \end{align*}
 By definition of $n_i$, it follows that
\begin{align}\label{for:85}
\bigl\lvert\pi_{\phi+\alpha_i}(X_{t}(\nu+\mu))\bigl\lvert\geq\abs{t}c_0^{1/3^{\ell}}-C\abs{t}^{2}-2Cc_0^{1/3^{(\ell-1)}}.
\end{align}
\end{sect6}
\begin{sect6}$\mathcal{W}_{c_0}^{\ell,i}= \mathcal{W}_{c_0}^{\sharp\Phi_1-1,\sharp\varrho}$\\
Note that now $\abs{\nu^{\sharp\Phi_1-1}_{\varrho}}\geq \frac{1}{4}$.  Also from \eqref{for:55}, one has
\begin{align}
 &\bigl\lvert\pi_{\phi+\alpha_i}(X_{t}(\nu+\mu))\bigl\lvert\notag\\
&\geq \abs{t}\bigl(\frac{1}{4}C_0-Cc_0^{1/3^{(\sharp\Phi_1-1)}}\bigl)-C\abs{t}^{2}-2Cc_0^{1/3^{(\sharp\Phi_1-2)}}\notag\\
&\geq\frac{1}{8}C_0\abs{t}-C\abs{t}^{2}-2Cc_0^{1/3^{(\sharp\Phi_1-2)}},
 \end{align}
here we use the fact that $\frac{1}{4}C_0-Cc_0^{1/3^{(\sharp\Phi_1-1)}}>\frac{1}{8}C_0$ if $c_0$ is small enough. Hence we finish the proof.
\end{sect6}

\end{sect2}
\end{proof}
Next, we will consider disjoint $\rho(K)$-orbits of $\mathcal{W}_{c_0}$.
\begin{lemma}\label{le:10}
Suppose $c_0$ is small enough. For each $\mathcal{W}_{c_0}^{\ell,i}$, there are at least $\Bigl[C^{-1}c_0^{-1/3^{\ell}}\Bigl]$ different elements $\tau_j$ in $K$ such that
\begin{align*}
\rho(\tau_j)\mathcal{W}_{c_0}^{\ell,i}\bigcap\rho(\tau_l)\mathcal{W}_{c_0}^{\ell,i}=\emptyset,\qquad\text{ if }l\neq j.
\end{align*}
If $\phi=\sharp\Phi_1-1$ and $i=\sharp\varrho$, we can improve this number to $\Bigl[C^{-1}c_0^{-1/3^{(\sharp\varrho-2)}}\Bigl]$.
\end{lemma}
\begin{proof}
\begin{sect3}$\mathcal{W}_{c_0}^{\ell,i}\neq \mathcal{W}_{c_0}^{\sharp\Phi_1-1,\sharp\varrho}$\\
For any $v\in \mathcal{W}_{c_0}^{\ell,i}$,  using Lemma \ref{le:13} and relation \eqref{for:110} we have
\begin{align}\label{for:7}
 &\Bigl\lvert\pi_{\phi+\alpha_i}\bigl(\rho(X'_{t}) v\bigl)\Bigl\rvert-3C_1c_0^{1/3^{(\ell-1)}}\notag\\
 &=\Bigl\lvert\pi_{\phi+\alpha_i}\bigl(X_{t}(v)\bigl)\Bigl\rvert-3C_1c_0^{1/3^{(\ell-1)}}\notag\\
 &> \abs{t}c_0^{1/3^{\ell}}-C\abs{t}^{2}-2Cc_0^{1/3^{(\ell-1)}}-3C_1c_0^{1/3^{(\ell-1)}}\notag\\
  &=\abs{t}(c_0^{1/3^{\ell}}-C\abs{t})-(2C+3C_1)c_0^{1/3^{(\ell-1)}}.
 \end{align}
 Here $C_1$ is in \ref{sec:14} of Section \ref{sec:10}.

It follows that $\Bigl\lvert\pi_{\phi+\alpha_i}\bigl(\rho(X'_{t}) v\bigl)\Bigl\rvert-3C_1c_0^{1/3^{(\ell-1)}}>0$ if
\begin{align}\label{for:145}
(4C+6C_1)c_0^{2/3^{\ell}}\leq\abs{t}\leq \frac{1}{2}C^{-1}c_0^{1/3^{\ell}}.
\end{align}
\end{sect3}
\begin{sect3} \label{for:29}$\mathcal{W}_{c_0}^{\ell,i}= \mathcal{W}_{c_0}^{\sharp\Phi_1-1,\sharp\varrho}$\\
For any $v\in \mathcal{W}_{c_0}^{\sharp\Phi_1-1,\sharp\varrho}$,  using Lemma \ref{le:13} we have
\begin{align*}
 &\Bigl\lvert\pi_{\phi+\alpha_i}\bigl(\rho(X'_{t}) v\bigl)\Bigl\rvert-3C_1c_0^{1/3^{(\sharp\phi-2)}}\notag\\
 &=\Bigl\lvert\pi_{\phi+\alpha_i}\bigl(X_{t}(v)\bigl)\Bigl\rvert-3C_1c_0^{1/3^{(\sharp\phi-2)}}\notag\\
 &> \frac{1}{8}C_0\abs{t}-C\abs{t}^{2}-2Cc_0^{(\frac{1}{3})^{(\sharp\phi-2)}}-3C_1c_0^{1/3^{(\sharp\phi-2)}}\\
 &=\abs{t}\bigl(\frac{1}{8}C_0-C\abs{t}\bigl)-(2C+3C_1)c_0^{1/3^{(\sharp\phi-2)}}.
 \end{align*}
 Then  for small enough $c_0$, $\Bigl\lvert\pi_{\phi+\alpha_i}\bigl(\rho(X'_{t}) v\bigl)\Bigl\rvert-3C_1c_0^{1/3^{(\sharp\phi-2)}}>0$ if
\begin{align}\label{for:9}
16(2C+3C_1)C_0^{-1}c_0^{1/3^{(\sharp\phi-2)}}\leq\abs{t}\leq \sigma_1,
\end{align}
here $\sigma_1$ is small enough that $t\rightarrow X_t'$ is injective if $\abs{t}\leq \sigma_1$.
\end{sect3}
Since $\bigl\lvert\pi_{\phi+\alpha_i}(w)\bigl\rvert\leq 2C_1c_0^{1/3^{(\ell-1)}}$ for any $w\in\mathcal{W}_{c_0}^{\ell,i}$, \eqref{for:145} and \eqref{for:9} imply that if $c_0$ is small enough, then for any $v\,,w\in\mathcal{W}_{c_0}^{\ell,i}$,
\begin{align*}
\Bigl\lvert\pi_{\phi+\alpha_i}\bigl(\rho(X'_{t}) v\bigl)\Bigl\rvert>\bigl\lvert\pi_{\phi+\alpha_i}(w)\bigl\rvert\qquad\text{ if }
\end{align*}
\begin{align}\label{for:136}
\left\{\begin{aligned} &(4C+6C_1)c_0^{2/3^{\ell}}\leq\abs{t}\leq \frac{1}{2}C^{-1}c_0^{1/3^{\ell}},\quad & \mathcal{W}_{c_0}^{\ell,i}\neq \mathcal{W}_{c_0}^{\sharp\Phi_1-1,\sharp\varrho}\\
&16(2C+3C_1)C_0^{-1}c_0^{1/3^{(\sharp\phi-2)}}\leq\abs{t}\leq \sigma_1,\quad& \mathcal{W}_{c_0}^{\ell,i}= \mathcal{W}_{c_0}^{\sharp\Phi_1-1,\sharp\varrho}.
 \end{aligned}
 \right.
\end{align}
 When $k$ is archimedean, let
\begin{enumerate}
  \item $t_0=(4C+6C_1)c_0^{2/3^{\ell}}$ and $n_0=\Bigl[\frac{1}{2}C^{-1}(4C+6C_1)^{-1}c_0^{-1/3^{\ell}}\Bigl]$ if $\mathcal{W}_{c_0}^{\ell,i}\neq \mathcal{W}_{c_0}^{\sharp\Phi_1-1,\sharp\varrho}$;
  \item $t_0=16(2C+3C_1)C_0^{-1}c_0^{1/3^{(\sharp\phi-2)}}$ and $n_0=\Bigl[(32C+48C_1)^{-1}C_0\sigma_1c_0^{-1/3^{(\sharp\phi-2)}}\Bigl]$ if $\mathcal{W}_{c_0}^{\ell,i}= \mathcal{W}_{c_0}^{\sharp\Phi_1-1,\sharp\varrho}$.
\end{enumerate}
 Then for any $0\leq \ell_1\neq\ell_2\leq n_0$, it follows from \eqref{for:136} that
\begin{align*}
&\rho\bigl(\exp(t_0(\ell_1-\ell_2)k_i)\bigl)\mathcal{W}_{c_0}^{\ell,i}\bigcap \mathcal{W}_{c_0}^{\ell,i}\\
&=\rho\bigl(X'_{t_0(\ell_1-\ell_2)}\bigl)\mathcal{W}_{c_0}^{\ell,i}\bigcap \mathcal{W}_{c_0}^{\ell,i}=\emptyset,
\end{align*}
which implies that
\begin{align*}
 &\rho\bigl(\exp(t_0\ell_1k_i)\bigl)\mathcal{W}_{c_0}^{\ell,i}\bigcap\rho\bigl(\exp(t_0\ell_2k_i)\bigl)\mathcal{W}_{c_0}^{\ell,i}=\emptyset.
 \end{align*}
Let $\tau_j=\exp(t_0 jk_i)=X'_{t_0j}$ where $0\leq j\leq n_0$, then smallness of $c_0$ guarantees that these $\tau_j$ are pairwise different. It is clear that these $\tau_j$ satisfy the requirement. We have thus proved the case
for $k$ archimedean.

Now suppose $k$ is non-archimedean and $\mathcal{W}_{c_0}^{\ell,i}\neq \mathcal{W}_{c_0}^{\sharp\Phi_1-1,\sharp\varrho}$. For any $0<c_0<1$ there exist
$\abs{q}< \eta_0\leq 1$ and $z_0\in\NN$ such that $c_0^{1/3^{\ell}}=\abs{q}^{z_0}\eta_0$. Choose the smallest $m\in\NN$ such that $\abs{q}^{-m}\geq\max\{4C+6C_1,\,2C\}$.

Any element in $\mathcal{O}$ can be written as a power series $b_0+b_1q+b_2q^2+\cdots$ where the $b's$ are Teichm\"{u}ller representatives
which are $0$ together with the group of $(\abs{q}^{-1}-1)$st roots of unity in $k$. Denote by $Q_{\ell, i}=\{\sum_{j=z_0+m}^{2z_0-m}b_jq^{j}\}$.
It is easy to see that for any $t_1\neq t_2\in Q_{\ell,i}$, $t=t_1-t_2$ satisfy the condition \eqref{for:145} and hence it follows from \eqref{for:136}  that
\begin{align*}
&\rho\bigl(\exp((t_1-t_2)k_i)\bigl)\mathcal{W}_{c_0}^{\ell,i}\bigcap \mathcal{W}_{c_0}^{\ell,i}\\
&=\rho\bigl(X'_{t_1-t_2}\bigl)\mathcal{W}_{c_0}^{\ell,i}\bigcap \mathcal{W}_{c_0}^{\ell,i}=\emptyset,
\end{align*}
hence we have
\begin{align*}
 &\rho\bigl(\exp(t_1k_i)\bigl)\mathcal{W}_{c_0}^{\ell,i}\bigcap\rho\bigl(\exp(t_2k_i)\bigl)\mathcal{W}_{c_0}^{\ell,i}=\emptyset.
 \end{align*}
Let $\tau_t=\exp(tk_i)=X'_{t}$ where $t\in Q_{\ell, i}$ then these $\tau_t$ satisfy the requirement and are pairwise different (see \eqref{for:50} and \eqref{for:62}). Moreover, the number of these $\tau_t$ is
\begin{align*}
\sharp Q_{\ell, i}=\abs{q}^{-(z_0-2m+1)}\geq \abs{q}^{2m}c_0^{-1/3^{\ell}}\geq\frac{\abs{q}^{2}c_0^{-1/3^{\ell}}}{\max\{(4C+6C_1)^2,\,4C^2\}}.
\end{align*}
If $\mathcal{W}_{c_0}^{\ell,i}=\mathcal{W}_{c_0}^{\sharp\Phi_1-1,\sharp\varrho}$. There exists $z_1,\,z_2\in\NN$ and $\abs{q}< \eta_1\leq 1$ such that $\abs{q}\sigma_1\leq\abs{q}^{z_1}\leq\sigma_1$ and $c_0^{1/3^{(\sharp\phi-2)}}=\abs{q}^{z_2}\eta_1$. Choose smallest $m_1\in\NN$ such that $\abs{q}^{-m_1}\geq 16(2C+3C_1)C_0^{-1}$. Then we consider
$Q_{\sharp\Phi_1-1, \sharp\varrho}=\{\sum_{j=z_1}^{z_2-m_1}b_jq^{j}\}$. Similar to earlier arguments, $\sharp Q_{\sharp\Phi_1-1, \sharp\varrho}$ gives the number of different $\tau_t$. Also we have
\begin{align*}
\sharp Q_{\sharp\Phi_1-1, \sharp\varrho}&=\abs{q}^{-(z_2-z_1-m_1+1)}\geq \abs{q}^{z_1+m_1}c_0^{-1/3^{\sharp\varrho-2}}\notag\\
&\geq\frac{\abs{q}^2\sigma_1C_0}{16(2C+3C_1)}c_0^{-1/3^{\sharp\varrho-2}}.
\end{align*}
Then we finished the proof for $k$ non-archimedean.

\end{proof}

We are now in a position to proceed with the proof of Proposition \ref{po:7}.
\subsection{Proof of Proposition \ref{po:7}}
If $k$ is archimedean, let
\begin{align*}
\text{Cone}_1(c_0,s)'=\Bigl\{\frac{\nu}{\abs{\nu}}:\nu\in \text{Cone}_1(c_0,s)\Bigl\};
\end{align*}
if $k$ is non-archimedean, let
\begin{align*}
\text{Cone}_1(c_0,s)'=\{\nu q^{-o(\nu)}:\nu\in \text{Cone}_1(c_0,s)\}.
\end{align*}
where $o$ is the normalized valuation of $k$ (for the uniformizer $q$). It is clear that elements in $\text{Cone}_1(c_0,s)'$ are of norm $1$.

For any $v\in \text{Cone}_1(c_0,s)'$, write $v=v_1+v_2$ where $v_1\in V_\lambda$ and $v_2\in\sum_{\psi\neq \lambda}V_\psi$ (see \ref{for:142} of Section \ref{sec:10}), then $\abs{v_1}\leq c_0s^{-1}$
and $1-c_0s^{-1}\leq\abs{v_2}\leq 1+c_0s^{-1} $. If $c_0$ is small enough (relative to $s$), then $c_0s^{-1}<\frac{1}{8}$, which implies $\frac{3}{4}\leq\abs{v_2}\leq\frac{5}{4}$. Then we see that
$\text{Cone}_1(c_0,s)'\subset \mathcal{W}_{c_0s^{-1}}.$

Lemma \ref{le:10} implies that there is an open cover of $\mathcal{W}_{c_0s^{-1}}\subset\bigcup_{i=1}^{\ell_0} \mathcal{Y}_{i}$ ($\ell_0$ is defined in \eqref{for:49})
and for each $\mathcal{Y}_{i}$, there are at least $[C^{-1}(c_0s^{-1})^{\mathfrak{q}}]$ different elements $\tau^i_j\in K$ such that if $\ell\neq j$
\begin{align}\label{for:65}
 \rho(\tau^i_j)\mathcal{Y}_{i}\bigcap\rho(\tau^i_\ell)\mathcal{Y}_{i}=\emptyset,
\end{align}
where $\mathfrak{q}=-(\frac{1}{3})^{(\sharp\Phi_1-1)}$; and $\mathfrak{q}=-(\frac{1}{3})^{(\sharp\Phi_1-2)}$ if $\dim V_\varrho=1$.

It is clear that $\bigcup_{i=1}^{\ell_0} \mathcal{Y}_{i}$ is also an open cover of $\text{Cone}_1(c_0,s)'$. When $k$ is archimedean, let
\begin{align*}
\mathcal{E}_i=\{t\nu:\nu\in \mathcal{Y}_{i}\text{ and }t\in\RR^+\};
\end{align*}
when $k$ is non-archimedean, let
\begin{align*}
\mathcal{E}_i=\{q^n\nu:\nu\in \mathcal{Y}_{i}\text{ and }n\in\ZZ\}.
\end{align*}
It is easy to check that the sets $\mathcal{E}_i$, $1\leq i\leq \ell_0$ are open and
\begin{align}\label{for:47}
\text{Cone}_1(c_0,s)\subset \bigcup_{i=1}^{\ell_0} \mathcal{E}_{i}.
\end{align}
It is now fairly clear what one has to do. We have to prove the following:
\begin{align}\label{for:143}
 \rho(\tau^i_j)\mathcal{E}_{i}\bigcap\rho(\tau^i_\ell)\mathcal{E}_{i}=\emptyset,\qquad \ell\neq j.
\end{align}
Indeed, if it is not true, then there exist $x_1\,,x_2\in\mathcal{Y}_{i}$ and $t_1\,,t_2\in \RR^+$ (resp. $t_1\,,t_2\in \{q^n:n\in\ZZ\}$) such that $\rho(\tau_j)(t_1x_1)=\rho(\tau_\ell)(t_2x_2) $. Since $\rho(K)$ preserves the norm on $V$ (see \ref{for:84} in Section \ref{sec:10}),  $\abs{t_1}=\abs{t_2}$, which gives $t_1=t_2$. Then it follows that $\rho(\tau_j)(x_1)=\rho(\tau_\ell)(x_2)$,  which is a contradiction to \eqref{for:65}.
Then we thus proved \eqref{for:143}. It is clear that \eqref{for:47} and \eqref{for:143} imply Proposition \ref{po:7}.
\section{matrix coefficients of $G\ltimes_\rho V$ restricted to $G$}\label{sec:11}
In this part, the basic ideas are as follows:
\begin{enumerate}
  \item \label{for:14}The crucial step in proving Theorem \ref{th:11} is Proposition \ref{po:6}. When $\rho$ is irreducible, we analyze unitary representations of $G\ltimes_\rho V$ without non-trivial $V$-fixed vectors, and show an upper bound of $K$-finite matrix coefficients on $G$ for a dense set of vectors. There are four steps in proof of Proposition \ref{po:6}: step $1$: reduction to consider projected matrix coefficients by using the deformation of $K$-orbits on $\widehat{V}$ under the action of maximal $k$-split torus; step $2$: once only the $K$-orbits are mattered,  we can consider the restriction of the representation on $K\ltimes_\rho V$, which is weekly contained in its regular representation, and then Lemma \ref{le:1} allows the reduction to  the consideration of the regular representation; step $3$: using central idempotents in $L^1(K)$ to reduce to the consideration of $K$-fixed vectors; step $4$: using disjoint $K$-orbits to get upper bounds of matrix coefficients on $G$ for $K$-fixed vectors.

  \item The upper bound in \eqref{for:14} shows that there exists $p>0$ such that the $G$ matrix coefficients are strongly $L^{p+\epsilon}$ for all $\epsilon>0$; and by the work of Cowling, Haggerup and Howe \cite{cow1} (see Theorem $2.5$), we have
a passage from a uniform $L^p$-bound to a uniform pointwise bound.

  \item \label{for:15}For general $\rho$, there are two steps in our proof: step $1$: reduction to the case of $\rho$ irreducible if the unitary representation is irreducible; step $2$: reduction to the consideration of irreducible representations by using a Howe's trick (see Proposition \ref{th:2}).
\end{enumerate}

\subsection{Projected matrix coefficients for $K$-finite vectors} Recall notations in \ref{sec:18} of Section \ref{sec:10}. Suppose $\rho$ is irreducible and excellent on $V$. Then we have:
\begin{proposition}\label{po:6}
Suppose $K$ is a good maximal compact subgroup of $G$. Let $\Pi$ be a unitary representation of $G\ltimes_\rho V$ on a Hilbert space $\mathcal{H}$ which contains no non-trivial $V$-fixed vectors. Then there is a dense set of vectors $\xi$, $\eta$ in $\mathcal{H}$ such that if $a\in D^+$ and $g=k_1a\omega k_2\in KD^+FK=G$ (recall
$F=\{e\}$  for $k$ archimedean), then
\begin{align}\label{for:8}
\bigl\lvert\phi_{\eta,\xi}(g)\bigl\rvert=\bigl\lvert\langle \Pi(g)\eta,\xi\rangle\bigl\rvert\leq C_{\eta,\xi}\bigl\lvert\lambda(a)\varrho(a^{-1})\bigl\rvert^{\frac{\mathfrak{q}}{2}},
\end{align}
where $C_{\eta,\xi}$ is a constant only depending on $\xi$ and $\eta$; and $\mathfrak{q}=-(\frac{1}{3})^{(\sharp\Phi_1-1)}$. Moreover, If $\dim V_\lambda=1$, $\mathfrak{q}=-(\frac{1}{3})^{(\sharp\Phi_1-2)}$.
\end{proposition}
\begin{proof} Recall notations in Section \ref{sec:22}. We first claim that $\rho^\ast$ is also an irreducible representation of $G$ on $V$. Indeed, if not, we find a $\rho^\ast(G)$-invariant proper subspace $0\neq W\subseteq V$. Let
\begin{align*}
W^\bot=\{v\in V:\varsigma(v,w)=0,\text{ for }\forall w\in W\}.
\end{align*}
It is clear that $W^\bot\neq 0$ and is a proper subspace of $V$. For any $w\in W$, $v\in W^\bot$ and $g\in G$, we find that
\begin{align*}
\varsigma(\rho(g)v,w)=\varsigma\bigl(v,(\rho(g^{-1})^{-1})^\tau w\bigl)=\varsigma\bigl(v,\rho^\ast(g^{-1}) w\bigl)=0.
\end{align*}
Then $W^\bot$ is $\rho(G)$-invariant, which contradicts the assumption that $\rho$ is irreducible.

Denote by $G_i$ the non-compact almost $k$-simple factors of $G$. We also claim that $\rho^\ast$ is excellent.  If not, let $v\neq 0$ be a $\rho^\ast(G_i)$-fixed vector for some $i$. By earlier arguments, $v^\bot$ is also a $\rho(G_i)$-invariant
hyperplane in $V$. By complete reducibility of $G_i$, we find a $1$-dimensional $\rho(G_i)$-invariant complement $\{tw:t\in k\}$ of $v^\bot$. Then for any $h\in G_i$ there exists a homomorphism  $\iota:G_i\rightarrow k$ such that $\rho(h)w=\iota(h)w$, and thus we have
\begin{align*}
\iota(h)\varsigma(w,v)=\varsigma(\rho(h)w,v)=\varsigma\bigl(w,\rho^\ast(h^{-1}) v\bigl)=\varsigma(w,v),\quad \forall h\in G_i.
\end{align*}
Since $\varsigma(w,v)\neq 0$, $\iota$ is trivial on $G_i$; and then we see that $\{tw:t\in k\}$ is fixed by $\rho(G_i)$, which contradicts the assumption that $\rho$ is excellent.
 In the expression of $\varsigma$, let $\{e_1,e_2,\cdots\}$ (see Section \ref{sec:1}) be composed of vectors from weight spaces, then we see $\lambda_1=-\varrho$ and $\varrho_1=-\lambda$ where $\lambda_1$ and $\varrho_1$ are highest and lowest weights of $\rho^\ast$ respectively.

\noindent\emph{$^\ast$ Reduction to consider projected matrix coefficents}. We assume notations and constructions in Section \ref{sec:10} with respect to the irreducible representation $\rho^\ast$. For a number $s>1$, set
\begin{align}\label{for:17}
X_s=\bigl\{\nu\in V: \frac{1}{s}\leq\abs{\nu}\leq s\bigl\}.
\end{align}
Let $\{P_X\}$ be the projection-valued measure associated to $\Pi$ on $\widehat{V}$. The assumption that there are no non-trivial $V$-fixed vectors implies that $P_0=0$. For any $s>2$ choose $f_s\in \mathcal{S}(V)$ such that
\begin{align}\label{for:76}
0\leq f_s\leq 1;\text{ and }f_s=1\text{ on }X_{s-1}\text{ and }f_s=0 \text{ on }V\backslash X_{s}.
\end{align}
$f_s$ can also be viewed as functions in $\mathcal{S}(\widehat{V})$ by the identification between $V$ and $\widehat{V}$ (see Section \ref{sec:22}).  Since $V\backslash \{0\}=\bigcup_{s>2} X_s$, from \eqref{for:79} and \eqref{for:89} we see that
\begin{align}\label{for:20}
\widehat{\Pi}(f_s)v\rightarrow v \qquad\text{as } s\rightarrow \infty,\qquad \forall v\in \mathcal{H}.
\end{align}
Since the $K$-finite vectors of $\Pi$ are dense by
Peter-Weyl theorem and a matrix coefficient $\phi_{\eta,\xi}$ depends bilinearly  on $\xi$ and $\eta$, it will be enough to prove \eqref{for:8} when $\xi$ and $\eta$ are $K$-finite vectors of length $1$. Moreover, from \eqref{for:20} it will suffice to prove \eqref{for:8} for vectors $\widehat{\Pi}(f_s)\eta$ and $\widehat{\Pi}(f_s)\xi$ for $s$ arbitrary.

Now consider the matrix coefficient $\phi_{\widehat{\Pi}(f_s)\eta,\widehat{\Pi}(f_s)\xi}(k_1a\omega k_2)$ where $a\in D^+$, $\omega\in F$ and $k_1,\,k_2\in K$.
By using \eqref{for:90} we have
\begin{align*}
&\phi_{\widehat{\Pi}(f_s)\eta,\widehat{\Pi}(f_s)\xi}(k_1a\omega k_2)\\
&=\Bigl\langle \Pi(k_1a\omega k_2)\widehat{\Pi}(f_s)\eta,\,\,\widehat{\Pi}(f_s)\xi\Bigl\rangle\notag\\
&=\Bigl\langle \Pi(a\omega)\widehat{\Pi}(k_2\star f_s)\Pi(k_2)\eta,\,\,\Pi(k_1^{-1})\widehat{\Pi}(f_s)\xi\Bigl\rangle\notag\\
&=\Bigl\langle \Pi(a\omega)\widehat{\Pi}(k_2\star f_s)\Pi(k_2)\eta,\,\,\widehat{\Pi}(k_1^{-1}\star f_s)\Pi(k_1^{-1})\xi\Bigl\rangle.
\end{align*}
Let
\begin{align}\label{for:82}
g_s=k_2\star f_s,\quad h_s=k_1^{-1}\star f_s, \quad\eta'=\Pi(k_2)\eta\quad\text{and}\quad\xi'=\Pi(k_1^{-1})\xi.
\end{align}
Notice that the norm $\abs{\cdot}$ on $V$ is $\rho^\ast(K)$ invariant (\ref{for:84} in Section \ref{sec:10}). Then for any $\kappa\in K$ and $s>2$, from \eqref{for:12} we see that $\kappa\star f_s$ also satisfies \eqref{for:76}. Especially, $g_s$ and $h_s$ satisfy \eqref{for:76}. Also, $\eta'$ and $\xi'$ are $K$-finite vectors of length $1$.

Further computations by using \eqref{for:24}, \eqref{for:6} and \eqref{for:90} show that
\begin{align*}
&\phi_{\widehat{\Pi}(f_s)\eta,\widehat{\Pi}(f_s)\xi}(k_1a\omega k_2)\\
&=\bigl\langle \Pi(a\omega)\widehat{\Pi}(g_s^{\frac{1}{2}})\widehat{\Pi}(g_s^{\frac{1}{2}})\eta',\,\,\widehat{\Pi}(h_s)\xi'\bigl\rangle\\
&=\Bigl\langle \widehat{\Pi}\bigl((a\omega)\star g_s^{\frac{1}{2}}\bigl) \Pi(a\omega)\widehat{\Pi}(g_s^{\frac{1}{2}})\eta',\,\,\widehat{\Pi}(h_s)\xi'\Bigl\rangle\\
&=\Bigl\langle \Pi(a\omega)\widehat{\Pi}(g_s^{\frac{1}{2}})\eta',\,\,\widehat{\Pi}\bigl((a\omega)\star g_s^{\frac{1}{2}}\bigl) \widehat{\Pi}(h_s)\xi'\Bigl\rangle\\
&=\Bigl\langle \Pi(a\omega)\widehat{\Pi}(g_s^{\frac{1}{2}})\eta',\,\,\widehat{\Pi}\bigl((a\omega)\star g_s^{\frac{1}{2}}\cdot h_s\bigl)\xi'\Bigl\rangle\\
&=\Bigl\langle \Pi(a\omega)\widehat{\Pi}(g_s^{\frac{1}{2}})\eta',\,\,\widehat{\Pi}(h_s^{\frac{1}{2}})\widehat{\Pi}\bigl((a\omega)\star g_s^{\frac{1}{2}}\cdot h_s^{\frac{1}{2}}\bigl)\xi_1\Bigl\rangle\\
&=\Bigl\langle \widehat{\Pi}(h_s^{\frac{1}{2}})\Pi(a\omega)\widehat{\Pi}(g_s^{\frac{1}{2}})\eta',\,\,\widehat{\Pi}\bigl((a\omega)\star g_s^{\frac{1}{2}}\cdot h_s^{\frac{1}{2}}\bigl)\xi'\Bigl\rangle\\
&=\Bigl\langle \Pi(a\omega)\widehat{\Pi}\bigl((a\omega)^{-1}\star h_s^{\frac{1}{2}}\bigl)\widehat{\Pi}(g_s^{\frac{1}{2}})\eta',\,\,\widehat{\Pi}\bigl((a\omega)\star g_s^{\frac{1}{2}}\cdot h_s^{\frac{1}{2}}\bigl)\xi'\Bigl\rangle\\
&=\Bigl\langle \Pi(a\omega)\widehat{\Pi}\bigl((a\omega)^{-1}\star h_s^{\frac{1}{2}}\cdot g_s^{\frac{1}{2}}\bigl)\eta',\,\,\widehat{\Pi}\bigl((a\omega)\star g_s^{\frac{1}{2}}\cdot h_s^{\frac{1}{2}}\bigl)\xi'\Bigl\rangle.
\end{align*}
Then it follows from the Cauchy-Schwarz inequality and the unitarity of $\Pi$ that
\begin{align}\label{for:19}
&\bigl\lvert \phi_{\widehat{\Pi}(f_s)\eta,\widehat{\Pi}(f_s)\xi}(k_1a\omega k_2)\bigl\rvert \notag\\
&\leq \Bigl\lVert \langle \Pi(a\omega)\widehat{\Pi}\bigl((a\omega)^{-1}\star h_s^{\frac{1}{2}}\cdot g_s^{\frac{1}{2}}\bigl)\eta'\Bigl\rVert\cdot\Bigl\lVert\widehat{\Pi}\bigl((a\omega)\star g_s^{\frac{1}{2}}\cdot h_s^{\frac{1}{2}}\bigl)\xi'\Bigl\rVert\notag\\
&= \Bigl\lVert \widehat{\Pi}\bigl((a\omega)^{-1}\star h_s^{\frac{1}{2}}\cdot g_s^{\frac{1}{2}}\bigl)\eta'\Bigl\rVert\cdot\Bigl\lVert \widehat{\Pi}\bigl((a\omega)\star g_s^{\frac{1}{2}}\cdot h_s^{\frac{1}{2}}\bigl)\xi'\Bigl\rVert.
\end{align}
Next, we will get estimates of $\Bigl\lVert \widehat{\Pi}\bigl((a\omega)^{-1}\star h_s^{\frac{1}{2}}\cdot g_s^{\frac{1}{2}}\bigl)\eta'\Bigl\rVert$ and $\Bigl\lVert \widehat{\Pi}\bigl((a\omega)\star g_s^{\frac{1}{2}}\cdot h_s^{\frac{1}{2}}\bigl)\xi'\Bigl\rVert$ respectively.
\smallskip

\noindent\emph{$^{\ast\ast}$ Reduction to consider action of the regular representation}. Note that $K\ltimes_\rho V$ is an amenable group. Recall the well known fact that any unitary representation of an amenable group is weakly contained in its regular representation (cf. \cite[Ch1, 5.5.3]{Margulis}) or \cite[Proposition 7.3.6]{Zimmer}. Hence $\Pi$ viewed as a unitary representation of $K\ltimes_\rho V$ is weakly contained in the  regular representation $\Lambda$ of $K\ltimes_\rho V$ on $L^2(K\ltimes_\rho V)$, with norm and inner product denoted by $\norm{\cdot}_2$ and $\langle\,,\,\rangle_2$ respectively. Let $d\kappa$, $d\mathfrak{n}$ and $d\mathfrak{\widehat{n}}$ be the normalized Haar measures of $K$, $V$ and $\widehat{V}$ respectively so that $\int_K dk=1$ and $d\mathfrak{n}$ and $d\mathfrak{\widehat{n}}$ are as stated in Theorem \ref{th:1}. Since $\rho(K)\subset SL(V)$, $d\mathfrak{n} d\kappa$ is a bi-invariant measure on $K\ltimes_\rho V$. As $K\ltimes_\rho V$ is unimodular, the bi-invariant Haar measure $ds$ of $K\ltimes_\rho V$ can be normalized so that
\begin{align*}
    \int_{K\ltimes_\rho V}f(s)ds=\int_K\int_{V}f(\kappa\mathfrak{n})d\mathfrak{n}d\kappa,\qquad \forall f\in C_c(K\ltimes_\rho V),
\end{align*}
where $C_c(K\ltimes_\rho V)$ is the space of continuous functions on $K\ltimes_\rho V$ with compact support. For any $f(\kappa\mathfrak{n})\in L^2(K\ltimes_\rho V)$ (resp. $f(\kappa\widehat{\mathfrak{n}})\in L^2(K\ltimes_\rho \widehat{V})$), we use $\widehat{f}(\kappa\,\cdot)$ and $\mathcal{F}(f)(\kappa\,\cdot)$ to denote the Fourier transform and inverse Fourier transform on factor $V$ (resp. $\widehat{V}$) respectively.

We now make a slight digression to state an important proposition (see \cite{cow1}), which allows us to consider matrix coefficients of $\Pi$ (resp. $\Lambda$) restricted to $\mathcal{H}_\mu$ (resp $L^2(K\ltimes_\rho V)_\mu$)-isotropic subspaces for any $K$-finite vector $\mu$.
\begin{proposition}\label{po:9}
Let $\Upsilon$ be a unitary representation of a compact group $K$ and let $\mu$ be a $K$-finite vector of $\Upsilon$. Denote by $\mathcal{H}_\mu$ the span of $\Upsilon(K)\mu$, which is finite dimensional. Then there exists a unique function $e_\mu$ in $C(K)$ so that
\begin{align*}
e_\mu\ast e_\mu=e_\mu, \quad\Upsilon(\overline{e_\mu})\mu=\mu,\quad\text{and }d_\mu=\dim\bigl(\emph{span}\Upsilon(K)\mu\bigl)=\norm{e_\mu}_2^2.
\end{align*}
\end{proposition}
We now consider the projected matrix coefficient
\begin{align*}
\phi^{\Pi}:(x,y)&\rightarrow \Bigl\langle \widehat{\Pi}\bigl((a\omega)^{-1}\star h_s^{\frac{1}{2}}\cdot g_s^{\frac{1}{2}}\bigl)\Pi(x)\eta',\,\widehat{\Pi}\bigl((a\omega)^{-1}\star h_s^{\frac{1}{2}}\cdot g_s^{\frac{1}{2}}\bigl)\Pi(y)\eta'\Bigl\rangle\\
&=\Bigl\langle \widehat{\Pi}\bigl((a\omega)^{-1}\star h_s\cdot g_s\bigl)\Pi(x)\eta',\,\Pi(y)\eta'\Bigl\rangle.
\end{align*}
Then from the proposition, there is a self-adjoint projection $e_{\eta'}\in C(K)$ so that
\begin{align*}
e_{\eta'}\ast\phi^{\Pi}\ast e_{\eta'}=\Bigl\langle \widehat{\Pi}\bigl((a\omega)^{-1}\star h_s\cdot g_s\bigl)\Pi(x)\Pi(\overline{e_{\eta'}})\eta',\,\Pi(y)\Pi(\overline{e_{\eta'}})\eta'\Bigl\rangle=\phi^{\Pi}.
\end{align*}
Since $\Pi$ is weekly contained in $\Lambda$, using Lemma \ref{le:1} we can approximate $\phi^{\Pi}$, uniformly on compacta,  by sums of projected matrix coefficients
\begin{align*}
\phi^{\Lambda}_i: (x,y)\rightarrow \Bigl\langle \widehat{\Lambda}\bigl((a\omega)^{-1}\star h_s\cdot g_s\bigl)\Lambda(x)\eta_i,\,\Lambda(y)\eta_i\Bigl\rangle_2,
\end{align*}
satisfying the condition that
$\sum_{i=1}^n\norm{\eta_i}_2^2\leq \norm{\eta'}^2$.

Since $e_{\eta'}$ has compact support, and further $\Lambda(\overline{e_{\eta'}})$ is a projection, we can approximate $\phi^{\Pi}$ by sums of projected matrix coefficients $\sum_{i=1}^n e_{\eta'}\ast\phi^{\Lambda}_i\ast e_{\eta'}$, where
\begin{align*}
e_{\eta'}\ast\phi^{\Lambda}_i\ast e_{\eta'}=\Bigl\langle \widehat{\Lambda}\bigl((a\omega)^{-1}\star h_s\cdot g_s\bigl)\Lambda(x)\Lambda(\overline{e_{\eta'}})\eta_i,\,\Lambda(y)\Lambda(\overline{e_{\eta'}})\eta_i\Bigl\rangle_2
\end{align*}
and
\begin{align*}
\sum_{i=1}^n \bigl\lVert\Lambda(\overline{e_{\eta'}})\eta_i\bigl\lVert_2^2\leq \norm{\eta'}^2.
\end{align*}
Consequently, to get the estimate of $\phi^{\Pi}(e,e)$, where $e$ is the identity in $K\ltimes_\rho V$, it will suffice to get the estimate of $\sum_{i=1}^n \phi^{\Lambda}_i(e,e)$, for $\overline{e_{\eta'}}\ast\eta_i=\eta_i $ for each $i$, subject to $\sum_{i=1}^n \norm{\eta_i}_2^2\leq \norm{\eta'}^2$.

We end this part by proving the following fact which is useful in the next step:
\begin{align}\label{for:94}
\widehat{\eta_i}(\kappa\widehat{\mathfrak{n}})=\overline{e_{\eta'}}\ast\widehat{\eta_i}(\kappa\widehat{\mathfrak{n}})\qquad\text{ for each }i.
\end{align}
For any $f(\kappa\widehat{\mathfrak{n}})\in C_c(K\ltimes_\rho \widehat{V})$, we have
\begin{align*}
\bigl\langle \overline{e_{\eta'}}\ast\widehat{\eta_i},\,f\bigl\rangle_2&=\bigl\langle \Lambda(\overline{e_{\eta'}})\widehat{\eta_i},\,f\bigl\rangle_2\overset{\text{(1)}}{=}\bigl\langle
\widehat{\eta_i},\,\Lambda(e_{\eta'}^\vee)f\bigl\rangle_2\\
&=\bigl\langle
\widehat{\eta_i},\,e_{\eta'}^\vee\ast f\bigl\rangle_2\overset{\text{(2)}}{=}\bigl\langle
\eta_i,\,\mathcal{F}(e_{\eta'}^\vee\ast f)\bigl\rangle_2\\
&\overset{\text{(3)}}{=}\bigl\langle
\eta_i,\,e_{\eta'}^\vee\ast \mathcal{F}(f)\bigl\rangle_2=\bigl\langle
\eta_i,\,\Lambda(e_{\eta'}^\vee)\mathcal{F}(f)\bigl\rangle_2\\
&=\bigl\langle
\Lambda(\overline{e_{\eta'}})\eta_i,\,\mathcal{F}(f)\bigl\rangle_2=\bigl\langle
\overline{e_{\eta'}}\ast\eta_i,\,\mathcal{F}(f)\bigl\rangle_2\\
&\overset{\text{(4)}}{=}\bigl\langle
\eta_i,\,\mathcal{F}(f)\bigl\rangle_2\overset{\text{(5)}}{=}\bigl\langle \widehat{\eta_i},\,f\bigl\rangle_2,
\end{align*}
which proves \eqref{for:94} since the set of such $f$ is dense in $L^2(K\ltimes_\rho \widehat{V})$.

$(1)$ follows from \eqref{for:48}; $(2)$ and $(5)$ use Plancherel's theorem; $(4)$ just uses the fact that $\overline{e_{\eta'}}\ast\eta_i=\eta_i $ for each $i$. To show $(3)$ it suffices to show that $\mathcal{F}(e_{\eta'}^\vee\ast f)=e_{\eta'}^\vee\ast \mathcal{F}(f)$,  which follows from a simple computation by using Fubini's theorem and by noting that $e_{\eta'}$ is a function on $K$ while $\mathcal{F}$ is on factor $\widehat{V}$.

\smallskip

\noindent\emph{$^{\ast\ast}_\ast$ Reduction to the consideration of $K$-fixed vectors}. Now we define left $K$-invariant functions $\underline{\eta_i}$, $1\leq i\leq n$ by the formula
\begin{align*}
\underline{\eta_i}(\kappa\widehat{\mathfrak{n}})&=\sup_{u\in K} \bigl\lvert\widehat{\eta_i}(u\kappa\widehat{\mathfrak{n}})\bigl\rvert,&\quad &\forall \, (\kappa,\widehat{\mathfrak{n}})\in K\times \widehat{V}.
\end{align*}
Next we will show $\underline{\eta_i}\in L^2(K\ltimes_\rho \widehat{V})$, $1\leq i\leq n$. It follows from \eqref{for:94} that
\begin{align*}
\widehat{\eta_i}(u\kappa\widehat{\mathfrak{n}})=\overline{e_{\eta'}}\ast\widehat{\eta_i}(u\kappa\widehat{\mathfrak{n}}),\qquad \forall u\in K.
\end{align*}
Then by Proposition \ref{po:9}, one has
\begin{align*}
\Big\lvert\widehat{\eta_i}(u\kappa\widehat{\mathfrak{n}})\Big\rvert&\leq\norm{e_{\eta'}}_2\Big(\int_K\big\lvert\widehat{\eta_i}(\nu^{-1} u\kappa\widehat{\mathfrak{n}})\big\rvert^2 d\nu\Bigl)^{\frac{1}{2}}\\
&\leq d_{\eta'}^{\frac{1}{2}}\Big(\int_K\big\lvert\widehat{\eta_i}(\nu\kappa\widehat{\mathfrak{n}})\big\rvert^2 d\nu\Bigl)^{\frac{1}{2}},
\end{align*}
whence
\begin{align*}
\underline{\eta_i}(\kappa\widehat{\mathfrak{n}})&\leq d_{\eta'}^{\frac{1}{2}}\Big(\int_K\big\lvert\widehat{\eta_i}(\nu\kappa\widehat{\mathfrak{n}})\big\rvert^2 d\nu\Bigl)^{\frac{1}{2}}.
\end{align*}
Also, we see that
\begin{align}\label{for:45}
\bigg(\int_K\int_{\widehat{V}}\big\lvert\underline{\eta_i}(\kappa\widehat{\mathfrak{n}})
\big\rvert^2 d\widehat{\mathfrak{n}}d\kappa\bigg)^{\frac{1}{2}}&\leq d_{\eta'}^{\frac{1}{2}}\Big(\int_K\int_{\widehat{V}}\int_K\big\lvert\widehat{\eta_i}(\nu\kappa\widehat{\mathfrak{n}})\big\rvert^2 d\nu d\widehat{\mathfrak{n}}d\kappa\Bigl)^{\frac{1}{2}}\notag\\
&= d_{\eta'}^{\frac{1}{2}}\Big(\int_K\int_{\widehat{V}}\big\lvert\widehat{\eta_i}(\nu\widehat{\mathfrak{n}})\big\rvert^2 d\widehat{\mathfrak{n}}d\nu \Bigl)^{\frac{1}{2}}\notag\\
&\overset{\text{(1)}}{=}d_{\eta'}^{\frac{1}{2}}\Big(\int_K\int_{V}\big\lvert\eta_i(\nu\mathfrak{n})\big\rvert^2 d\mathfrak{n}d\nu \Bigl)^{\frac{1}{2}}\notag\\
&=d_{\eta'}^{\frac{1}{2}}\norm{\eta_i}_2.
\end{align}
Here $(1)$ follows from Plancherel's Theorem.

Set $\widetilde{\eta_i}=\mathcal{F}(\underline{\eta_i})$, $1\leq i\leq n$. It is clear that $\widetilde{\eta_i}$ are also left $K$-invariant functions. Also, using Plancherel's Theorem and \eqref{for:45}, we see that
\begin{align}\label{for:73}
\norm{\widetilde{\eta_i}}_2&=\norm{\mathcal{F}(\underline{\eta_i})}_2=\bigg(\int_K\int_{V}\big\lvert\mathcal{F}(\underline{\eta_i})
(\kappa\mathfrak{n})\big\rvert^2d\mathfrak{n}d\kappa\bigg)^{\frac{1}{2}}\notag\\
&=
\bigg(\int_K\int_{\widehat{V}}\big\lvert\underline{\eta_i}(\kappa\widehat{\mathfrak{n}})
\big\rvert^2 d\widehat{\mathfrak{n}}d\kappa\bigg)^{\frac{1}{2}}\leq d_{\eta'}^{\frac{1}{2}}\norm{\eta_i}_2.
\end{align}
Denote by $F_s=(a\omega)^{-1}\star h_s\cdot g_s$. For any $\theta,\,\vartheta\in L^2(K\ltimes_\rho V)$ we have
\begin{align*}
\bigl\langle \widehat{\Lambda}(F_s)\theta,\vartheta\bigl\rangle_2&=\int_{V}\widehat{F_s}(\mathfrak{n})\int_{K}\int_{V}\theta(\mathfrak{n}^{-1}\kappa v)\overline{\vartheta(\kappa v)}dvd\kappa d\mathfrak{n}.
\end{align*}
%For simplicity's sake, we use $\cdot$ and $\bullet$ to denote the $K$ actions on $V$ and $\widehat{V}$ respectively.
%Using the identification between $V$ and $\widehat{V}$ in Section \ref{sec:1}, we see that
%\begin{align*}
%\kappa\cdot \mathfrak{n}=\rho(\kappa)\mathfrak{n}\quad\text{ and }\quad\kappa\bullet \widehat{\mathfrak{n}}=\widehat{\rho(\kappa^{-1})^\tau(\mathfrak{n})}.
%\end{align*}
Recall notations in Section \ref{sec:1} one has
\begin{align*}
\theta(\mathfrak{n}^{-1}\kappa v)=\theta\bigl(\kappa,(\kappa^{-1}\cdot\mathfrak{n}^{-1})v\bigl).
\end{align*}
Here we add a comma between two variables to avoid confusion.   Denote by $f_{\kappa v}(\mathfrak{n})=\theta(\mathfrak{n}^{-1}\kappa v)$, then
\begin{align*}
\widehat{f}_{\kappa, v}(\widehat{\mathfrak{n}})=\widehat{\theta}\bigl(\kappa,(-\kappa^{-1}\cdot\widehat{\mathfrak{n}})\bigl)\cdot\overline{(\kappa^{-1}\cdot\widehat{\mathfrak{n}})(v)}.
\end{align*}
Let $\vartheta_m(\kappa v)=\vartheta(\kappa v)$ if $\abs{v}\leq m$, otherwise set $\vartheta_m(\kappa v)=0$, then
\begin{align}\label{for:87}
\lim_m \lVert \vartheta_m-\vartheta\lVert_2\rightarrow 0\qquad \text{ as }m\rightarrow\infty.
\end{align}
Using these notations and noting that  $F_s\in \mathcal{S}(\widehat{V})$ we have
\begin{align*}
\bigl\langle \widehat{\Lambda}(F_s)\theta,\vartheta\bigl\rangle_2&=\lim_m\bigl\langle \widehat{\Lambda}(F_s)\theta,\vartheta_m\bigl\rangle_2\\
&\overset{\text{(1)}}{=}\lim_m\int_{K}\int_{V}\overline{\vartheta_m(\kappa v)}\int_{V}\widehat{F_s}(\mathfrak{n})f_{\kappa v}(\mathfrak{n})d\mathfrak{n}dvd\kappa \notag\\
& \overset{\text{(2)}}{=}\lim_m\int_{K}\int_{V}\overline{\vartheta_m(\kappa v)}\int_{\widehat{V}}F_s(\widehat{\mathfrak{n}})\widehat{f}_{\kappa, v}(\widehat{\mathfrak{n}})d\widehat{\mathfrak{n}}dvd\kappa\notag\\
&=\lim_m\int_{K}\int_{V}\overline{\vartheta_m(\kappa v)}\int_{\widehat{V}}F_s(\widehat{\mathfrak{n}})
\widehat{\theta}\bigl(\kappa,-\kappa^{-1}\cdot\widehat{\mathfrak{n}}\bigl)\cdot\overline{(\kappa^{-1}\cdot\widehat{\mathfrak{n}})(v)}d\widehat{\mathfrak{n}} dvd\kappa\notag\\
&=\lim_m\int_{K}\int_{V}\overline{\vartheta_m(\kappa v)}\int_{\widehat{V}}F_s(\kappa\cdot\widehat{\mathfrak{n}})
\widehat{\theta}(\kappa,-\widehat{\mathfrak{n}})\cdot\overline{\widehat{\mathfrak{n}}(v)}d\widehat{\mathfrak{n}} dvd\kappa\notag\\
&\overset{\text{(3)}}{=}\lim_m\int_{\widehat{V}}\int_{K}F_s(\kappa\cdot\widehat{\mathfrak{n}})\widehat{\theta}(\kappa,-\widehat{\mathfrak{n}})\int_{V}
\overline{\vartheta_m(\kappa v)\cdot\widehat{\mathfrak{n}}(v)}dvd\kappa d\widehat{\mathfrak{n}}\notag\\
&=\lim_m\int_{K}\int_{\widehat{V}
}F_s(\kappa\cdot\widehat{\mathfrak{n}})\widehat{\theta}(\kappa,-\widehat{\mathfrak{n}})\overline{\widehat{\vartheta_m}
(\kappa,-\widehat{\mathfrak{n}})}d\widehat{\mathfrak{n}}d\kappa\\
&=\lim_m\int_{K}\int_{\widehat{V}
}F_s(-\kappa\cdot\widehat{\mathfrak{n}})\widehat{\theta}(\kappa\widehat{\mathfrak{n}})\overline{\widehat{\vartheta_m}
(\kappa\widehat{\mathfrak{n}})}d\widehat{\mathfrak{n}}d\kappa\\
&\overset{\text{(4)}}{=}\int_{K}\int_{\widehat{V}
}F_s(-\kappa\cdot\widehat{\mathfrak{n}})\widehat{\theta}(\kappa\widehat{\mathfrak{n}})\overline{\widehat{\vartheta}
(\kappa\widehat{\mathfrak{n}})}d\widehat{\mathfrak{n}}d\kappa.
%&=\biggl\langle ch_{\mathcal{C}}\int_{K
%}\widehat{\theta}\bigl(\kappa,-(\kappa^{-1}\bullet\widehat{\mathfrak{n}})\bigl)d\kappa, \int_{K
%}\widehat{\vartheta}\bigl(\kappa,-(\kappa^{-1}\bullet\widehat{\mathfrak{n}})\bigl)d\kappa\biggl\rangle_{\widehat{V}}.
\end{align*}
Since $\widehat{F_s},\,\vartheta_m\in L^1(V)\bigcap L^2(V) $ and $\vartheta_m$ has compact support, $(1)$ and $(3)$ hold by using Fubini's theorem; $(2)$ follows from the polarized form of Plancherel's theorem; from \eqref{for:87} and  Plancherel's theorem one has
\begin{align*}
\lim_m \lVert \widehat{\vartheta_m}-\widehat{\vartheta}\lVert_2\rightarrow 0\qquad \text{ as }m\rightarrow\infty,
\end{align*}
which implies $(4)$.

Consequently, for $\eta_i$, $1\leq i\leq n$ we also have
\begin{align}\label{for:46}
\bigl\langle \widehat{\Lambda}(F_s)\eta_i,\eta_i\bigl\rangle_2&=\int_{K
}\int_{\widehat{V}}\Big\lvert F^{\frac{1}{2}}_s(-\kappa\cdot\widehat{\mathfrak{n}})\widehat{\eta_i}(\kappa\widehat{\mathfrak{n}})\Big\rvert^2 d\widehat{\mathfrak{n}}d\kappa \notag\\
&\leq\int_{K
}\int_{\widehat{V}}\Big\lvert F^{\frac{1}{2}}_s(-\kappa\cdot\widehat{\mathfrak{n}})\underline{\eta_i}(\kappa\widehat{\mathfrak{n}})\Big\rvert^2  d\widehat{\mathfrak{n}}d\kappa\notag\\
&=\int_{K
}\int_{\widehat{V}}\Big\lvert F^{\frac{1}{2}}_s(-\kappa\cdot\widehat{\mathfrak{n}})\widehat{\mathcal{F}(\underline{\eta_i})}(\kappa\widehat{\mathfrak{n}})\Big\rvert^2 d\widehat{\mathfrak{n}}d\kappa\notag\\
&=\bigl\langle \widehat{\Lambda}(F_s)\mathcal{F}(\underline{\eta_i}),\mathcal{F}(\underline{\eta_i})\bigl\rangle_2\notag\\
&=\bigl\langle \widehat{\Lambda}(F_s)\widetilde{\eta_i},\widetilde{\eta_i}\bigl\rangle_2.
\end{align}
\noindent\emph{$^{\ast\ast}_{\ast\ast}$ The $K$-fixed vector case}. In this part, we will get estimates of $\bigl\lVert\widehat{\Lambda}(F_s^{\frac{1}{2}})\widetilde{\eta_i}\bigl\rVert_2$ for each $i$.
Note that
\begin{align*}
    \text{supp}\bigl(F_s^{\frac{1}{2}}\bigl)&\subset \text{supp}\bigl((a\omega)^{-1}\star h_s\bigl)\bigcap \text{supp}(g_s)\qquad\text{ and }\\
    \text{supp}\bigl((a\omega)\star g_s^{\frac{1}{2}}\cdot h_s^{\frac{1}{2}}\bigl)&\subset \text{supp}\bigl(a\omega\star g_s\bigl)\bigcap \text{supp}(h_s),
\end{align*}
then by using \eqref{for:12} and condition \eqref{for:76} we see that
\begin{align*}
    \text{supp}\bigl(F_s^{\frac{1}{2}}\bigl)&\subset \rho^\ast\bigl((a\omega)^{-1}\bigl)(X_s)\bigcap X_s\qquad\text{ and }\\
    \text{supp}\bigl((a\omega)\star g_s^{\frac{1}{2}}\cdot h_s^{\frac{1}{2}}\bigl)&\subset \rho^\ast(a\omega)(X_s)\bigcap X_s.
\end{align*}
Recall
notations in \ref{for:142} of Section \ref{sec:10}. Since $F\subset C_G(D)$ (see Lemma \ref{le:5}), we have
\begin{align*}
\rho^\ast(a\omega)(X_s)\bigcap X_s&=\Bigl\{\nu\in X_s:s^{-1}\leq \Bigl\lvert\rho^\ast(\omega^{-1})\bigl(\sum_{\psi\in\Phi_1}\psi(a^{-1})\pi_{\psi}(\nu)\bigl)\Bigl\rvert\leq s\Bigl\}
\end{align*}
and
\begin{align*}
\rho^\ast\bigl((a\omega)^{-1}\bigl)(X_s)\bigcap X_s&=\Bigl\{\nu\in X_s:s^{-1}\leq\Bigl\lvert\rho^\ast(\omega)\bigl(\sum_{\psi\in\Phi_1}\psi(a)\pi_{\psi}(\nu)\bigl)\Bigl\rvert\leq s\Bigl\}.
\end{align*}
Using the equivalent relations between norms (see \eqref{for:80}) and finiteness of $F$, we see that the set $\rho^\ast(a\omega)(X_s)\bigcap X_s$ and
$\rho^\ast\bigl((a\omega)^{-1}\bigl)(X_s)\bigcap X_s$ are contained inside
\begin{align*}
E_1&=\bigl\{\nu\in X_s:\max_{\psi\in\Phi_1}\bigl\lvert\psi(a^{-1})\pi_{\psi}(\nu)\bigl\rvert\leq CC_2s\bigl\}\qquad\text{and}\\
E_2&=\bigl\{\nu\in X_s:\max_{\psi\in\Phi_1}\bigl\lvert\psi(a)\pi_{\psi}(\nu)\bigl\rvert\leq CC_2s\bigl\}
\end{align*}
respectively. Here $C_2=\max_{x,\,x^{-1}\in F}\norm{\rho^\ast(x)}$ (\ref{for:84} of Section \ref{sec:10}).

Furthermore, we see that $E_1$ and $E_2$ are contained in  the strips
\begin{align*}
S_1&=\big\{\nu\in X_s:\bigl\lvert\pi_{\varrho_1}(\nu)\bigl\rvert\leq CC_2s\bigl\lvert\varrho_1(a)\bigl\rvert\big\}\quad\text{ and }\\
S_2&=\bigl\{\nu\in X_s:\bigl\lvert\pi_{\lambda_1}(\nu)\bigl\rvert\leq CC_2s\bigl\lvert\lambda_1(a^{-1})\bigl\rvert\bigl\}
\end{align*}
respectively. Since $\rho^\ast$ is excellent, $\ker(\rho^\ast)\bigcap G_s\subset Z(G)$. Then \ref{sec:17} of Section \ref{sec:10} implies that for $a\in D^+$ with large $\abs{a}$, $\abs{\varrho_1(a)}$ and $\abs{\lambda_1(a^{-1})}$ are small. Hence it follows that $\rho^\ast(a\omega)(X_s)\bigcap X_s$ and $\rho^\ast\bigl((a\omega)^{-1}\bigl)(X_s)\bigcap X_s$ will be contained in the cones (see \eqref{for:124})
\begin{align}\label{for:18}
\text{Cone}_2\bigl(CC_2s\abs{\varrho_1(a)},s^{-1}\bigl)\quad\text{ and }\quad \text{Cone}_1\bigl(CC_2s\abs{\lambda_1(a^{-1})},s^{-1}\bigl)
\end{align}
respectively.

By Proposition \ref{po:7}, if $\abs{a}$ is large enough, there
  is an open cover of $\text{Cone}_1\bigl(CC_2s\abs{\lambda_1(a^{-1})},s^{-1}\bigl)$:
  \begin{align}\label{for:72}
  \text{Cone}_1\bigl(CC_2s\abs{\lambda_1(a^{-1})},s^{-1}\bigl)\subset\bigcup_{1\leq j\leq \ell_1}\mathcal{E}_j
  \end{align}
  such that for each $\mathcal{E}_j$, there are at least $\bigl[C^{-1}\bigl(\abs{\lambda_1(a^{-1})}CC_2s^2\bigl)^{\mathfrak{q}}\bigl]$ different elements $\tau^j_l$ in $K$ such that
  \begin{align}\label{for:16}
 \rho^\ast(\tau^j_l)\mathcal{E}_j\bigcap\rho^\ast(\tau^j_\ell)\mathcal{E}_j=\emptyset,\qquad \text{ if }\ell\neq l.
  \end{align}
Choose a partition of unity subordinate to the open cover $\bigcup_{1\leq j\leq \ell_1}\mathcal{E}_j$ of supp$\bigl(F_s^{\frac{1}{2}}\bigl)$, then we can write
 \begin{align}\label{for:27}
    F_s^{\frac{1}{2}}=\sum_{1\leq j\leq \ell_1}w_j,\qquad\text{ where }w_j\in \mathcal{S}(V)
 \end{align}
with $\norm{w_j}_\infty\leq \norm{F_s^{\frac{1}{2}}}_\infty$ and  supp$(w_j)\subset \mathcal{E}_j$ for each $j$.

For any $\kappa\in K$ and $\widetilde{\eta_i}$, $1\leq i\leq n$ using \eqref{for:12} and the left $K$-invariant property of $\widetilde{\xi_i}$ we find that
\begin{align*}
\Lambda(\kappa)\widehat{\Lambda}(w_j)(\widetilde{\eta_i})&=\widehat{\Lambda}(\kappa\star w_j)\Lambda(\kappa)(\widetilde{\eta_i})
=\widehat{\Lambda}(\kappa\star w_j)\widetilde{\eta_i},
\end{align*}
and so
\begin{align}\label{for:71}
   \bigl\lVert\widehat{\Lambda}(w_j)(\widetilde{\eta_i})\bigl\rVert_2=\bigl\lVert\Lambda(\kappa)\widehat{\Lambda}(w_j)(\widetilde{\eta_i})\bigl\rVert_2
   =\bigl\lVert\widehat{\Lambda}(\kappa\star w_j)\widetilde{\eta_i}\bigl\rVert_2.
\end{align}
On the other hand, from \eqref{for:12} and \eqref{for:27} one has
\begin{align*}
    \text{supp}(\kappa\star w_j)\subset\rho(\kappa)^\ast \bigl(\text{supp}(w_j)\bigl)\subset \rho(\kappa)^\ast (\mathcal{E}_j),\qquad \forall \kappa\in K.
\end{align*}
In particular, for $\tau^j_l$ it follows from \eqref{for:16} that
\begin{align}\label{for:28}
    \text{supp}(\tau^j_l\star w_j)\bigcap\text{supp}(\tau^j_\ell\star w_j)=\emptyset,\qquad \text{ if }\ell\neq l,
\end{align}
which means for any $\theta,\,\vartheta\in L^2(K\ltimes_\rho V)$, if $\ell\neq l$ then
\begin{align}\label{for:22}
&\bigl\langle \widehat{\Lambda}(\tau^j_l\star w_j)\theta,\,\widehat{\Lambda}(\tau^j_\ell\star w_j)\vartheta\bigl\rangle_2\overset{\text{(1)}}{=}\bigl\langle \widehat{\Lambda}\bigl(\tau^j_l\star w_j\cdot \tau^j_\ell\star w_j\bigl)\theta,\,\vartheta\bigl\rangle_2=0.
\end{align}
Here $(1)$ follows from \eqref{for:24} and \eqref{for:6}. Set
\begin{align*}
y_j=\sum_{l} \tau^j_l\star w_j,\qquad \text{for any }1\leq j\leq\ell_1,
\end{align*}
then \eqref{for:27} and \eqref{for:28} imply that $\norm{y_j}_\infty\leq \norm{F_s^{\frac{1}{2}}}_\infty\leq 1$, for any $1\leq j\leq\ell_1$. Then  it follows from \eqref{for:43} that
\begin{align}\label{for:44}
\norm{\widehat{\Lambda}(y_j)}\leq 1,\qquad \text{for any }1\leq j\leq\ell_1.
\end{align}
Therefore, \eqref{for:71} and \eqref{for:22} mean that
\begin{align*}
\widehat{\Lambda}(y_j)(\widetilde{\eta_i})=\sum_{l} \widehat{\Lambda}(\tau^j_l\star w_j)(\widetilde{\eta_i})
\end{align*}
is an orthogonal decomposition of $\widehat{\Lambda}(y_j)(\widetilde{\eta_i})$ into vectors of equal length; combined with \eqref{for:44} it follows that
\begin{align}\label{for:25}
\bigl\lVert\widehat{\Lambda}( w_j)(\widetilde{\eta_i})\bigl\rVert_2^2&=\frac{\bigl\lVert\widehat{\Lambda}(y_j)(\widetilde{\eta_i})\bigl\lVert_2^2}{\sharp\{\tau^j_l\}}\leq
\frac{\norm{\widetilde{\eta_i}}_2^2}{\bigl[C^{-1}\bigl(\abs{\lambda_1(a^{-1})}CC_2s^2\bigl)^{\mathfrak{q}}\bigl]}
\end{align}
for any $1\leq j\leq\ell_1$.

\eqref{for:27} and \eqref{for:25} also imply that for any $1\leq i\leq n$,
\begin{align}\label{for:26}
\bigl\lVert\widehat{\Lambda}(F_s^{\frac{1}{2}})\widetilde{\eta_i}\bigl\rVert^2_2&\leq\sum_{1\leq j\leq \ell_1}l_1\bigl\lVert\widehat{\Lambda}( w_j)(\widetilde{\eta_i})\bigl\rVert^2_2\notag\\
&\leq C\ell_1^2\bigl(CC_2s^2\bigl)^{-\mathfrak{q}}\abs{\lambda_1(a)}^{\mathfrak{q}}\,\norm{\widetilde{\eta_i}}^2_2.
\end{align}
Using \eqref{for:73}, \eqref{for:46} and \eqref{for:26} we get
\begin{align*}
\bigl\langle \widehat{\Lambda}(F_s)\eta_i,\eta_i\bigl\rangle_2&\leq\bigl\langle \widehat{\Lambda}(F_s)\widetilde{\eta_i},\widetilde{\eta_i}\bigl\rangle_2\\
&=\bigl\langle \widehat{\Lambda}(F_s^{\frac{1}{2}})\widetilde{\eta_i},\,\widehat{\Lambda}(F_s^{\frac{1}{2}})\widetilde{\eta_i}\bigl\rangle_2
=\bigl\lVert\widehat{\Lambda}(F_s^{\frac{1}{2}})\widetilde{\eta_i}\bigl\rVert^2_2\notag\\
&\leq C\ell_1^2\bigl(CC_2s^2\bigl)^{-\mathfrak{q}}d_{\eta'}\abs{\lambda_1(a)}^{\mathfrak{q}}\,\norm{\eta_i}^2_2.
\end{align*}
Then one has
\begin{align*}
\sum_{i=1}^n \phi^{\Lambda}_i(e,e)&=\sum_{i=1}^n\bigl\langle \widehat{\Lambda}(F_s)\eta_i,\eta_i\bigl\rangle_2\\
&\leq C\ell_1^2\bigl(CC_2s^2\bigl)^{-\mathfrak{q}}d_{\eta'}\abs{\lambda_1(a)}^{\mathfrak{q}}\,
\bigl(\sum_{i=1}^n\norm{\eta_i}^2_2\bigl)\notag\\
&\leq C\ell_1^2(CC_2s^2)^{-\mathfrak{q}}d_{\eta'}\abs{\lambda_1(a)}^{\mathfrak{q}}\,
\norm{\eta'}^2.
\end{align*}
\noindent\emph{$^{\ast\ast}_{\ast\ast\ast}$ Upper bound of $\bigl\lvert \phi_{\widehat{\Pi}(f_s)\eta,\widehat{\Pi}(f_s)\xi}(k_1a\omega k_2)\bigl\rvert$}. In $^{\ast\ast}$ we have showed that $\Bigl\lVert \widehat{\Pi}\bigl((a\omega)^{-1}\star h_s^{\frac{1}{2}}\cdot g_s^{\frac{1}{2}}\bigl)\eta'\Bigl\rVert^2$ is the limit of sums of projected matrix coefficients $\sum_{i=1}^n \phi^{\Lambda}_i(e,e)$, then we finally get
\begin{align}\label{for:75}
\Bigl\lVert \widehat{\Pi}\bigl((a\omega)^{-1}\star h_s^{\frac{1}{2}}\cdot g_s^{\frac{1}{2}}\bigl)\eta\Bigl\rVert\leq C\ell_1(CC_2s^2)^{-\frac{\mathfrak{q}}{2}}d_{\eta'}^{\frac{1}{2}}\norm{\eta'}\,\bigl\lvert\lambda_1(a)\bigl\rvert^{\frac{\mathfrak{q}}{2}}.
\end{align}
On the other hand, for  the projected matrix coefficient
\begin{align*}
\psi^{\Pi}:(x,y)&\rightarrow \Bigl\langle \widehat{\Pi}\bigl((a\omega)\star g_s^{\frac{1}{2}}\cdot h_s^{\frac{1}{2}}\bigl)\Pi(x)\xi',\,\widehat{\Pi}\bigl((a\omega)\star g_s^{\frac{1}{2}}\cdot h_s^{\frac{1}{2}}\bigl)\Pi(y)\xi'\Bigl\rangle\\
&=\Bigl\langle \widehat{\Pi}\bigl((a\omega)\star g_s\cdot h_s\bigl)\Pi(x)\xi',\,\Pi(y)\xi'\Bigl\rangle
\end{align*}
and its approximation,  uniformly on compacta,  by sums of projected matrix coefficients
\begin{align*}
\psi^{\Lambda}_i: (x,y)\rightarrow \Bigl\langle \widehat{\Lambda}\bigl((a\omega)\star g_s\cdot h_s\bigl)\Lambda(x)\xi_i,\,\Lambda(y)\xi_i\Bigl\rangle_2,
\end{align*}
satisfying the condition that
$\sum_{i=1}^n\norm{\xi_i}_2^2\leq \norm{\xi'}^2$, since $\xi'$ ia also $K$-finite, analogous results in $^{\ast\ast}$ and $^{\ast\ast}_\ast$ are true for
$\psi^{\Pi}$, $\psi^{\Lambda}_i$, $\underline{\xi_i}$ and $\widetilde{\xi_i}$, where $\underline{\xi_i}$ and $\widetilde{\xi_i}$ are defined in a way analogous to
how  $\underline{\eta_i}$ and $\widetilde{\eta_i}$ are defined respectively. Let $H_s=(a\omega)\star g_s^{\frac{1}{2}}\cdot h_s^{\frac{1}{2}}$. Then arguments in $^{\ast\ast}_{\ast\ast}$ show that
\begin{align*}
\text{supp}\bigl(H_s^{\frac{1}{2}}\bigl)\subset \text{Cone}_2\bigl(CC_2s\abs{\varrho_1(a)},s^{-1}\bigl)\qquad\text{ if $\abs{a}$ is large enough};
\end{align*}
and Corollary \ref{cor:9} gives an open cover of $\text{Cone}_2\bigl(CC_2s\abs{\varrho_1(a)},s^{-1}\bigl)$:
  \begin{align*}
  \text{Cone}_2\bigl(CC_2s\abs{\varrho_1(a)},s^{-1}\bigl)\subset\bigcup_{1\leq j\leq \ell_2}\mathcal{U}_j
  \end{align*}
  such that for each $\mathcal{U}_j$, there are at least $\bigl[C^{-1}\bigl(\abs{\varrho_1(a)}CC_2s^2\bigl)^{\mathfrak{q}}\bigl]$ different elements $\upsilon^j_l$ in $K$ such that
  \begin{align*}
 \rho^\ast(\upsilon^j_l)\mathcal{U}_j\bigcap\rho^\ast(\upsilon^j_\ell)\mathcal{U}_j=\emptyset,\qquad \text{ if }\ell\neq l.
  \end{align*}
Hence we see that  analogous results in $^{\ast\ast}_{\ast\ast}$ are true for $\bigl\langle \widehat{\Lambda}(H_s)\xi_i,\xi_i\bigl\rangle_2$ and $\bigl\langle \widehat{\Lambda}(H_s)\widetilde{\xi_i},\widetilde{\xi_i}\bigl\rangle_2$. Then it follows that
\begin{align*}
\sum_{i=1}^n \psi^{\Lambda}_i(e,e)&=\sum_{i=1}^n\bigl\langle \widehat{\Lambda}(H_s)\xi_i,\xi_i\bigl\rangle_2\leq C\ell_2^2(CC_2s^2)^{-\mathfrak{q}}d_{\xi'}\abs{\varrho_1(a^{-1})}^{\mathfrak{q}}\,
\norm{\xi'}^2
\end{align*}
for any $n$ as well. Since $\psi^{\Pi}(e,e)=\Bigl\lVert \widehat{\Pi}\bigl((a\omega)\star g_s^{\frac{1}{2}}\cdot h_s^{\frac{1}{2}}\bigl)\xi\Bigl\rVert^2$ is approximated by $\sum_{i=1}^n \psi^{\Lambda}_i(e,e)$, then we finally get
\begin{align}\label{for:66}
&\Bigl\lVert \widehat{\Pi}\bigl((a\omega)\star g_s^{\frac{1}{2}}\cdot h_s^{\frac{1}{2}}\bigl)\xi'\Bigl\rVert\leq C\ell_2(CC_2s^2)^{-\frac{\mathfrak{q}}{2}}d_{\xi'}^{\frac{1}{2}}\norm{\xi'}\,\bigl\lvert\varrho_1(a^{-1})\bigl\rvert^{\frac{\mathfrak{q}}{2}}.
\end{align}
It follows from \eqref{for:19}, \eqref{for:75} and \eqref{for:66} that
\begin{align*}
&\Bigl\lvert \phi_{\widehat{\Pi}(f_s)\eta,\widehat{\Pi}(f_s)\xi}(k_1a\omega k_2)\Bigl\rvert\leq C\ell_1\ell_2(CC_2s^2)^{-\mathfrak{q}}d_{\eta'}^{\frac{1}{2}}d_{\xi'}^{\frac{1}{2}}\norm{\xi'}\,\norm{\eta'}\,
\bigl\lvert\lambda_1(a)\varrho_1(a^{-1})\bigl\rvert^{\frac{\mathfrak{q}}{2}}.
\end{align*}
From \eqref{for:82} we  see that $d_{\xi'}=d_{\xi}$, $d_{\eta'}=d_{\eta}$ and $\norm{\xi'}=\norm{\xi}$, $\norm{\eta'}=\norm{\eta}$; and use $\lambda_1=-\varrho$, $\varrho_1=-\lambda$,   then
\begin{align*}
&\Bigl\lvert \phi_{\widehat{\Pi}(f_s)\eta,\widehat{\Pi}(f_s)\xi}(k_1a\omega k_2)\Bigl\rvert\leq C\ell_1\ell_2(CC_2s^2)^{-\mathfrak{q}}d_{\eta}^{\frac{1}{2}}d_{\xi}^{\frac{1}{2}}\norm{\xi}\,\norm{\eta}\,
\bigl\lvert\lambda(a)\varrho(a^{-1})\bigl\rvert^{\frac{\mathfrak{q}}{2}}.
\end{align*}
Hence we finished the proof.
\end{proof}
In \eqref{for:8}, the upper bounds of the matrix coefficients depend on the choice of $K$-finite vectors. Our next task is to show the uniformness of the upper bounds.  \subsection{Uniform bound for $K$-finite matrix coefficients}
\begin{corollary}\label{cor:4}
We assume notations in Proposition \ref{po:6}. Then $\Pi$ is $\bigl(K,\,\Xi_G^{\frac{1}{m}}\bigl)$ bounded on $G$
where $m\geq\frac{p_{(G,V,\Phi_1)}}{2}$.
\end{corollary}
\begin{proof}
We assume notations in the proof of Proposition \ref{po:6}. At first, we will show the following: there exists $p>0$ such that $\Pi$ is strongly $L^{p+\epsilon}$. Let $\Delta$ denote the set of simple roots of $G$. Then the modular function can be written as $\delta_B(a)=\prod_{\omega\in \Delta}\abs{\omega(a)}^{m_\omega}$ (see \eqref{for:114}) for any $a\in D^0$. Proposition \ref{po:6} and Lemma \ref{le:2} show that getting an upper
bound for $p$ such that $\phi_{\eta,\xi}\in L^{q}(G)$ on a dense set of vectors $\xi$ and $\eta$ in $\mathcal{H}$ for any $q>p$ boils down to a matter of comparing the coefficients of each simple
root in $(\lambda\varrho^{-1})^{-\frac{\mathfrak{q}}{2}}$ with those in the modular function $\delta_B$. That is, set
$\Lambda(\Phi_1)=-\frac{\mathfrak{q}}{2}(\lambda-\varrho)$, then
\begin{align}\label{for:119}
p\geq p_{(G,V,\Phi_1)}=\max_{1\leq j\leq n}\bigg\{\frac{\text{the coefficient of $\omega_j$ in $\delta_B$}}{\text{the coefficient of $\omega_j$ in $\Lambda(\Phi_1)$}}\bigg\}
\end{align}
where $\{\omega_1,\cdots,\omega_n\}$ is the set of simple roots of $G$. From \ref{sec:17} of Section \ref{sec:10} we see that $p_{(G,V,\Phi_1)}<\infty$, which means the existence of such $p$.

By Theorem \ref{th:10}, we have a passage from a uniform $L^p$-bound to a uniform pointwise bound, that is, let $m$ be
any integer such that $2m\geq p_{(G,V,\Phi_1)}$.
\end{proof}
Now we are ready to prove our main result.
\subsection{Proof of Theorem \ref{th:11}}\label{sec:20}
For $\rho$ irreducible Theorem \ref{th:11} follows directly from Corollary \ref{cor:4}. Next, we consider general cases, i.e. $\rho$ is not irreducible. Since char $k=0$, then by full reducibility of semisimple groups there is a natural decomposition of $V$ under representation $\rho$: $V=\oplus_{i=1}^N V_i$
such that for each $i$, $V_i$ is an irreducible representation of $G$. Denote the sets of weights on each $V_i$ by $\Phi_i$ and denote by $\lambda_i$
and $\varrho_i$ the highest and lowest weight of $\rho$ respectively for each $\Phi_i$, compatible with ordering of the root system of $G$. As showed in the proof of Proposition \ref{po:6}, $\rho^\ast$ is also irreducible and excellent on each $V_i$, $1\leq i\leq N$; and the highest weight $\lambda'_i$ and lowest weight $\varrho'_i$ of $\rho^\ast$ on each $V_i$ are equal to $-\varrho_i$ and $-\lambda_i$ respectively by a good choice of basis of $V_i$.

Now suppose $\Pi$ is irreducible. Denote by $\mathcal{H}$ be the attached Hilbert space. For each $1\leq i\leq N$, let $\mathcal{H}_i=\{\nu\in\mathcal{H}: \nu\text{ is fixed by }V_i\}$. For any $v\in V_i$ and any $s\in G\ltimes_\rho V$, note that $s^{-1}vs\in V_i$, then
\begin{align}\label{for:33}
\Pi(v)(\Pi(s)\nu)&=\Pi(s)\Pi(s^{-1}vs)\nu=\Pi(s)\nu,\qquad \forall \nu\in\mathcal{H}_i,
\end{align}
which implies that $\mathcal{H}_i$ is $G\ltimes_\rho V$-invariant. Hence $\mathcal{H}_i=0$ or $\mathcal{H}_i=\mathcal{H}$. Since there is no non-trivial $V$-fixed vectors, there exists a nonempty subset $E$ of $\{1,\cdots, N\}$ such that $\mathcal{H}_i=0$ if $i\in E$. For any $i\in E$, $\Pi$ can be viewed as a representation of  $G\ltimes_\rho V_i$ on $\mathcal{H}$ without non-trivial $V_i$-fixed vectors. Let $m$ be any integer such that $m\geq\frac{p_{(G,V,\Phi)}}{2}$ (see Definition \ref{def:3}). Also, \ref{sec:17} of Section \ref{sec:10} implies that such $m$ exists. Then by conclusions for the case of $\rho$ irreducible, $\Pi$ is $\bigl(K,\,\Xi_G^{\frac{1}{m}}\bigl)$ bounded on $G$.

For $\Pi$ not irreducible, $\Pi$ is decomposed into a direct integral $\int_X\Pi_xd\mu(x)$ of irreducible unitary representations of $G\ltimes_\rho V$ for some measure space $(X,\mu)$ (we refer to
\cite[Chapter 2.3]{Zimmer} or \cite{Margulis} for more detailed account for the direct integral theory). If $\Pi$ has no non-trivial
$V$-fixed vectors then for almost all $x\in X$, $\Pi_x$ has no non-trivial $V$-fixed vectors. Then by conclusions for the case of $\Pi$ irreducible, we  see that  for almost all $x\in X$ $\Pi_x$ is $\bigl(K,\,\Xi_G^{\frac{1}{m}}\bigl)$ bounded on $G$. Since $G\ltimes_\rho V$ is
known to be of Type I (see \cite{Be} and \cite{warner}), we may assume that $\Pi_x$
is weakly contained in $\Pi$ (up to equivalence) for each $x\in X$. Hence Proposition \ref{th:2} implies that $\Pi$ is $\bigl(K,\,\Xi_G^{\frac{1}{m}}\bigl)$ bounded on $G$.

\subsection{Results for $G$ if each almost $k$-simple factor has Kazhdan's property $(T)$}\label{sec:25} In this section we will show that if each almost $k$-simple factor of $G$ has Kazhdan's property $(T)$, then property $(T)$ implies uniform upper bounds for $K$-finite matrix coefficients of $G\ltimes_\rho V$ on $G$.

The following lemma is an immediate consequence of Theorem \ref{th:11}. But we prefer a proof independent of Theorem \ref{th:11} since the lemma implies that some property of $G$ itself gives upper bounds of $K$-finite matrix coefficients on $G$.
\begin{lemma}\label{le:3}
Suppose $\rho$ is excellent. Then for any unitary representation $\pi$ of $G\ltimes_\rho V$ without non-trivial $V$-fixed vectors, there are no non-trivial $G_i$-fixed vectors either for each $i$, where $G_i$ are non-compact almost $k$-simple factors of $G$.
\end{lemma}
\begin{proof}
Through this identification  of $\widehat{V}$ with $V$ in Section \ref{sec:22}, the action of
$G$ on $\widehat{V}$ is equivalent to the action $\rho^\ast$ of $G$ in $V$ (see \eqref{for:36}). Since $\rho$ is normal, the actions $\rho^\ast$ of $G$ in $V$ are algebraic and hence the $G$-orbits on $V$ are locally
closed (see \cite{Zimmer}).  Since $\rho$ is excellent, $\rho^\ast$ is also excellent. Hence the stabilizer in $G_i$ of any non-zero element in $\widehat{V}$ is the $k$-rational points of a proper closed (in the Hausdorff topology) $k$-algebraic group of $\tilde{G_i}$ for $k$ non-archimedean or isomorphic to $\CC$; or a proper closed subgroup of $G_i$ with finitely many connected components for $k$ isomorphic to $\RR$.

Now we consider the restriction of $\pi$ on $G_i\ltimes_\rho V$. Assume that $\pi$ is irreducible on $G_i\ltimes_\rho V$. Applying Mackey's theory, we conclude that $\pi$ is induced from an irreducible unitary representation $\pi_\sigma$ of the stabilizer in $G_i\ltimes_\rho V$ of an
element, say $\chi$, of $\widehat{V}$ and if $\chi$ is trivial, then $\pi\mid_V$ contains the trivial representation (see
\cite{Mac} and \cite[Theorem 7.3.1]{Zimmer}). It then follows from the assumption that $\chi$ must be non-trivial. Then $\pi\mid_{G_i}$ is induced from $\pi_\sigma\mid_{S_\chi}$, where $S_\chi$ is the stabilizer of $\chi$ in $G_i$. If $\pi\mid_{G_i}$ contains the trivial representation, then the space $G_i/S_\chi$ has a finite invariant regular Borel measure \cite[Theorem E.3.1]{Valette}. Hence by Borel density theorem, $S_\chi$ is Zariski dense in $\tilde{G_i}$ for $k$ non-archimedean or isomorphic to $\CC$ \cite{wang}; or $S_\chi$ is a lattice in $G_i$ for $k$ isomorphic to $\RR$ \cite[Theorem 3.2.5]{Zimmer}. It then follows from the assumption about $S_\chi$ that $\pi\mid_{G_i}$ has no non-zero invariant vector. In general, in the direct integral decomposition $\pi=\int_X\pi_xd\mu(x)$ where $\pi_x$ is irreducible, for almost
all $x\in X$, $\pi_x\mid_{V}$ has no non-zero invariant vector. Hence $\pi_x\mid_{G_i}$ has no non-zero invariant vector for almost all $x\in X$, by the above argument,
which means that $\pi\mid_{G_i}$ has no non-zero invariant vector.
\end{proof}
If $G_i$ has Kazhdan's property $(T)$, Lemma \ref{le:3} shows that $\pi\mid_{G_i}$ has spectral gap, that is $\pi\mid_{G_i}$ is outside a fixed neighborhood of the trivial
representation of $G_i$. If each $G_i$ has Kazhdan's property $(T)$, using the argument from \cite[$\S$3.1]{cow}, one can show
that $\pi$ is strongly $L^{p}(G_s')$ for some $p$ where $G_s'$ is the almost direct product of all $G_i$. Since $G_s'$ has finite index in $G_s$ (see \cite[Ch I, Proposition 2.3.4]{Margulis}) and $G_c$ is compact,
$\pi$ is also strongly $L^{p}(G)$.

By Theorem \ref{th:10}, we have
a passage from a uniform $L^p$-bound to a uniform pointwise bound. Then we have the following:
\begin{proposition}\label{po:3}
Suppose each non-compact almost $k$-simple factor $G_i$ of G has Kazhdan's property $(T)$ and $\rho$ is excellent. Then any unitary representation $\pi$ of $G\ltimes_\rho V$ without non-trivial $V$-fixed vectors is $\bigl(K,\,\Xi_G^{1/m}\bigl)$ bounded on $G$. Here $m$ is the integer such that $2m\geq p(G_0)$ (see Section \ref{sec:19}).
\end{proposition}
\begin{remark} For $G$ semisimple, a strategy of Howe (see \cite[Proposition 6.3]{howe}) shows that the upper bounds of matrix coefficients on each almost $k$-simple factor of $G$ provide a uniform upper bound of matrix coefficients on $G$. Hence one can without loss of generality assume
that $G$ is almost $k$-simple. The exact number $p(G_0)$ is obtained for classical simple Lie groups by combining the known cases of a classification of the
unitary dual by Vogan and Barbasch, and the results of Li (see \cite{Li} for references). The
precise values of $p(G_0)$ are not known in general but upper bounds have been given in
many cases (see \cite{oh}, \cite{howe} and \cite{Li}). If rank$_k(G)\geq 2$ we remark that even in the case when the number $p(G_0)$ is precisely known, higher rank trick (or Howe's trick) provides a much sharper pointwise bound  in general (see \cite[Theorem 5.7]{oh}).
\end{remark}

\section{improvement of uniform pointwise bound for matrix coefficients}\label{sec:19}
As before, let $K$ be a good maximal compact subgroup of $G$ and let $G_i$ be the non-compact almost $k$-simple factors of $G$. Denote by $\widehat{(G\ltimes_\rho V)}_0$ (resp. $\widehat{G_0}$) the set of equivalence classes of irreducible unitary representations of $G\ltimes_\rho V$ (resp. $G$) without $V$-fixed vectors (resp. without $G_i$-fixed vectors for each $i$). We denote by $p\bigl((G\ltimes_\rho V)_0\bigl)$ (resp. $p(G_0)$) the
smallest real number such that for any non-trivial $\pi\in \widehat{(G\ltimes_\rho V)}_0$ (resp. $\pi\in \widehat{G_0}$), $\pi$ is strongly $L^q$ on $G$ for any
$q>p\bigl((G\ltimes_\rho V)_0\bigl)$ (resp. $q>p(G_0)$).
%Similarly we denote by $p_K\bigl((G\ltimes_\rho V)_0\bigl)$ (resp. $p_K(G_0)$) the smallest real number such
%that for any non-trivial $\pi\in \widehat{(G\ltimes_\rho V)}_0$ (resp. $\pi\in \widehat{G_0}$), the $K$-finite matrix coefficients of $\pi$ are $L^q$-integrable
%for any $q>p_K\bigl((G\ltimes_\rho V)_0\bigl)$ (resp. $q>p_K(G_0)$). By Peter-Weyl theorem, we have $p\bigl((G\ltimes_\rho V)_0\bigl)\leq p_K\bigl((G\ltimes_\rho V)_0\bigl)$ (resp. $p(G_0)\leq p_K(G_0)$).
 Note that for $\pi\in \widehat{(G\ltimes_\rho V)}_0$, even though $\pi$ is irreducible on $G\ltimes_\rho V$, it is not necessarily irreducible on $G$ in general. From Lemma \ref{le:7} and \ref{le:3} we see that $p\bigl((G\ltimes_\rho V)_0\bigl)\leq 2m$ where $m$ is the smallest integer such that $m\geq \frac{1}{2}\cdot p(G_0)$.

Cowling showed that $p(G_0)<\infty$ if and only if each $G_i$ has Kazhdan's property $(T)$ \cite{cow}. Indeed, the explicit determination of $p(G_0)$ may be viewed as a
quantitative version of Kazhdan's property $(T)$. The method used in proving Theorem \ref{th:11} yields an upper bound of $p\bigl((G\ltimes_\rho V)_0\bigl)$ even though $p(G_0)=\infty$.

In this section we will show how to make an improvement of the estimates of $p\bigl((G\ltimes_\rho V)_0\bigl)$ obtained from Theorem \ref{th:11} for different examples. From the proofs in Section \ref{sec:11} we see that the bounds  are determined in the following way (without loss of much generality we assume that $\rho$ is irreducible):
\begin{enumerate}
  \item we consider  the  deformations of the set $X_s$ (see \eqref{for:17}) under $\rho^\ast(a)$ and $\rho^\ast(a^{-1})$ actions where $a\in D^+$, which are measured by $f_1(a)$ or $f_2(a)$ (see \eqref{for:18}) respectively where $0< f_1,\,f_2< 1$ are functions on $D^+$;
  \item for $c_0=f_1(a)$ or $f_2(a)$, choose an open cover of the set $\mathcal{W}_{c_0}$: $\mathcal{W}_{c_0}\subset\bigcup_{i} \mathcal{W}_{c_0}(i)$ (see \ref{for:105} of Section \ref{sec:10}); and find as many as possible $\tau^i_j\in K$ such that the $\rho^\ast(\tau^i_j)$-orbits of $\mathcal{W}_{c_0}(i)$ are pairwise disjoint. The number of these $\tau^i_j$, say $c_0^{\gamma}$ with $\gamma<0$, measures the matrix coefficients decay of a dense set of  $K$-finite vectors (see \eqref{for:25});
  \item comparing the increasing rate of each simple root in
  $(f_1f_2)^{-\gamma/2}$ with those in the modular function $\delta_B$ gives an upper bound $p$ of $p\bigl((G\ltimes_\rho V)_0\bigl)$ (see \eqref{for:119}); and then any unitary representation of $G\ltimes_\rho V$ without non-trivial $V$-fixed vectors is $\bigl(K,\,\Xi_G^{1/m}\bigl)$ bounded on $G$ where $m$ is an integer such that $2m\geq p$ (see Section \ref{sec:20}).
\end{enumerate}
\begin{remark}
It is clear that the above arguments apply to the case of a local field $k$ with char$\,k\neq 0$.
\end{remark}
Next, we will show how to increase the number of
disjoint $\rho(K)$ orbits (it is not difficult to see  $\rho(K)$ and  $\rho^\ast(K)$ have the same number of disjoint orbits) for some examples, which results better estimates for $p\bigl((G\ltimes_\rho V)_0\bigl)$.

For a subset $V'\subset V$, $V'$ is said to have $c$-disjoint property if there is an open cover of $V'$: $V'\subset\bigcup_{i=1}^l V_i'$ ($l$ is not dependent on $c$) such that for each $V_i'$ there are at least $C^{-1}c^{-1}$ elements in $\tau_j\in K$ such that the $\rho(\tau_j)$-orbits of $V_i'$ are pairwise disjoint.
\begin{sect15}\label{for:21}
Let $\rho$ be the standard representation of $G=SL(3,L)$ on $L^{3} $ where $L$ denotes a local field $k$ or $\HH$. For $L=\HH$, it is a $\RR$-rank $2$ group. We can realize $D$ as $D=\{a=\text{diag}(a_1,a_2,a_3):a_i\in k^\ast\}$ and the simple roots are
$\alpha_i(a)=a_ia^{-1}_{i+1}$, $i=1,2$; the weights are $w_i(a)=a_i$, $1\leq i\leq 3$ and $\lambda=w_1=\frac{1}{3}(2\alpha_1+\alpha_2)$ and $\varrho=w_3=-\frac{1}{3}(\alpha_1+2\alpha_2)$.
By Theorem \ref{th:11},
\begin{enumerate}
\item for $L=\CC$, $\gamma=-\frac{1}{3}\alpha_1$ and $-\frac{\gamma}{2}(\lambda-\varrho)=\frac{1}{6}(\alpha_1+\alpha_2)$. Using  $\delta_B(a)=\abs{\alpha_1(a)\alpha_2(a)}^4$ we have $p\bigl((G\ltimes_\rho V)_0\bigl)\leq 24$;
  \item for $L=\HH$, $\gamma=-(\frac{1}{3})^2\alpha_1$ and $-\frac{\gamma}{2}(\lambda-\varrho)=\frac{1}{18}(\alpha_1+\alpha_2)$. Using  $\delta_B(a)=\abs{\alpha_1(a)\alpha_2(a)}^8$ we have $p\bigl((G\ltimes_\rho V)_0\bigl)\leq 144$;
      \item for $L=\RR$ or non-archimedean, $\gamma=-\frac{1}{3}\alpha_1$ and $-\frac{\gamma}{2}(\lambda-\varrho)=\frac{1}{6}(\alpha_1+\alpha_2)$. Using  $\delta_B(a)=\abs{\alpha_1(a)\alpha_2(a)}^2$ we have $p\bigl((G\ltimes_\rho V)_0\bigl)\leq 12$.
\end{enumerate}
\textbf{Improvement}: let $K$ be $SO(3)$, $SU(3)$, $Sp(3)$ or $SL_3(\mathcal{O})$ for $L=\RR$, $\CC$, $\HH$ or non-archimedean correspondingly. Here $\mathcal{O}$ is the ring of integers of $k$. Set
\begin{align*}
\mathcal{W}_{c_0}=\bigl\{v=(v_1,v_2,v_3)\in L^3:\abs{v_1}\leq c_0\text{ and }\frac{3}{4}\leq\abs{v}\leq \frac{5}{4}\bigl\}.
\end{align*}
Then $\mathcal{W}_{c_0}\subset\bigcup_{i=2}^3\mathcal{W}_{c_0}^{i}$, where
\begin{align*}
\mathcal{W}_{c_0}^{i}=\bigl\{v\in L^3:\abs{v_1}<\frac{3}{2}c_0,\,\frac{1}{2}<\abs{v}<\frac{3}{2}\text{ and }\abs{v_i}>\frac{1}{4}\bigl\}.
\end{align*}
Ar first, we consider the case of $L$ archimedean or isomorphic to $\HH$. For $\mathcal{W}_{c_0}^{2}$ consider $h(a,z)=\left( \begin{array}{ccc}
a & z & 0\\
-\bar{z} & a &0 \\
0 & 0 &1 \\
\end{array} \right)\in K$, where $a\in\RR$, $z\in L$ with $a^2+\abs{z}^2=1$. For any $a_1,\,a_2\in\RR$ and $z_1,\,z_2\in L$ we have
\begin{align}\label{for:60}
h(a_1,z_1)^{-1}h(a_2,z_2)=\left( \begin{array}{ccc}
a_1a_2+z_1\overline{z_2} & a_1z_2-a_2z_1 & 0\\
-a_1\overline{z_2}+a_2\overline{z_1} & a_1a_2+\overline{z_1}z_2 &0 \\
0 & 0 &1 \\
\end{array} \right).
\end{align}
Let $A=\{(a,z)\in\RR\times L:\frac{1}{4}\leq a,\abs{z}\leq\frac{\sqrt{15}}{4}\text{ and }a^2+\abs{z}^2=1\}$. Next, we will show:
\begin{enumerate}
  \item [(*)]if $(a_1,z_1),\,(a_2,z_2)\in A$ satisfying $\abs{z_1-z_2}\geq 15^2\cdot 64\cdot 8c_0$, then $\abs{a_1z_2-a_2z_1\big}\geq 16c_0$.
\end{enumerate}
If $\bigl\lvert\lvert z_1\rvert -\lvert z_2\rvert\bigr\lvert\geq 256c_0$, let $f(x)=x(1-x^2)^{-1/2}$, then there exists some $z_0$ between $\abs{z_1}$ and
$\abs{z_2}$ such that
\begin{align*}
f(\abs{z_1})-f(\abs{z_2})=f'(z_0)(\abs{z_1}-\abs{z_2})=(\abs{z_1}-\abs{z_2})(1-z_0^2)^{-3/2},
\end{align*}
the it follows that
\begin{align*}
\abs{a_1z_2-a_2z_1}&\geq a_1a_2\bigl\lvert\lvert z_1a_1^{-1}\rvert -\lvert z_2a_2^{-1}\rvert \bigr\rvert\\
&= a_1a_2\bigl\lvert f(\abs{z_1})-f(\abs{z_2})\bigr\rvert\\
&=a_1a_2\bigl\lvert\lvert z_1\rvert -\lvert z_2\rvert\bigr\lvert\cdot(1-z_0^2)^{-3/2}\\
&\geq \frac{1}{4}\cdot\frac{1}{4}\cdot(\frac{16}{15})^{3/2}\bigl\lvert\lvert z_1\rvert -\lvert z_2\rvert\bigr\lvert\\
&\geq \frac{1}{4}\cdot\frac{1}{4}\cdot 256c_0=16c_0.
\end{align*}
If $\bigl\lvert\lvert z_1\rvert -\lvert z_2\rvert\bigr\lvert< 256c_0$, let $g(x)=(1-x^2)^{-1/2}$ then there exists some $u$ between $\abs{z_1}$ and
$\abs{z_2}$ such that
\begin{align*}
\bigl\lvert g(\abs{z_1})-g(\abs{z_2})\bigr\lvert&=\bigl\lvert g'(u)(\abs{z_1}-\abs{z_2})\bigr\lvert\\
&=\bigl(\bigl\lvert\lvert z_1\rvert -\lvert z_2\rvert\bigr\lvert\cdot u\bigl)(1-u^2)^{-3/2}\\
&\leq (16\sqrt{15})256c_0.
\end{align*}
By the
above discussion we get:
\begin{align*}
\abs{a_1z_2-a_2z_1}&=a_1a_2\bigl\lvert z_1a_1^{-1}-z_2a_2^{-1} \bigr\rvert\\
&=a_1a_2\bigl\lvert (z_1-z_2)a_1^{-1}+z_2(a_1^{-1}-a_2^{-1}) \bigr\rvert\\
&\geq a_2\abs{z_1-z_2}-a_1a_2\abs{z_2}\cdot\bigl\lvert g(\abs{z_1})-g(\abs{z_2})\bigr\rvert\\
&\geq \frac{1}{4}\abs{z_1-z_2}-\frac{15}{16}\cdot\frac{\sqrt{15}}{4}\cdot(16\sqrt{15})256c_0\\
&\geq \frac{1}{4}\cdot 15^2\cdot 64\cdot 8c_0-15^2\cdot 64c_0> 16c_0.
\end{align*}
Hence we proved (*).

For any $(a_1,z_1),\,(a_2,z_2)\in A$ satisfying the condition in (*), and any $v\in \mathcal{W}_{c_0}^{2}$ using \eqref{for:60} and (*),  we have
\begin{align*}
\bigl\lvert\pi_1\big(h(a_1,z_1)^{-1}h(a_2,z_2)v\big)\bigl\rvert&=\bigl\lvert(a_1z_2-a_2z_1)v_2+(a_1a_2+z_1\overline{z_2})v_1\bigl\rvert\\
&\geq \bigl\lvert(a_1z_2-a_2z_1)v_2\bigl\rvert-\bigl\lvert(a_1a_2+z_1\overline{z_2})v_1\bigl\rvert\\
&\overset{\text{(1)}}{>} \frac{1}{4}\cdot 16c_0-\frac{3}{2}c_0=\frac{5}{2}c_0
\end{align*}
where $\pi_1$ projects $v=(v_1,v_2,v_3)\in L^3$ to $v_1$. $(1)$ follows from Cauchy-Schwarz inequality
$\bigl\lvert a_1a_2+z_1\overline{z_2}\bigl\rvert^2\leq (a_1^2+\abs{z_1}^2)\cdot(a_2^2+\abs{z_2}^2)=1$. Then we deduce that
\begin{align}\label{for:99}
h(a_1,z_1)(\mathcal{W}_{c_0}^{2})\bigcap h(a_2,z_2)(\mathcal{W}_{c_0}^{2})=\emptyset.
\end{align}
Let
\begin{align*}
z_{m}&=\frac{1}{4}+15^2\cdot 64\cdot 8mc_0\quad\text{ and }\quad a_{m}=(\sqrt{1-\abs{z_{m}}^2},z_{m})
\end{align*}
for $L=\mathbb{R}$; or let
\begin{align*}
z_{m,n}&=\frac{1}{4\sqrt{2}}+15^2\cdot 64\cdot 8mc_0+\Bigl(\frac{1}{4\sqrt{2}}+15^2\cdot 64\cdot 8nc_0\Bigl)\textrm{i}\qquad\text{and}\\
a_{m,n}&=(\sqrt{1-\abs{z_{m,n}}^2},z_{m,n})
\end{align*}
for $L=\mathbb{C}$; or let
\begin{align*}
z_{m,n,\ell,r}&=\frac{1}{8}+15^2\cdot 64\cdot 8mc_0+\Bigl(\frac{1}{8}+15^2\cdot 64\cdot 8nc_0\Bigl)\textrm{i}\\
&+\Bigl(\frac{1}{8}+15^2\cdot 64\cdot 8\ell c_0\Bigl)\textrm{j}+\Bigl(\frac{1}{8}+15^2\cdot 64\cdot 8rc_0\Bigl)\textrm{k}\qquad\text{and}\\
a_{m,n,\ell,r}&=(\sqrt{1-\abs{z_{m,n}}^2},z_{m,n,\ell,r})
\end{align*}
for $L=\mathbb{H}$, where $m$, $n$, $\ell$, $r$ are integers satisfying $0\leq m,n,\ell,r\leq \bigl[\frac{\sqrt{15}-1}{8\cdot 15^2\cdot 64\cdot 8c_0}\bigl]$.

It is clear that $a_{m}\in A$ (resp. $a_{m,n}\in A$ or $a_{m,n,\ell,r}\in A$) and the number of such $a_{m}$ (resp. $a_{m,n}$ or $a_{m,n,\ell,r}$) is
$\bigl[\frac{\sqrt{15}-1}{8\cdot 15^2\cdot 16\cdot 8c_0}\bigl]+1$ (resp. $(\bigl[\frac{\sqrt{15}-1}{8\cdot 15^2\cdot 16\cdot 8c_0}\bigl]+1)^2$ or $(\bigl[\frac{\sqrt{15}-1}{8\cdot 15^2\cdot 16\cdot 8c_0}\bigl]+1)^4$) for $L=\RR$ (resp. $L=\CC$ or $L=\HH$). Furthermore, for any $m\neq m_1$ (resp. $(m,n)\neq (m_1,n_1)$ or $(m,n,\ell,r)\neq (m_1,n_1,\ell_1,r_1)$) $h(a_{m})$ and $h(a_{m_1})$ (resp. $h(a_{m,n})$ and $h(a_{m_1,n_1})$ or $h(a_{m,n,\ell,r})$ and $h(a_{m_1,n_1,\ell_1,r_1})$) satisfy the condition \eqref{for:99}.

Hence $\mathcal{W}_{c_0}^{2}$ has $c_0$ (resp. $c_0^2$ or $c_0^4$)-disjoint property for $L=\mathbb{R}$ (resp. $L=\mathbb{C}$ or $L=\mathbb{H}$). For $\mathcal{W}_{c_0}^{3}$, we consider $h(a,z)=\left( \begin{array}{ccc}
a & 0 &z \\
0 & 1 & 0 \\
-\bar{z} & 0 & a \\
\end{array} \right)\in K$, where $a\in\RR$, $z\in L$ with $a^2+\abs{z}^2=1$. An argument analogous to the case of $\mathcal{W}_{c_0}^{2}$ shows that  $\mathcal{W}_{c_0}^{3}$ also has $c_0$ (resp. $c_0^2$ or $c_0^4$)-disjoint property, which implies that $\mathcal{W}_{c_0}$
has $c_0$ (resp. $c_0^2$ or $c_0^4$)-disjoint property for $L=\mathbb{R}$ (resp. $L=\mathbb{C}$ or $L=\mathbb{H}$).

Now we consider the case of $L$ non-archimedean. For $\mathcal{W}_{c_0}^{2}$, we consider $h(z)=\left( \begin{array}{ccc}
1 & z &0 \\
0 & 1 & 0 \\
0 & 0 & 1 \\
\end{array} \right)\in K$, where $a\in \mathcal{O}$. If $\abs{z_1-z_2}\geq 12c_0$ then
\begin{align*}
\bigl\lvert\pi_1\big(h(z_1)^{-1}h(z_2)v\big)\bigl\rvert&=\bigl\lvert v_1+(z_2-z_1)v_2\bigl\rvert= \bigl\lvert (z_2-z_1)v_2\bigl\rvert\geq 2c_0.
\end{align*}
Using the analogous construction to the one in proof of Lemma \ref{le:10} for non-archimedean case, we see that $\mathcal{W}_{c_0}^{2}$ has $c_0$-disjoint property. For $\mathcal{W}_{c_0}^{2}$, we consider $h(z)=\left( \begin{array}{ccc}
1 & 0 &z \\
0 & 1 & 0 \\
0 & 0 & 1 \\
\end{array} \right)\in K$ instead. It is clear that  $\mathcal{W}_{c_0}^{3}$ also has $c_0$-disjoint property, which implies that $\mathcal{W}_{c_0}$
has $c_0$-disjoint property for $L$ non-archimedean. Hence,
\begin{itemize}
\item for $G=SL(3,L)$ where $L=\RR$ or non-archimedean, $\gamma=-1$ and
$-\frac{\gamma}{2}(\lambda-\varrho)=\frac{1}{2}(\alpha_1+\alpha_2)$. Then $p\bigl((G\ltimes_\rho V)_0\bigl)\leq 4$;
  \item for $G=SL(3,\CC)$ $\gamma=-2$ and
$-\frac{\gamma}{2}(\lambda-\varrho)=\alpha_1+\alpha_2$.
Then $p\bigl((G\ltimes_\rho V)_0\bigl)\leq 4$;
  \item for $G=SL(3,\mathbb{H})$ $\gamma=-4$ and
$-\frac{\gamma}{2}(\lambda-\varrho)=2(\alpha_1+\alpha_2)$.
Then $p\bigl((G\ltimes_\rho V)_0\bigl)\leq4$.
\end{itemize}
\end{sect15}
\begin{sect15}\label{for:67}
Let $\rho$ be the adjoint representation of $G=SL(2,k)$ and let $h=\left( \begin{array}{ccc}
1 & 0 \\
0 & -1\\
\end{array} \right)$, $U=\left( \begin{array}{ccc}
0 & 1\\
0 & 0\\
\end{array} \right)$ and $V=\left( \begin{array}{ccc}
0 & 0 \\
1 & 0\\
\end{array} \right)$ in the Lie algebra $\mathfrak{sl}(2,k)$.  We can realize $D$ as $D=\{a=\text{diag}(a_1,a_1^{-1}):a_1\in k^\ast\}$ and simple root is
$\alpha_1(a)=a_1^2$; the weights are $\lambda(a)=\alpha_1$ and $\varrho=-\alpha_1$. By Theorem \ref{th:11}, $\gamma=-\frac{1}{3}$ and $-\frac{\gamma}{2}(\lambda-\varrho)=\frac{1}{3}\alpha_1$. Using  $\delta_B(a)=\abs{\alpha_1(a)}$ for $k\neq \CC$ and $\delta_B(a)=\abs{\alpha_1(a)}^2$ for $k=\CC$, we have $p\bigl((G\ltimes_\rho V)_0\bigl)\leq 3$ for $k\neq \CC$ and $p\bigl((G\ltimes_\rho V)_0\bigl)\leq 6$ for $k= \CC$.\\
\textbf{Improvement}: set
\begin{align*}
\mathcal{W}_{c_0}=\{(v_{U},v_h,v_{V}):\abs{v_{U}}\leq c_0\text{ and
}\frac{3}{4}\leq\abs{v}\leq\frac{5}{4}\}.
\end{align*}
We have a decomposition of $\mathcal{W}_{c_0}$: $\mathcal{W}_{c_0}=\bigcup_{i=1}^2\mathcal{W}_{c_0}^{i}$ where
\begin{align*}
\mathcal{W}_{c_0}^{1}&=\{v\in \mathfrak{sl}(2,k):\abs{v_{U}}<\frac{4}{3}c_0,\,\,\frac{1}{2}<\abs{v}<\frac{3}{2}\text{ and
}\frac{4}{3}\abs{v_{h}}>\abs{v_{V}}\}\qquad\text{ and }\\
\mathcal{W}_{c_0}^{2}&=\{v\in
\mathfrak{sl}(2,k):\abs{v_{U}}< \frac{4}{3}c_0,\,\,\frac{1}{2}<\abs{v}<\frac{3}{2}\text{ and
}\abs{v_{h}}< \abs{v_{V}}\}.
\end{align*}
Let $X=\{\theta\in \RR:0\leq \theta\leq \frac{\pi}{4}\}$. For $k$ archimedean consider $z(\theta)=\left( \begin{array}{ccc}
\cos\theta & -\sin\theta \\
\sin\theta & \cos\theta\\
\end{array} \right)\in K$ with $\theta\in X$. For any $\theta_1,\,\theta_2\in\RR$
$z(\theta_1)^{-1}z(\theta_2)=z(\theta_2-\theta_1)$. We also have
\begin{align}\label{for:61}
\pi_{U}\bigl(\text{Ad}_{z(\theta)}(v_{U}+v_h+v_{V})\bigl)=\cos^2\theta \cdot v_{U}+\sin 2\theta \cdot v_h-\sin^2\theta \cdot v_{V}.
\end{align}
For any $v\in \mathcal{W}_{c_0}^{1}$ and $\theta_1,\,\theta_2\in X$ using \eqref{for:61} we have
\begin{align}\label{for:23}
&\abs{\pi_{U}(\text{Ad}_{z(\theta_1)^{-1}z(\theta_2)}\cdot v)}\notag\\
&\geq
\abs{\sin(2\theta_2-2\theta_1) v_h-\sin^2(\theta_2-\theta_1)  v_{V}}-\abs{v_{U}}\notag\\
&\overset{\text{(1)}}{\geq}  \frac{1}{2}\cdot \frac{3}{5}\delta\abs{\theta_2-\theta_1}-2\abs{\theta_2-\theta_1}^2-c_0.
\end{align}
$(1)$ holds since there exists $\delta>0$ such that $\delta\abs{\theta}\leq\abs{\sin\theta}\leq \abs{\theta}$ for any $-\frac{\pi}{4}\leq\theta\leq \frac{\pi}{4}$. Hence $\mathcal{W}_{c_0}^{1}$  has $c_0$-disjoint property (see Case \eqref{for:29} and
\eqref{for:9} in the proof of Lemma \ref{le:10}).

We have a decomposition of $\mathcal{W}_{c_0}^{2}$:
$\mathcal{W}_{c_0}^{2}=A_1\bigcup A_2$ where
\begin{align*}
A_1=\{v\in\mathcal{W}_{c_0}^{2}: \text{Re}\Bigl(\frac{v_h}{v_{V}}\Bigl)>-c_0\}\quad\text{ and }\quad A_2=\{v\in\mathcal{W}_{c_0}^{2}: \text{Re}\Bigl(\frac{v_h}{v_{V}}\Bigl)<0\}.
\end{align*}
For any $v\in A_1$ and $\theta_1,\,\theta_2\in X$ with $\theta_1<\theta_2$,  using \eqref{for:61} we have:

\begin{align}\label{for:38}
&\bigl\lvert\pi_{U}(\text{Ad}_{z(\theta_2)^{-1}z(\theta_1)}\cdot v)\bigl\rvert\notag\\
&\geq\abs{v_{V}}\cdot\bigl\lvert\sin(2\theta_1-2\theta_2)\cdot\frac{v_h}{v_{V}}-\sin^2(\theta_1-\theta_2)\bigl\rvert-\abs{v_{U}}\notag\\
&\geq\abs{v_{V}}\cdot\bigl\lvert\sin(2\theta_2-2\theta_1)\text{Re}\Bigl(\frac{v_h}{v_{V}}\Bigl)+\sin^2(\theta_1-\theta_2)\bigl\rvert-\abs{v_{U}}\notag\\
&\geq \frac{1}{4}\bigl(\sin^2(\theta_2-\theta_1)-c_0\bigl)-\frac{4}{3}c_0\notag\\
&\geq \frac{1}{4}\delta^{2}\abs{\theta_2-\theta_1}^2-\frac{5}{3}c_0,
\end{align}
which implies that $A_1$ has $c_0^{1/2}$ disjoint property. Similar arguments hold for $A_2$, therefore $A_2$ also has $c_0^{1/2}$ disjoint property. Hence $\mathcal{W}_{c_0}$ has $c_0^{1/2}$ disjoint property.

For $k$ non-archimedean and $g=\left( \begin{array}{ccc}
1 & b \\
 0& 1\\
\end{array} \right)\in K$ with $b\in \mathcal{O}$, we have
\begin{align}\label{for:39}
\pi_{U}(\text{Ad}_g(v_{U}+v_h+v_{V}))=-2bv_h-b^2v_{V}+v_{U}.
\end{align}
For any $v\in \mathcal{W}_{c_0}^{1}$ we have
\begin{align*}
\abs{\pi_{U}(\text{Ad}_g(v))}&\geq
\abs{2bv_h+b^2v_{V}}-\abs{v_{U}}\geq \frac{\abs{2}}{2}\cdot\frac{3}{4}\abs{b}-2b^2-\frac{4}{3}c_0.
\end{align*}
Similar to \eqref{for:23}, $\mathcal{W}_{c_0}^{1}$  has $c_0$-disjoint property. We just need to consider $\mathcal{W}_{c_0}^{2}$. Using \eqref{for:39}, similar to \eqref{for:38} we see that for any $v\in\mathcal{W}_{c_0}^{2}$
\begin{align}\label{for:40}
&\abs{\pi_{U}(\text{Ad}_g(v))}\geq \frac{1}{2}\abs{2bv_hv_{V}^{-1}+b^2}-\frac{4}{3}c_0.
\end{align}
For any $0<c_0<1$ there exist
$\abs{q}^2< \eta_0\leq 1$ and $z_0\in \NN$ such that $c_0=\abs{q}^{z_0}\eta_0$ and $z_0$ is even. Suppose $i_0,\,i_1\in\ZZ$ such that $\abs{2}=\abs{q}^{i_0}$ and $\abs{q}^{i_1}>3$. We consider the decomposition of $\mathcal{W}_{c_0}^{2}$: $\mathcal{W}_{c_0}^{2}=\bigcup_{i=0}^{z_0+1}A_i$ where
\begin{align*}
A_{z_0+1}&=\{v\in\mathcal{W}_{c_0}^{2}: \abs{v_hv_{V}^{-1}}<2\abs{q}^{z_0+1}\}&\qquad&\text{ and }\\
A_{i}&=\{v\in\mathcal{W}_{c_0}^{2}: \abs{q}^{i+1}<\abs{v_hv_{V}^{-1}}<\abs{q}^{i-1}\},& \qquad &0\leq i\leq z_0.
\end{align*}
For any $v\in A_{z_0+1}$, it follows from \eqref{for:40} that
\begin{align*}
&\abs{\pi_{U}(\text{Ad}_g(v))}\geq \frac{1}{2}\abs{2bv_hv_{V}^{-1}+b^2}-\frac{4}{3}c_0\geq\frac{1}{2}b^2-Cc_0,
\end{align*}
which implies that $A_{z_0+1}$ has $c_0^{1/2}$ disjoint property.

For any $0\leq i\leq z_0$ and any $v\in A_{i}$,
\begin{align*}
\abs{2bv_hv_{V}^{-1}+b^2}=&\left\{\begin{aligned} &\abs{b}^2,\quad &\abs{b}>\abs{2}\cdot\abs{q}^{i-1}\\
&\abs{2bv_hv_{V}^{-1}},\quad &\abs{2}\cdot\abs{q}^{i+1}>\abs{b}.
 \end{aligned}
 \right.
\end{align*}
For  $A_i$ with $i\geq \frac{z_0}{2}$, if
\begin{align*}
\abs{b}\geq\abs{q}^{\frac{z_0}{2}-\abs{i_0}-1+i_1}\geq \abs{q}^{-\abs{i_0}-1+i_1}\cdot c_0^{1/2},
\end{align*}
 then $\abs{b}>\abs{2}\cdot\abs{q}^{i-1}$, and thus by \eqref{for:40} we have
\begin{align*}
\abs{\pi_{U}(\text{Ad}_g(v))}&\geq \frac{1}{2}\abs{b^2}-\frac{4}{3}c_0> 3\abs{q}^{z_0}-\frac{4}{3}c_0> 3c_0-\frac{4}{3}c_0=\frac{5}{3}c_0.
\end{align*}
Then we see that  $A_i$ has $c_0^{1/2}$ disjoint property if $i\geq\frac{z_0}{2}$.

Now suppose $0\leq i<\frac{z_0}{2}$.  If $\abs{b}>\abs{q}^{i-\abs{i_0}-1+i_1}$
then $\abs{b}>\abs{2}\cdot\abs{q}^{i-1}$, hence  we have
\begin{align*}
\abs{\pi_{U}(\text{Ad}_g(v))}&\geq \frac{1}{2}\abs{b^2}-\frac{4}{3}c_0> 3\abs{q}^{z_0}-\frac{4}{3}c_0> 3c_0-\frac{4}{3}c_0=\frac{5}{3}c_0;
\end{align*}
also if $\abs{q}^{z_0-i-\abs{i_0}+i_1-1}<\abs{b}<\abs{2}\cdot\abs{q}^{i+1}=\abs{q}^{i+i_0+1}$, then
\begin{align*}
\abs{\pi_{U}(\text{Ad}_g(v))}&\geq \frac{1}{2}\abs{2b}\cdot\abs{q}^{i+1}-\frac{4}{3}c_0\geq 3\abs{q}^{z_0}-\frac{4}{3}c_0\geq \frac{5}{3}c_0.
\end{align*}
Then we see that for $0\leq i<\frac{z_0}{2}$, if
\begin{align*}
\abs{q}^{z_0-i-\abs{i_0}+i_1-1}<\abs{b}<\abs{q}^{i+i_0+1}\quad\text{ or }\quad\abs{q}^{i-\abs{i_0}-1+i_1}<\abs{b}\leq\abs{q}^\sigma
\end{align*}
then $\abs{\pi_{U}(\text{Ad}_g(v))}>\frac{4}{3}c_0$. Furthermore,
since
\begin{align*}
\abs{q}^{z_0-i-\abs{i_0}+i_1-1}&<\abs{q}^{z_0-\frac{z_0}{2}-\abs{i_0}+i_1-1}= \abs{q}^{-\abs{i_0}+i_1-1}\eta_0^{-1/2}c_0^{1/2}\\
&\leq\abs{q}^{-\abs{i_0}+i_1-2}\cdot c_0^{1/2},
\end{align*}
it follows that $\abs{\pi_{U}(\text{Ad}_g(v))}>\frac{4}{3}c_0$, if
\begin{align*}
\abs{q}^{-\abs{i_0}+i_1-2}\cdot c_0^{1/2}\leq\abs{b}<\abs{q}^{i+i_0+1}\quad\text{ or }\quad\abs{q}^{i-\abs{i_0}-1+i_1}<\abs{b}\leq\abs{q}^\sigma.
\end{align*}
Hence $A_i$ also has $c_0^{1/2}$ disjoint property if $0\leq i<\frac{z_0}{2}$. Note that $\sum_{i=0}^{z_0+1}c_0^{1/2}\leq c_0^{1/2}(\log_{\abs{q}}c_0)$, then $\mathcal{W}_{c_0}^{2}$ has $c_0^{1/2}(\log_{\abs{q}}c_0)$ disjoint property. So,
\begin{itemize}
  \item for $k$ archimedean, $\gamma=-\frac{1}{2}$ and $-\frac{\gamma}{2}(\lambda-\varrho)=\frac{1}{2}\alpha_1$.
Then one has $p\bigl((G\ltimes_\rho V)_0\bigl)\leq 2$ for $k=\RR$ and $p\bigl((G\ltimes_\rho V)_0\bigl)\leq 4$ for $k=\CC$;
  \item for $k$ non-archimedean, from proof of Proposition \ref{po:6}, $\mathcal{W}_{c_0}$ has $c_0^{1/2}(\log_{\abs{q}}c_0)$ disjoint property means that for any $\pi\in \widehat{(G\ltimes_\rho V)}_0$ and for any $g=k_1a\omega k_2$
  \begin{align*}
\bigl\lvert\langle \pi(g)\eta,\xi\rangle\bigl\rvert&\leq C_{\eta,\xi}\abs{\alpha_1(a)}^{-1/2}\cdot\log_{\abs{q}}(\abs{\alpha_1(a)}^{-1})\\
&\overset{\text{(1)}}{\leq} c(\epsilon)C_{\eta,\xi}\abs{\alpha_1(a)}^{-\frac{1}{2}+\epsilon},
\end{align*}
for a dense set of vectors $\eta,\,\xi$ of $\pi$ and for any $\epsilon>0$. Here $c(\epsilon)>0$ is a constant only dependent on $\epsilon$. $(1)$ holds since $\log_{\abs{q}}c_0 \leq c(\epsilon)c_0^{-\epsilon}$  if $c_0$ is small enough. Hence $p\bigl((G\ltimes_\rho V)_0\bigl)\leq 2$.
\end{itemize}
\end{sect15}
\begin{sect15}\label{for:69} Let $L=\RR$, $\CC$, or $\HH$ and let $G$ be the group defined by the symmetric or Hermitian form $Q$ on $L^3$: $\left( \begin{array}{ccc}
0 & 1 & 0\\
1 & 0 &0 \\
0 & 0 &-1 \\
\end{array} \right)$. Then $G$ is $SO_0(1,2)$ for $L=\RR$, $SU(1,2)$ for $L=\CC$, or $Sp(1,2)$ for $L=\HH$.  Then $G$ is a real rank one group. Let $\rho$ be the standard representation of $G$ on $L^{3}$. We can realize $D$ as $D=\{a=\text{diag}(a_1,a_1^{-1},1):a_i\in \RR^\ast\}$ and the simple root is $\alpha_1(a)=a_1$; the weights are  $\lambda(a)=\alpha_1$, $\varrho=-\alpha_1$ and $0$.  By Theorem \ref{th:11},
\begin{enumerate}
\item for $G=SO_0(1,2)$, $\gamma=-\frac{1}{3}$ and $-\frac{\gamma}{2}(\lambda-\varrho)=\frac{1}{3}\alpha_1$. Using  $\delta_B(a)=\abs{\alpha_1(a)}$ we have $p\bigl((G\ltimes_\rho V)_0\bigl)\leq 3$;
  \item for $G=SU(1,2)$, $\gamma=-(\frac{1}{3})^2$ and $-\frac{\gamma}{2}(\lambda-\varrho)=\frac{1}{9}\alpha_1$. Using  $\delta_B(a)=\abs{\alpha_1(a)}^4$ we have $p\bigl((G\ltimes_\rho V)_0\bigl)\leq 36$;
  \item for $G=Sp(1,2)$, $\gamma=-(\frac{1}{3})^2$ and $-\frac{\gamma}{2}(\lambda-\varrho)=\frac{1}{9}\alpha_1$. Using  $\delta_B(a)=\abs{\alpha_1(a)}^{10}$ we see that $p\bigl((G\ltimes_\rho V)_0\bigl)\leq 90$.
\end{enumerate}
\textbf{Improvement}: Let $E=\left( \begin{array}{ccc}
0 & 1 & 0\\
1 & 0 &0 \\
0 & 0 &-1 \\
\end{array} \right)$, $M=\left( \begin{array}{ccc}
1 & 0 & 0\\
0 & -1 &0 \\
0 & 0 &-1 \\
\end{array} \right)$ and $T=\left( \begin{array}{ccc}
-1 & -\frac{1}{2} & 0\\
-1 & \frac{1}{2} &0 \\
0 & 0 & 1 \\
\end{array} \right)$ then we see that
\begin{align*}
Q(x,y)=\bar{x}^\tau\cdot E\cdot y=\bar{x}^\tau\cdot T^\tau\cdot M\cdot T\cdot y\quad\text{ for any }x,\,y\in L^{3}.
\end{align*}
Let $G'$ be the group defined by the standard symmetric or Hermitian form $M$ and let $K'$ be the subgroup in $G'$ with the form: $K'=\left( \begin{array}{ccc}
x & 0 & 0\\
0 & a & b\\
0 & c & d\\
\end{array} \right)$ where $x\in Sp(1)$ and $\left( \begin{array}{cc}
a & b\\
c & d\\
\end{array} \right)\in Sp(2)$ for $L=\HH$; and $x\in U(1)$ and $\left( \begin{array}{cc}
a & b\\
c & d\\
\end{array} \right)\in U(2)$ with $x\cdot \text{det}\left( \begin{array}{cc}
a & b\\
c & d\\
\end{array} \right)=1$ for $L=\CC$. Then a maximal compact subgroup $K$ in $G$ has the form: $K=T^{-1}K'T=\left( \begin{array}{ccc}
\frac{1}{2}(x+a)& \frac{1}{4}(x-a) & -\frac{1}{2}b\\
x-a & \frac{1}{2}(x+a) & b\\
-c & \frac{1}{2}c & d\\
\end{array} \right)$.  Denote by $\mathcal{P}$ the map from $K'\rightarrow K$: $\mathcal{P}(a)=T^{-1}aT$ for any $a\in K'$. Set
\begin{align*}
\mathcal{W}_{c_0}=\{v=(v_1,v_2,v_3)\in L^{3}:\abs{v_1}\leq c_0\text{ and }\frac{3}{4}\leq\abs{v}\leq\frac{5}{4}\}.
\end{align*}
Then $\mathcal{W}_{c_0}\subset\mathcal{W}_{c_0}^{1}\bigcup \mathcal{W}_{c_0}^{2}$, where
\begin{align*}
\mathcal{W}_{c_0}^{1}&=\{v\in L^{3}:\abs{v_1}< \frac{4}{3}c_0,\,\,\frac{1}{2}<\abs{v}<\frac{3}{2} \text{ and }\abs{v_2}> \abs{v_3}\}\qquad\text{ and}\\
\mathcal{W}_{c_0}^{2}&=\{v\in L^{3}:\abs{v_1}< \frac{4}{3}c_0,\,\,\frac{1}{2}<\abs{v}<\frac{3}{2} \text{ and }\abs{v_2}<\frac{4}{3}\abs{v_3}\}.
\end{align*}
For $\mathcal{W}_{c_0}^{2}$ consider $f(a,z)=\left( \begin{array}{ccc}
a+z & 0 & 0\\
0 & a & z \\
0 & -\bar{z} & a \\
\end{array} \right)\in K'$, where $a\in\RR$, $z\in L$ with $a^2+\abs{z}^2=1$ and $z=-\bar{z}$. Then $f(a_1,z_1)^{-1}f(a_2,z_2)$ has the form
\begin{align*}
\left( \begin{array}{ccc}
a_1a_2+a_1z_2-a_2z_1-z_1z_2 & 0 & 0\\
0 & a_1a_2-z_1z_2 & a_1z_2-a_2z_1 \\
0 & a_1z_2-a_2z_1 & a_1a_2-z_1z_2 \\
\end{array} \right),
\end{align*}
and then $\mathcal{P}\bigl(f(a_1,z_1)^{-1}f(a_2,z_2)\bigl)\in K$ has the form
\begin{align*}
\left( \begin{array}{ccc}
z & \frac{1}{4}(a_1z_2-a_2z_1) & -\frac{1}{2}(a_1z_2-a_2z_1)\\
a_1z_2-a_2z_1 & z & a_1z_2-a_2z_1 \\
-a_1z_2+a_2z_1 & \frac{1}{2}(a_1z_2-a_2z_1) & a_1a_2-z_1z_2 \\
\end{array} \right)
\end{align*}
where $z=\frac{1}{2}(2a_1a_2+a_1z_2-a_2z_1-2z_1z_2)$. Set $X=\{x\in L:x=-\bar{x}\}$ and
\begin{align*}
A=\{(a,z)\in\RR\times X:\frac{1}{4}\leq a,\abs{z}\leq\frac{\sqrt{15}}{4}\text{ and }a^2+\abs{z}^2=1\}.
\end{align*}
Using (*) in Example \ref{for:21}, we also have: \begin{enumerate}
\item [(**)] if $(a_1,z_1),\,(a_2,z_2)\in A$ satisfying $\abs{z_1-z_2}\geq 6\cdot 15^2\cdot 64\cdot 8c_0$, then $\abs{a_1z_2-a_2z_1\big}\geq 6\cdot 16c_0$.
\end{enumerate}
For any $v\in \mathcal{W}_{c_0}^{2}$ and $(a_1,z_1),\,(a_2,z_2)\in A$ satisfying $\abs{z_1-z_2}\geq 6\cdot 15^2\cdot 64\cdot 8c_0$, let $z=\frac{1}{2}(2a_1a_2+a_1z_2-a_2z_1-2z_1z_2)$, then
\begin{align*}
&\bigl\lvert\pi_1\bigl(\mathcal{P}\bigl(f(a_1,z_1)^{-1}f(a_2,z_2)\bigl)v\bigl)\bigr\lvert\\
&=\Bigl\lvert zv_1+\frac{1}{4}(a_1z_2-a_2z_1)v_2 -\frac{1}{2}(a_1z_2-a_2z_1)v_3\Bigr\lvert\\
&\geq \frac{1}{4}\abs{a_1z_2-a_2z_1}\cdot\bigl\lvert v_2 -2v_3\bigr\lvert-\abs{v_1}\\
&\geq \frac{1}{4}\cdot\frac{2}{3}\abs{a_1z_2-a_2z_1}\cdot\abs{v_3}-\frac{4}{3}c_0\\
&>\frac{1}{6}\cdot 6\cdot 16c_0\cdot\frac{1}{4}-\frac{4}{3}c_0\geq \frac{5}{3}c_0
\end{align*}
where $\pi_1$ projects $v=(v_1,v_2,v_3)\in L^3$ to $v_1$. Hence arguments similar to Example \ref{for:21} show that $\mathcal{W}_{c_0}^{2}$ has $c_0^3$-disjoint property for $L=\HH$; and $\mathcal{W}_{c_0}^{2}$ has $c_0$-disjoint property for $L=\CC$.

Set  $L_1=\{x\in L:\abs{x}=1\}$. For $\mathcal{W}_{c_0}^{1}$ consider $h(a)=\left( \begin{array}{ccc}
1 & 0 & 0\\
0 & a & 0 \\
0 & 0 & \bar{a} \\
\end{array} \right)\in K'$, where $a\in L_1$. For any $a_1,\,a_2\in L_1$, $\mathcal{P}\bigl(h(a_1)^{-1}h(a_2)\bigl)\in K$ has the form
\begin{align*}
\left( \begin{array}{ccc}
\frac{1}{2}(1+a_1^{-1}a_2) & \frac{1}{4}(1-a_1^{-1}a_2) & 0\\
1-a_1^{-1}a_2& \frac{1}{2}(1+a_1^{-1}a_2) & 0 \\
0 & 0 & a_1a_2^{-1} \\
\end{array} \right).
\end{align*}
Notice that there exists $\delta>0$ such that for any $a_1,\,a_2\in L_1$ $\abs{1-a_1^{-1}a_2}\geq\delta\abs{a_1-a_2}$.

For any $v\in \mathcal{W}_{c_0}^{1}$ and $a_1,\,a_2\in L_1$ satisfying $\abs{a_1-a_2}\geq 48\delta^{-1}c_0$, we have
\begin{align*}
\bigl\lvert\pi_1\bigl(\mathcal{P}\bigl(h(a_1)^{-1}h(a_2)\bigl)v\bigl)\bigr\lvert&=\Bigl\lvert\frac{1}{4}\bigl(1-a_1^{-1}a_2\bigl)v_2+\frac{1}{2}\bigl
(1+a_1^{-1}a_2\bigl)v_1\Bigr\lvert\\
&\geq \frac{\delta}{4}\abs{a_1-a_2}\cdot\abs{v_2}-\frac{4}{3}c_0\\
&>\frac{\delta}{4}\cdot 48\delta^{-1}c_0\cdot\frac{1}{4}-\frac{4}{3}c_0\geq \frac{5}{3}c_0
\end{align*}
Hence an argument analogous to one for Example \ref{for:21} shows that $\mathcal{W}_{c_0}^{1}$ has $c_0^3$-disjoint property for $L=\HH$; and $\mathcal{W}_{c_0}^{1}$ has $c_0$-disjoint property for $L=\CC$. Then we see that $\mathcal{W}_{c_0}$ has $c_0^3$-disjoint property for $L=\HH$; and $\mathcal{W}_{c_0}$ has $c_0$-disjoint property for $L=\CC$. So,
\begin{itemize}
  \item for $G=SU(1,2)$ $\gamma=-1$ and
$-\frac{\gamma}{2}(\lambda-\varrho)=\alpha_1$.
Then $p\bigl((G\ltimes_\rho V)_0\bigl)\leq4$.
  \item for $G=Sp(1,2)$,
$\gamma=-3$ and
$-\frac{\gamma}{2}(\lambda-\varrho)=3\alpha_1$.
Then $p\bigl((G\ltimes_\rho V)_0\bigl)\leq\frac{10}{3}$.
\end{itemize}
For $G=SO_0(1,2)$, it is isomorphic to the projective group $PSL(2,\RR)$ and $\rho$ is isomorphic to the adjoint representation of $PSL(2,\RR)$ on $\mathfrak{g}=\mathfrak{sl}(2,\RR)$. From the relation
\begin{align*}
SL(2,\RR)\overset{j_1}{\longrightarrow} PSL(2,\RR)\overset{\rho}\longrightarrow \mathfrak{sl}(2,\RR),
\end{align*}
where $j_1$ is the natural projection. We see that $\rho\circ j_1$ is isomorphic to the adjoint representation of $SL(2,\RR)$. It is clear that arguments in Example \ref{for:67} also hold for $SO_0(1,2)$, which implies that
\begin{itemize}
  \item for $G=SO_0(1,2)$, $p\bigl((G\ltimes_\rho V)_0\bigl)\leq 2$.
 \end{itemize}

\end{sect15}

\begin{sect15}\label{ex:1} Let $L=\RR$, $\CC$, or $\HH$ and let $G_n$, $n\geq 2$ be the real Lie group defined by the symmetric or Hermitian form form $Q$ on $L^{n+1}$: $\left( \begin{array}{ccc}
0 & 1 & 0\\
1 & 0 &0 \\
0 & 0 &-I_{n-1} \\
\end{array} \right)$  where $I_{n-1}$ denotes the $(n-1)\times (n-1)$-identity matrix. Then $G_n$ is $SO_0(1,n)$ for $L=\RR$, $SU(1,n)$ for $L=\CC$, or $Sp(1,n)$ for $L=\HH$.  Let $\rho$ be the standard representation of $G$ on $L^{n+1}$. We can realize $D$ as $D=\{a=\text{diag}(a_1,a_1^{-1},1,\cdots,1):a_i\in \RR^\ast\}$ and the simple root is $\alpha_1(a)=a_1$; the weights are  $\lambda(a)=\alpha_1$, $\varrho=-\alpha_1$ and $0$. By Theorem \ref{th:11},
\begin{enumerate}
\item for $G=SO_0(1,n)$, $\gamma=-\frac{1}{3}$ and $-\frac{\gamma}{2}(\lambda-\varrho)=\frac{1}{3}\alpha_1$. Using  $\delta_B(a)=\abs{\alpha_1(a)}^{n-1}$ we have $p\bigl((G\ltimes_\rho V)_0\bigl)\leq 3(n-1)$ ;
  \item for $G=SU(1,n)$, $\gamma=-(\frac{1}{3})^2$ and $-\frac{\gamma}{2}(\lambda-\varrho)=\frac{1}{9}\alpha_1$. Using  $\delta_B(a)=\abs{\alpha_1(a)}^{2n}$ we have $p\bigl((G\ltimes_\rho V)_0\bigl)\leq 18n$;
  \item for $G=Sp(1,n)$, $\gamma=-(\frac{1}{3})^2$ and $-\frac{\gamma}{2}(\lambda-\varrho)=\frac{1}{9}\alpha_1$. Using  $\delta_B(a)=\abs{\alpha_1(a)}^{4n+2}$ we see that $p\bigl((G\ltimes_\rho V)_0\bigl)\leq 9(2+4n)$.
\end{enumerate}
\textbf{Improvement}: Let $\mathcal{P}_{i_1,i_2,i_3}:G_2\hookrightarrow G_n$, $1\leq i_1,i_2,i_3\leq n$ be the natural embedding which maps $A=(a_{i,j})\in G_2$ to $(b_{i,j})\in G$ where $b_{i_j,i_l}=a_{j,l}$  and the other diagonal elements are $1$ and off-diagonal elements are $0$.  Set
\begin{align*}
\mathcal{W}_{c_0}=\{v=(v_1,\cdots,v_{n+1})\in L^{n+1}:\abs{v_1}\leq c_0\text{ and }3/4\leq\abs{v}\leq5/4\}.
\end{align*}
Then $\mathcal{W}_{c_0}\subset\bigcup_{i=2}^{n+1} \mathcal{W}_{c_0}^{i}$, where
\begin{align*}
\mathcal{W}_{c_0}^{i}=\{v\in L^{n+1}:\abs{v_1}< \frac{4}{3}c_0,\,\frac{1}{2}<\abs{v}<\frac{3}{2}\text{ and }\abs{v_i}> \frac{1}{4\sqrt{n}}\}.
\end{align*}
For $i=2,\,3$, consider the embedding $\mathcal{P}_{1,2,3}$; or consider embedding $\mathcal{P}_{1,2,i}$ for $4\leq i\leq n+1$. Since for any $v=(v_1,\cdots,v_{n+1})\in L^{n+1}$,  $\pi_j\bigl(\mathcal{P}_i(g\cdot v)\bigl)=v_j$ for any $g\in A$,  any $j\neq 1,\,2,\,i$, where $\pi_j$ denotes the $j$-th coordinate projection,
an argument analogous to the one in Example \ref{for:69} shows that
\begin{itemize}
\item for $G=SO_0(1,n)$, each $\mathcal{W}_{c_0}^{i}$ has $c_0^{1/2}$-disjoint property, then $\gamma=-\frac{1}{2}$ and
$-\frac{\gamma}{2}(\lambda-\varrho)=\frac{1}{2}\alpha_1$.
Then $p\bigl((G\ltimes_\rho V)_0\bigl)\leq 2(n-1)$.
  \item for $G=SU(1,n)$, each $\mathcal{W}_{c_0}^{i}$ has $c_0$-disjoint property, then $\gamma=-1$ and
$-\frac{\gamma}{2}(\lambda-\varrho)=\alpha_1$.
Then $p\bigl((G\ltimes_\rho V)_0\bigl)\leq 2n$.
  \item for $G=Sp(1,n)$,
$\gamma=-3$ and
$-\frac{\gamma}{2}(\lambda-\varrho)=3\alpha_1$.
Then $p\bigl((G\ltimes_\rho V)_0\bigl)\leq\frac{2+4n}{3}$.
\end{itemize}
\end{sect15}
Recall notations in Section \ref{sec:2}. We shall make use of a general strategy of Howe \cite{howe} in the following example. Let $H$ be an almost $k$-algebraic subgroup of $G$, such that $B(H) = B\bigcap H$, $D(H) = D\bigcap H$ and $K(H) = K\bigcap H$ are a minimal parabolic subgroup, a maximal split torus, and a good maximal compact subgroup of $H$ respectively,  and the corresponding positive Weyl chamber $D^+(H)$ in $D(H)$ contains $D^+$. Suppose $H=H_1\times \cdots \times H_m$, a product of reductive subgroups. Then in obvious
notations one has $D(H)=D(H_1)\times \cdots \times D(H_m)$ and $K(H)=K(H_1)\times \cdots \times K(H_m)$.
\begin{lemma}\label{le:8}Let $\pi$ be a unitary representation of $G$.
Suppose that $\pi_{H_i}$ is $(K(H_i),\Psi_i)$ bounded on $H_i$, $1\leq i\leq m$. Then $\pi$ itself is $(K, r\prod_{i=1}^m \Psi_i)$ bounded on $G$ where $r=\max_{\omega\in F}[K:K\bigcap \omega K\omega^{-1}]$ (note that $F=\{e\}$ for $k$ archimedean).
\end{lemma}
For $k$ archimedean this result follows directly from Theorem \ref{th:10}, Proposition $6.3$ of \cite{howe} and the considerations of \cite{howe}, $\S 8$; for $k$ non-archimedean a detailed proof is given in \cite{oh}.
\begin{sect15}\label{for:37}Let $G=Sp(2n,k)$ over a local field $k$, which is defined by the
bi-linear form: $\left( \begin{array}{ccc}
0 & I_n \\
-I_n & 0 \\
\end{array} \right)$. Let $\rho$ be the standard representation of $G$ on $k^{2n}$. Take the maximal compact subgroup $K$ to be the intersections of those of $SL_{2n}(k)$ with $G$. We can realize $D$ as
 \begin{align*}
 D=\{a=\text{diag}(a_1,\cdots, a_n,a_1^{-1},\cdots, a_n^{-1}):a_i\in k^\ast\}.
 \end{align*}
The simple roots are $\alpha_i(a)=a_ia_{i+1}^{-1}$, $1\leq i\leq n-1$ and $\alpha_n(a)=a_n^{2}$; the weights are $\pm\lambda_i$ where $\lambda_i(a)=a_i$, $1\leq i\leq n$. Then $\lambda(a)=\lambda_1(a)=a_1$ and $\varrho=-\lambda_1(a)=a_1^{-1}$. Then $\gamma=-(\frac{1}{3})^{2n-2}$ and $-\frac{\gamma}{2}(\lambda-\varrho)=\frac{1}{3^{2n-2}}\lambda_1$. Using 
\begin{align*}
\delta_B(a)=\prod_{i=1}^{n}\abs{a_i}^{4(n+1-i)}\text{ for }k=\CC\quad\text{ or }\quad\prod_{i=1}^{n}\abs{a_i}^{2(n+1-i)}\text{ for }k\neq\CC, 
\end{align*}
by Theorem \ref{th:11} $p\bigl((G\ltimes_\rho V)_0\bigl)\leq 2\cdot 3^{2n-2}n(n+1)$ for $k=\CC$ and $p\bigl((G\ltimes_\rho V)_0\bigl)\leq 3^{2n-2}n(n+1)$ for $k\neq\CC$. \\ \textbf{Improvement}: for any $a\in D^+$, denote by $c_i=\abs{\lambda_i(a^{-1})}$, $1\leq i\leq n$. Set
\begin{align*}
\mathcal{W}_{c_1,\cdots,c_n}=\bigl\{v=(v_1,\cdots,v_{2n})\in k^{2n}:\abs{v_i}\leq c_i,\,1\leq i\leq n,\text{ and }\frac{3}{4}\leq\abs{v}\leq\frac{5}{4}\bigl\}.
\end{align*}
Then $\mathcal{W}_{c_1,\cdots,c_n}\subset\bigcup_{i=1}^{n}\mathcal{W}_{c_1,\cdots,c_n}^{i}$, where
\begin{align*}
\mathcal{W}_{c_1,\cdots,c_n}^{i}=\bigl\{v\in k^{2n}:\abs{v_l}<\frac{4}{3} c_l,\,1\leq l\leq n,\text{ and }\frac{1}{2}<\abs{v}<\frac{3}{2},\,\abs{v_{n+i}}> \frac{1}{4\sqrt{n}}\bigl\}.
\end{align*}
Fix $i$, $1\leq i\leq n$. Suppose $k$ is archimedean. Set $X=\{(a,z)\in \RR\times k:a^2+\abs{z}^2=1\}$. For $1\leq j\leq n$ define $h_j(a,z)=(a_{l,m})\in K$ and $f_i(a,z)=(b_{l,m})\in K$ as follows: $a_{i,j+n}=a_{j,i+n}=z$, $a_{j+n,i}=a_{i+n,j}=-\overline{z}$, $a_{i,i}=a_{j+n,j+n}=a_{j,j}=a_{i+n,i+n}=a$; and the other diagonal elements are $1$ and off-diagonal elements are $0$; $b_{i,i+n}=z$, $b_{i+n,i}=-\overline{z}$, $b_{i,i}=b_{i+n,i+n}=a$; and the other diagonal elements are $1$ and off-diagonal elements are $0$.

For $\mathcal{W}_{c_1,\cdots,c_n}^{j}$, $j\neq i$ consider $h_j$ and for $\mathcal{W}_{c_1,\cdots,c_n}^{i}$, consider $f_i$. Then for any $v\in \mathcal{W}_{c_1,\cdots,c_n}^{j}$, $j\neq i$ and any $(a_m,z_m)\in X$, $m=1,\,2$ we have
\begin{align}\label{for:95}
\bigl\lvert\pi_i\big(h_j(a_1,z_1)^{-1}h_j(a_2,z_2)v\big)\bigl\rvert&=\bigl\lvert(a_1z_2-a_2z_1)v_j+(a_1a_2+z_1\overline{z_2})v_i\bigl\rvert;
\end{align}
where $\pi_i$ projects $v=(v_1,\cdots,v_{2n})\in k^{2n}$ to $v_i$.

Also, for any $v\in \mathcal{W}_{c_1,\cdots,c_n}^{i}$ and any $(a_m,z_m)\in X$, $m=1,\,2$ we have
\begin{align}\label{for:98}
\bigl\lvert\pi_i\big(f_i(a_1,z_1)^{-1}f_i(a_2,z_2)v\big)\bigl\rvert&=\bigl\lvert(a_1z_2-a_2z_1)v_j+(a_1a_2+z_1\overline{z_2})v_i\bigl\rvert.
\end{align}
Now suppose $k$ is non-archimedean. For $\mathcal{W}_{c_1,\cdots,c_n}^{j}$, $1\leq j\leq n$ consider $h_j(z)=(a_{l,m})$ where $z\in \mathcal{O}$: $a_{i,j+n}=a_{j,i+n}=z$ and the diagonal elements are $1$ and the other off-diagonal elements are $0$. Then for any $v\in \mathcal{W}_{c_1,\cdots,c_n}^{i}$ and any $z_1,\,z_2\in \mathcal{O}$ we have
\begin{align*}
\bigl\lvert\pi_i\big(h(z_1)^{-1}h(z_2)v\big)\bigl\rvert&=\bigl\lvert v_i+(z_2-z_1)v_j\bigl\rvert.
\end{align*}
An argument similar to the one in Example \ref{for:21} shows that $\mathcal{W}_{c_1,\cdots,c_n}^{j}$ has $c_i^\delta$-disjoint property where $\delta=1$ if $k\neq\CC$ and $\delta=2$ if $k=\CC$. Also, $\mathcal{W}_{c_1,\cdots,c_n}$ has $c_i^\delta$-disjoint property, $1\leq i\leq n$. Then follow the proof line in Proposition \ref{po:6}, for any $\pi\in \widehat{(G\ltimes_\rho V)}_0$ and for any  $g=k_1a\omega k_2$, one has
  \begin{align}\label{for:100}
\bigl\lvert\langle \pi(g)\eta,\xi\rangle\bigl\rvert\leq C_{\eta,\xi}\abs{\lambda_i(a)\cdot(-\lambda_i)(a^{-1})}^{- \delta/2}=C_{\eta,\xi}\abs{\lambda_i(a)}^{-\delta}
\end{align}
for a dense set of vectors $\eta,\,\xi$ of $\pi$. Let $H_{2\lambda_i}$ be the subgroup of $G$ with Lie algebra generated by $\{\mathfrak{g}_{2\lambda_i},\mathfrak{g}_{-2\lambda_i}\}$, $1\leq i\leq n$. Then $H_{2\lambda_i}$ is isomorphic to $SL(2,k)$ and $\delta_{B(H_{2\lambda_i})}(a)=\abs{\lambda_i(a)}^{2\delta}$.
Therefore $\pi$ is strongly $L^{2+\epsilon}$ on $H_{2\lambda_i}$ and Theorem \ref{th:10} implies that $\pi$ is $(K(H_{2\lambda_i}),\Xi_{H_{2\lambda_i}})$ bounded on $H_{2\lambda_i}$, $1\leq i\leq n$. Then it follows from Lemma \ref{le:8} that any $\pi\in \widehat{(G\ltimes_\rho V)}_0$ is $(K,\,r\prod_{i=1}^n\Xi_{H_{2\lambda_i}})$ bounded on $G$ and from Proposition \ref{po:2}, it is clear that for any $\epsilon>0$, there exist constants $c_1$ and $c_2(\epsilon)$ such that for all $g=k_1a\omega k_2$
\begin{align*}
c_1\prod_{i=1}^n\abs{\lambda_i(a)}^{-\delta}\leq\prod_{i=1}^n\Xi_{H_{2\lambda_i}}(g)\leq c_2(\epsilon)\prod_{i=1}^n\abs{\lambda_i(a)}^{-\delta+\epsilon}.
\end{align*}
Hence we also get $p\bigl((G\ltimes_\rho V)_0\bigl)\leq 2n$.
\end{sect15}

\begin{sect15} Let $L=\CC$ or $\HH$ and let $G$ be the real Lie group defined by the symmetric or Hermitian form form $Q$ on $L^{n+m}$: $\left( \begin{array}{ccc}
0 & I_n & 0\\
I_n & 0 &0 \\
0 & 0 &-I_{m-n} \\
\end{array} \right)$ where $m\geq n\geq 2$. Then $G$ is $SU(n,m)$ for $L=\CC$ or $Sp(n,m)$ for $L=\HH$.  Let $\rho$ be the standard representation of $G$ on $L^{n+m}$. We assume notations in Example \ref{for:37} if there is no confusion. We can realize $D$ as
\begin{align*}
D=\{a=\text{diag}(a_1,\cdots, a_n,a_1^{-1},\cdots, a_n^{-1},1,\cdots,1):a_i\in \RR^\ast\}
\end{align*}
and simple roots are $\alpha_i(a)=a_ia_{i+1}^{-1}, 1\leq i\leq n-1$; and $\alpha_n(a)=a_n$  if $m>n$ or $\alpha_n(a)=\frac{1}{2}a_n$ if $m=n$. The weights are $\pm\lambda_i$ where $\lambda_i(a)=a_i$, $1\leq i\leq n$. By the definition of $\delta_B$, we have
\begin{align*}
\delta_B(a)=\prod_{i=1}^{n}\abs{a_i}^{2(n+m-2i)+2}\text{ for }k=\CC\quad\text{or}\quad\prod_{i=1}^{n}\abs{a_i}^{4n+4m+6-8i}\text{ for }k=\HH.
\end{align*}
%\begin{align*}
%\delta_B(a)=\left\{\begin{aligned} &\bigl(\prod_{i=1}^{n-1}\abs{\alpha_i(a)}^{(4n-2i)i}\bigl)\abs{\alpha_n(a)}^{(n-1)(n+1)}&\quad & \text{ for }L=\CC\\
%&\bigl(\prod_{i=1}^{n-1}\abs{\alpha_i(a)}^{(8n-4i+2)i}\bigl)\abs{\alpha_n(a)}^{(2n-1)(n+1)}&\quad & \text{ for }L=\HH.
% \end{aligned}
 %\right.
%\end{align*}
We use the same notations $\mathcal{W}_{c_1,\cdots,c_n}$ and $\mathcal{W}_{c_1,\cdots,c_n}^j$ as in Example \ref{for:37} by changing $k$ to $L$ and extend domain of $j$ to $1\leq j\leq m$. 

Fix $i$, $1\leq i\leq n$. For $\mathcal{W}_{c_1,\cdots,c_n}^{j}$ with $i\neq j\leq n$, consider $h_j(a,z)=(a_{l,m})$ where $(a,z)\in X$: $a_{i,j+n}=a_{i+n,j}=z$, $a_{j+n,i}=a_{j+n,i}=-\overline{z}$, $a_{i,i}=a_{j+n,j+n}=a_{j,j}=a_{i+n,i+n}=a$; and the other diagonal elements are $1$ and off-diagonal elements are $0$.

For $\mathcal{W}_{c_1,\cdots,c_n}^{i}$, consider $f_i(a,z)=(a_{l,m})$ where $(a,z)\in X$ and $\overline{z}=-z$:  $a_{i,i+n}=a_{i+n,i}=z$ and $a_{i,i}=a_{i+n,i+n}=a$; and the other diagonal elements are $1$ and off-diagonal elements are $0$. Then \eqref{for:95} and \eqref{for:98}
also hold. Similarly, for $1\leq j\leq n$ $\mathcal{W}_{c_1,\cdots,c_n}^j$ has $c_i^\delta$-disjoint property where $\delta=1$ for $k=\CC$ and $\delta=3$ for $k=\HH$, which implies that  \eqref{for:100} holds. 

For $\mathcal{W}_{c_1,\cdots,c_n}^{j}$ with $n<j\leq m$, consider the embedding $\mathcal{P}_{i,i+n,2n+j}:G_1\hookrightarrow G$ where $G_1$ is isomorphic to $SU(1,2)$ for $L=\CC$ or isomorphic to $Sp(1,2)$ for $L=\HH$ (we use the same notation $\mathcal{P}_{i,i+n,2n+j}$ as in Example \ref{ex:1}). An argument analogous to the one in Example \ref{ex:1} shows that $\mathcal{W}_{c_1,\cdots,c_n}^{j}$ with $n<j\leq m$ has $c_i^\delta$-disjoint property. Hence we deduce that $\mathcal{W}_{c_1,\cdots,c_n}$ has $c_i^\delta$-disjoint property for $1\leq i\leq m$.

Note that $\delta_{B(H_{2\lambda_i})}(a)=\abs{\lambda_i(a)}^{2}$ for $L=\CC$ and $\delta_{B(H_{2\lambda_i})}(a)=\abs{\lambda_i(a)}^{6}$ for $L=\HH$.
Therefore $\pi$ is strongly $L^{2+\epsilon}$ on $H_{2\lambda_i}$ and Theorem \ref{th:10} implies that $\pi$ is $(K(H_{2\lambda_i}),\Xi_{H_{2\lambda_i}})$ bounded on $H_{2\lambda_i}$, $1\leq i\leq n$. Then for $g=k_1ak_2$
\begin{align*}
c_1\prod_{i=1}^n\abs{\lambda_i(a)}^{-\delta}\leq\prod_{i=1}^n\Xi_{H_{2\lambda_i}}(g)&\leq c_2(\epsilon)\prod_{i=1}^n\abs{\lambda_i(a)}^{-\delta+\epsilon} .
\end{align*}
Hence we get $p\bigl((G\ltimes_\rho V)_0\bigl)\leq 2(n+m-1)$ for $L=\CC$ and $p\bigl((G\ltimes_\rho V)_0\bigl)\leq \frac{4m+4n-2}{3}$ for $L=\HH$.
\end{sect15}

\begin{sect15}Suppose rank$_kG=1$ and $\rho$ is irreducible and excellent on $V$. By Theorem \ref{th:11}, $\gamma=-1/3^{(\sharp\Phi_1-1)}$;  if $\dim V_\lambda=1$, $\gamma=-1/3^{(\sharp\Phi_1-2)}$. Set $\Lambda(\Phi_1)=-\frac{\gamma}{2}(\lambda-\varrho)$. We deduce from \eqref{for:119} that
\begin{align*}
p\bigl((G\ltimes_\rho V)_0\bigl)\leq p=\frac{\text{the coefficient of $\alpha$ in $\delta_B$}}{\text{the coefficient of $\alpha$ in $\Lambda(\Phi_1)$}}
\end{align*}
where $\alpha$ is the simple root of $G$.\\
\textbf{Improvement}: Choose an ordering of the weights, then we have a decomposition of $V$:
\begin{align*}
V=V_{\lambda=\phi_1}\oplus V_{\phi_2}\oplus\cdots\oplus V_{\varrho=\phi_n},
\end{align*}
and the wights are ordered in the way $\phi_1<\phi_2<\cdots<\phi_n$.

Let $\beta=\phi_{(\frac{\sharp\Phi_1}{2})}$ if $\sharp\Phi_1$ is even and $\beta=\phi_{\frac{\sharp\Phi_1-1}{2}}$ if $\sharp\Phi_1$ is odd. Then $\beta$ is the biggest positive weight and $-\beta$ is the smallest negative weight. Denote by $V_+$ the subspace of $V$ spanned by the positive weights and by $\pi_+$ the projection from $V$ to $V_+$. Let $\mathcal{W}_{c_0}'=\{v\in V:\abs{\pi_+v}\leq c_0\text{ and }\frac{3}{4}\leq\abs{v}\leq\frac{5}{4}\}$.

Then follow line by line the proof in Section \ref{sec:11} and note that we begin from $\beta$ instead of $\lambda$,  we see that $\mathcal{W}_{c_0}'$ has $c_0^{\gamma'}$ disjoint property, where $\gamma'=-(\frac{1}{3})^{(\frac{\sharp\Phi_1}{2})}$ if $\sharp\Phi_1$ is even and $\gamma'=-(\frac{1}{3})^{(\frac{\sharp\Phi_1+1}{2})}$ if $\sharp\Phi_1$ is odd. Moreover, if $\dim V_{\lambda}=1$, $\gamma'=-(\frac{1}{3})^{(\frac{\sharp\Phi_1-2}{2})}$ if $\sharp\Phi_1$ is even and $\gamma'=-(\frac{1}{3})^{(\frac{\sharp\Phi_1-1}{2})}$ if $\sharp\Phi_1$ is odd. Set $\Lambda'(\Phi_1)=-\frac{\gamma'}{2}(\beta-(-\beta))$. By \eqref{for:119} of Corollary \ref{cor:4}, using $\beta$ and $-\beta$ instead of $\lambda$ and $\varrho$ we have
\begin{align}\label{for:53}
p\bigl((G\ltimes_\rho V)_0\bigl)\leq p'=\frac{\text{the coefficient of $\alpha$ in $\delta_B$}}{\text{the coefficient of $\alpha$ in $\Lambda'(\Phi_1)$}}
\end{align}
Note that $\lambda=(\frac{\sharp\Phi_1+1}{2})\alpha$, $\varrho=-(\frac{\sharp\Phi_1+1}{2})\alpha$ and $\beta=\frac{1}{2}\alpha$ if $\sharp\Phi_1$ is even; or $\lambda=(\frac{\sharp\Phi_1-1}{2})\alpha$, $\varrho=-(\frac{\sharp\Phi_1-1}{2})\alpha$ and $\beta=\alpha$ if $\sharp\Phi_1$ is odd, then $\frac{p'}{p}=\frac{\sharp\Phi_1+1}{3^{(\frac{\sharp\Phi_1-2}{2})}}$ if $\sharp\Phi_1$ is even and  $\frac{p'}{p}=\frac{\sharp\Phi_1-1}{2\cdot 3^{(\frac{\sharp\Phi_1-3}{2})}}$ if $\sharp\Phi_1$ is odd. Then we see that $p'< p$ if $\sharp\Phi_1\neq 2,\,3,\,4$. Then \eqref{for:53} is a better estimate than $p$ given from Theorem \ref{th:11}.
\end{sect15}
\begin{remark}\label{re:3}
Vogan's classification of unitary duals for $GL_n(L)$, yields that for $G=SL_n(L)$, $L=\RR$, $\CC$ or $\HH$, $p(G_0)$ is $2(n-1)$, $2(n-1)$ and $2(n-1)$ respectively and for $L$ a non-archimedean local field, $p(G_0)$ is $2(n-1)$ (see \cite{oh}). In Example \ref{for:21}, we get the same result for $SL(3,L)$. For $Sp(2n,\CC)$, it follows from Howe's result in \cite{howe} that $p(Sp(2n,\CC)_0)=4n$. For $G=Sp(m,n)$, $m\geq n\geq 2$ or $G=SU(m,n)$, $m\geq n\geq 2$ the
exact number $p(G_0)$ is $2(m+n)-1$ or $2(m+n-1)$ respectively obtained by Li \cite{Li}. For $G=Sp(n,1)$, $n\geq 2$ the desired exponent $p(G_0)=2n+1$ follows from classification
\cite{Ba} and \cite{Ko}. For real Lie groups $SO_0(n,1)$ and $SU(n,1)$ $n\geq 2$, and for $SL(2,k)$ over a local field $k$, $p(G_0)$ are $\infty$ since they don't have property $(T)$. We list an upper bound $p$ of the number $p\bigl((G\ltimes_\rho V)_0\bigl)$ for some examples in Section \ref{sec:19} in the following table.\\
\begin{tabular}{cc|c|c|c|c|c|l}
  $G$:& $SL(2,k)$ & $SU(1,n)$ & $SO_0(n,1)$ & $Sp(1,n)$ & $Sp(2n,\CC)$ & $Sp(n,m)$\\
   $p$: & $2\delta(k)$ & $2n$ & $2(n-1)$ & $\frac{2+4n}{3}$ & 2n & $\frac{4n+4m-2}{3}$\\
   \end{tabular}\\
where $m$ is the smallest integer such that $m\geq \frac{p}{2}$ and $\delta(k)=2$ for $k=\CC$ or $\delta(k)=1$ otherwise.

Let $m(G_0)$ be the smallest integer such that $m(G_0)\geq \frac{1}{2}\cdot p\bigl((G\ltimes_\rho V)_0\bigl)$. From \eqref{for:41} of Proposition \ref{po:2} we see that $\frac{\xi(\Phi)}{m(G_0)}$ determines the decay rate of $\Xi_G^{1/m(G_0)}$. For above examples the condition $m(G_0)<\frac{1}{2}\cdot p(G_0)$ is satisfied. Then as an obvious consequence of \eqref{for:42} of Proposition \ref{po:2}, Theorem \ref{th:10} and Lemma \ref{le:7}, any $\pi\in \widehat{(G\ltimes_\rho V)}_0$ is $\bigl(K,\,\Xi_G^{1/m(G_0)}\bigl)$ bounded on $G$; while there exists $\pi_1\in \widehat{G_0}$ such that $\pi_1$ is not $\bigl(K,\,\Xi_G^{1/m}\bigl)$ bounded on $G$ for any $m<\frac{1}{2}\cdot p(G_0)$. Furthermore, if any $\pi\in \widehat{G_0}$ is $\bigl(K,\,\Psi\bigl)$ bounded on $G$ then Proposition \ref{th:2} means that any $\pi\in \widehat{(G\ltimes_\rho V)}_0$ is $\bigl(K,\,\min\{\Psi,\,\Xi_G^{1/m(G_0)}\}\bigl)$ bounded on $G$, which is a much sharper pointwise bound than $\Psi$. Hence we conclude that the minimal rate of decay of $K$-matrix coefficients on $G$ (from the consideration of $\rho(K)$-orbits) is much sharper than the one obtained from the semisimple part itself by using Kazhdan's property $(T)$.

Oh showed in \cite{oh} that for $G=Sp(2n,\CC)$ the best possible decay rate
of $K$-finite matrix coefficients is determined by $\sum_{i=1}^n \lambda_i$ by using Howe's trick and the one obtained by using $p\bigl((G\ltimes_\rho V)_0\bigl)\leq 2n$ is $\frac{1}{m}\cdot\xi(\Phi)=\sum_{i=1}^n\frac{2(n+1-i)}{n}\lambda_i$. Then by above arguments a better decay rate is
\begin{align*}
\max\bigl\{\sum_{i=1}^n \lambda_i,\sum_{i=1}^n\frac{2(n+1-i)}{n}\lambda_i\bigl\}=\sum_{i=1}^{[\frac{n}{2}]+1}\frac{2(n+1-i)}{n}\lambda_i+\sum_{i=[\frac{n}{2}]+2}^{n}\lambda_i.
\end{align*}
Hence we see that $\sum_{i=1}^n 2\lambda_i$, the one provided by the combination of $\rho(K)$-orbit-deformation method  and Howe's trick is the best compared to the above results  (see Example \ref{for:37}).
\end{remark}

%Consider the $n$-fold tensor product
%$\mathcal{H}\otimes\mathcal{H}\cdots\otimes\mathcal{H}=\mathcal{H}^{\otimes m}$ where $m\geq \frac{p}{2}$.
%let $\pi^{\otimes m}$ denote the action of $H\ltimes_\rho V$ on $\mathcal{H}^{\otimes m}$ via the $m$-fold tensor product
%of $\pi$. Let $\xi_n'=\underbrace{\xi_n\otimes\xi_n\cdots\otimes\xi_n}_{m \text{ copies}}$ then
%\begin{align*}
   %\phi_{\xi_n',\xi_n'}'(g)=\langle\pi^{\otimes n}(g)\xi_n',\xi_n'\rangle'=\langle\pi(g)\xi_n,\xi_n\rangle^m\qquad\text{for }\forall g\in H,
%\end{align*}
%where $\langle\cdot\rangle'$ indicates the inner product on $\mathcal{H}^{\otimes m}$. It follows from H\"{o}lder's inequality that
%$\phi_{\xi_n',\xi_n'}'\in L^{2+\epsilon}(H)$ for any $\epsilon>0$. Denote $\mathcal{H}'$ \foot{cyclic representation not irreducible }the closed subspaces spanned by
%$\{\pi^{\otimes m}(g)\xi_n':g\in H\}$.
%Then by Theorem \ref{th:10} $\pi^{\otimes m}$ is strongly $L^{2+\epsilon}(H)$ restricted to $\mathcal{H}'$ and furthermore
%\begin{align*}
%\abs{\langle\pi(g)\xi_n,\xi_n\rangle}^m&=\abs{\langle\pi^{\otimes n}(g)\xi_n',\xi_n'\rangle'}\\
%&\leq
%(\norm{\xi_n'}')^2\Xi_H(g)=\norm{\xi_n}^{2m}\Xi_H(g)\qquad\text{for any }g\in H
%\end{align*}
%where $\norm{\cdot}'$ indicates the norm on $\mathcal{H}^{\otimes m}$.

%Taking $m$-th roots yields
 %\begin{align*}
%\abs{\phi_{\xi_n,\xi_n}(g)}\leq \norm{\xi_n}^{2}\Xi^{\frac{1}{m}}_H(g)\qquad\text{for any }g\in H.
%\end{align*}
\section{Kazhdan constants}\label{sec:13} In this section, we discuss some applications of the above results in terms of a quantitative
estimate of Kazhdan's property $(T)$ of the pair $(G\ltimes V,V)$, namely, Kazhdan constant.
\begin{definition}
For a locally compact group $S$, we say that a unitary representation $\pi$ of $S$ almost
has an invariant vector if for any $\epsilon>0$ and any compact subset $Q$ of $S$, there exists
a unit vector $v$ which is $(Q,\epsilon)$-invariant, that is,
\begin{align*}
    \sup_{g\in Q}\norm{\pi(g)v-v}\leq\epsilon.
\end{align*}
A pair $(S,S')$ is said to have Kazhdan's property $(T)$ where $S'$ is a closed subgroup of $S$ if any unitary representation of $S$ which
almost has an invariant vector actually has a non-zero $S'$-invariant vector.
\end{definition}
\begin{remark}
$S$ is a Kazhdan group if and only if the pair $(S,S)$ has
property $(T)$. An immediate consequence of the above definition is that for any $S'_2\subset S'_1\subset S_1\subset S_2$
if $(S_1,S'_1)$ has Kazhdan's property $(T)$ then $(S_2,S'_2)$ has Kazhdan's property $(T)$
too. In particular if $S$ or $S'$ is a Kazhdan group then the pair $(S,S')$ has Kazhdan's property
$(T)$.
\end{remark}
\begin{definition}
For a locally compact group $S$ with a compact subset
$Q$, a positive number $\epsilon$ is said to be a Kazhdan constant for $((S,S'),Q)$ if
\begin{align*}
\tau((S,S'),Q):=\inf_{\pi\in\mathcal{R}_{S'}}\inf_{v\in \mathcal{H}_{\pi}^1}\max_{g\in Q}\norm{\pi(g)v-v}\geq\epsilon
\end{align*}
where $\mathcal{R}_{S'}$ is the set of unitary representations of $S$ without non-trivial $S'$-invariant vectors and $\mathcal{H}_{\pi}^1$ are the set of unit vectors
in the attached Hilbert space for $\pi$. If there exists such an $\epsilon$, we call $Q$ a Kazhdan set for $S$.
\end{definition}
In other words, if $\epsilon$ is a Kazhdan constant for $(S,Q)$, then any unitary representation
of $S$ which has a $(Q, \epsilon)$-invariant vector actually has a non-zero invariant vector.
\begin{proposition}\label{po:8}
If $\rho$ is good, then the pair ($G\ltimes_\rho V$, $V$) has Kazhdan's property $(T)$. Furthermore, a Kazhdan constant for $((G\ltimes_\rho V,V),Q_{(h_1,\cdots,h_N)})$ is $\frac{1}{\sqrt{N}}\cdot\min_{1\leq i\leq N}\kappa_{l_i}(h_i)$, where $Q_{(h_1,\cdots,h_N)}$ is defined in \eqref{for:150} and $\kappa_{l_i}$ is defined in  \eqref{for:54}.
\end{proposition}
\begin{proof}
Let $\pi$ be any representation of $G\ltimes_\rho V$ without non-trivial $V$-fixed vectors. Denote by $\mathcal{H}$ the attached Hilbert space of $\pi$. Since char $k=0$, then by full reducibility of semisimple groups there is a natural decomposition of $V$ under representation $\rho$: $V=\oplus_{i=1}^N V_i$
such that for each $i$, $V_i$ is an irreducible representation of $G$. Next, we will show:
\begin{enumerate}
  \item [(*)] there exists a subset $X\subset \{1,\cdots, N\}$ such that there is an orthogonal decomposition of $\mathcal{H}:\mathcal{H}=\bigoplus_{i\in X} E_i$ such that each $E_i$ is invariant under $G\ltimes_\rho V$ and there is no non-trivial $V_i$-fixed vectors on $E_i$.
\end{enumerate}
For each $1\leq i\leq N$, let $\mathcal{H}_i=\{v\in\mathcal{H}: v\text{ is fixed by }V_i\}$.
Similar to \eqref{for:33} we see that $\mathcal{H}_i$ is $G\ltimes_\rho V$-invariant. Hence $\mathcal{H}^{\bot}_i$, the orthogonal complement of $\mathcal{H}_i$ is also $G\ltimes_\rho V$-invariant. Since there is no non-trivial $V$-fixed vectors, there exists $ 1\leq i_1\leq N$ such that $\mathcal{H}^{\bot}_{i_1}\neq 0$. Let $E_{i_1}=\mathcal{H}^{\bot}_{i_1}$. If $\mathcal{H}^{\bot}_{i_1}= \mathcal{H}$, then we get (*). Otherwise, on $\mathcal{H}_{i_1}$, analogous to above arguments, there is $ 1\leq i_2\neq i_1\leq N$ such that there is a $G\ltimes_\rho V$-invariant orthogonal decomposition of $\mathcal{H}_{i_1}:\mathcal{H}_{i_1}=E_{i_2}\bigoplus E_{i_2}'$ such that $E_{i_2}\neq 0$ and there is no non-trivial $V_{i_2}$-fixed vectors on $E_{i_2}$ and $E_{i_2}'$ is fixed by $V_{i_2}$. If $E_{i_2}=\mathcal{H}_{i_1}$, then we also get (*).  Otherwise, we repeat the above procedure on $E_{i_2}'$. Note that $E_{i_2}'$ is fixed by $V_{i_1}$ and $V_{i_2}$. Since there is no non-trivial $V$-fixed vectors, this procedure stops after at most $N-1$ steps. So (*) is proved.

Recall notations in Definition \ref{def:2}. Denote by $G_i$, $1\leq i\leq j$ the non-compact almost $k$-simple factors of $G$. Let $w^l_i=\{v\in V_i: v\text{ is fixed by }G_l\}$. Notice that each $G_l$, $1\leq l\leq j$ is a normal subgroup in $G$, then
\begin{align*}
\Pi(g_l)(\Pi(g)\nu)&=\Pi(g)\Pi(g^{-1}g_lg)\nu=\Pi(g)\nu,\qquad \forall \nu\in w^l_i,\,\forall g\in G,\,\forall g_l\in G_l,
\end{align*}
for any $\Pi\in \mathcal{R}_{V}$ which implies that $w^l_i$, $1\leq l\leq j$ is $G$-invariant. By irreducibility of $V_i$, $w^l_i=0$ or $w^l_i=V_i$ for each $l$. Since the product $\prod_{i=1}^jG_i$ is either equal to $G_s$ or is Zariski dense in $\tilde{G}_s$ (see \cite{Margulis}), $\rho$ is also good on $\prod_{i=1}^jG_i$. Then  for any $1\leq i\leq N$ there exists a non-empty subset $A_i$ of $\{1,\cdots,j\}$ such that $w^l_i=0$ for any $l\in A_i$.
Since $G_l$ is excellent on $V_i$, $p_{(G_l,V_i,\Phi_{l,i})}<\infty$ where $\Phi_{l,i}$ is the set of weights of $G_l$ on $V_i$. Let $p_{(l,i)}$ be the smallest integer such that $2p_{(l,i)}\geq p_{(G_l,V_i,\Phi_{l,i})}$.
Denote by
\begin{align}\label{for:56}
l_i\in A_i \text{ such that }p_{l_i}=\min_{l\in A_i}p_{(l,i)}.
\end{align}
Next, we consider the restriction of $\pi\mid_{G_{l_i}\ltimes V_i}$ for $i\in X$. Then by Theorem \ref{th:11}, for any $K_{l_i}$-invariant unit vectors $v$ and $w$ in $E_i$ where $K_{l_i}$ is a good maximal compact subgroup in $G_{l_i}$, one has
\begin{align}\label{for:147}
\abs{\langle \pi(g)v,w\rangle}\leq
\Xi_{G_{l_i}}^{1/p_{l_i}}(g)\qquad\text{for any }g\in G_{l_i}.
\end{align}
For any unit vector $\nu\in E_i$ and  any $h\in G_{l_i}$ such that $h\notin K_{l_i}$,
%and $\Xi_{G_{l_i}}(h)\neq 1$,
we will show that
\begin{align}\label{for:54}
 \max_{s\in \{h,K_{l_i}\}}\norm{\pi(s)\nu-\nu}\geq \kappa_{l_i}(h)= \frac{\sqrt{2\bigl(1-\Xi^{1/p_{l_i}}_{G_{l_i}}(h)\bigl)}}{\sqrt{2\bigl(1-\Xi^{1/p_{l_i}}_{G_{l_i}}(h)\bigl)}+3}.
\end{align}
Suppose for all $\tau\in K_{l_i}$, we have $\norm{\pi(\tau)\nu-\nu}\leq\kappa_{l_i}(h)$. We will show that
$\norm{\pi(h)\nu-\nu}\geq \kappa_{l_i}(h)$. Let $\nu_1$ be the average of the $K_{l_i}$-transform of $\nu$:
\begin{align*}
\nu_1=\int_{K_{l_i}}\pi(\tau)\nu d\tau,
\end{align*}
where $d\tau$ is the normalized Haar measure on $K_{l_i}$. We compute
\begin{align*}
    \norm{\nu-\nu_1}\leq\kappa_{l_i}(h),\quad\text{ so that }\quad\norm{\nu_1}\geq 1-\kappa_{l_i}(h).
\end{align*}
Since $\kappa_{l_i}(h)<1$, the inequality implies that $\nu_1$ is non-zero. Note that for any unit
vector $w$,
\begin{align*}
    \norm{\pi(g)w-w}^2=2-2\text{Re}\langle\pi(g)w,w\rangle,\qquad \forall g\in G.
\end{align*}
Hence by using \eqref{for:147} we have
\begin{align*}
\Big\|\pi(h)\bigl(\frac{\nu_1}{\norm{\nu_1}}\bigl)-\frac{\nu_1}{\norm{\nu_1}}\Bigl\|\geq \sqrt{2\bigl(1-\Xi_{G_{l_i}}^{1/p_{l_i}}(h)\bigl)}
\end{align*}
and then
\begin{align*}
\norm{\pi(h)\nu_1-\nu_1}&\geq \sqrt{2\bigl(1-\Xi_{G_{l_i}}^{1/p_{l_i}}(h)\bigl)}\cdot\norm{\nu_1}\\
&\geq\sqrt{2\bigl(1-\Xi_{G_{l_i}}^{1/p_{l_i}}(h)\bigl)}\cdot(1-\kappa_{l_i}(h)).
\end{align*}
Therefore
\begin{align*}
\norm{\pi(h)\nu-\nu}&=\norm{\pi(h)\nu_1-\nu_1+(\pi(h)\nu-\pi(h)\nu_1)+(\nu_1-\nu)}\\
&\geq\norm{\pi(h)\nu_1-\nu_1}-2\norm{\nu_1-\nu}\\
&\geq\sqrt{2\bigl(1-\Xi_{G_{l_i}}^{1/p_{l_i}}(h)\bigl)}\cdot(1-\kappa_{l_i}(h))-2\kappa_{l_i}(h)\\
&=\kappa_{l_i}(h),
\end{align*}
which implies \eqref{for:54}.

For any $h_i\in G_{l_i}$ such that $h_i\notin K_{l_i}$
%and $\Xi_{G_{l_i}}(h_i)\neq 1$
set
\begin{align}\label{for:150}
Q_{(h_1,\cdots,h_N)}=\bigcup_{i=1}^N\{K_{l_i},h_i\}.
\end{align}
Fix a unit vector $\nu$ of $\pi$. Using (*) there is an orthogonal decomposition of $\nu:\nu=\sum_{i\in X}\nu_i$. There exists $i_0\in X$ such that
$\norm{\nu_{i_0}}=\max_{i}\norm{\nu_i}\geq\frac{1}{\sqrt{N}}$. Then it follows from \eqref{for:54} that
\begin{align}\label{for:59}
 \inf_{\pi\in\mathcal{R}_{V}}\max_{s\in Q_{(h_1,\cdots,h_N)}}\norm{\pi(s)\nu-\nu}\geq \frac{1}{\sqrt{N}}\cdot\min_{1\leq i\leq N}\kappa_{l_i}(h_i), 
\end{align}
which proves that the pair ($G\ltimes_\rho V$, $V$) has Kazhdan's property $(T)$ and a Kazhdan constant for $((G\ltimes_\rho V,V),Q_{(h_1,\cdots,h_N)})$ is $\frac{1}{\sqrt{N}}\cdot\min_{1\leq i\leq N}\kappa_{l_i}(h_i)$.
\end{proof}
\begin{remark}
Any compact
generating subset of $G\ltimes_\rho V$ is a Kazhdan set if the pair ($G\ltimes_\rho V$, $V$) has Kazhdan's property $(T)$ (see \cite[Ch
1, Proposition 15]{Valette2}). The above proposition yields
examples of Kazhdan sets which are contained in a proper closed semisimple subgroup of
$G$ (see Example \ref{for:57}).
\end{remark}
\begin{sect15}\label{for:57}
Suppose $\rho$ is excellent on $V$ and denote by $G_i$, $1\leq i\leq j$ the non-compact almost $k$-simple factors of $G$. Let $p_{i}$ be the smallest integer such that $2p_{i}\geq \min_lp_{(G_i,V_l,\Phi_{i,l})}$ where $\Phi_{i,l}$ is the set of weights of $G_i$ on $V_l$. Then from the proof of Proposition \ref{po:8}, we see that for any $h\in G_{i}$ such that $h\notin K_{i}$, a Kazhdan constant for the set $((G\ltimes_\rho V,V),\{K_i,h\})$ is $\frac{\sqrt{2\bigl(1-\Xi^{1/p_i}_{G_{i}}(h)\bigl)}}{\sqrt{2\bigl(1-\Xi^{1/p_i}_{G_{i}}(h)\bigl)}+3}$.
\end{sect15}

\begin{sect15}
Suppose rank$_kG=1$ and $\rho$ is irreducible and good on $V$. Denote by $G_1$ is the non-compact almost $k$-simple factor of $G$ and by $K_1$ a good maximal compact subgroup in $G_1$. Let $m$ be the smallest integer such that $2m\geq p'$ where $p'$ is defined in \eqref{for:53}. For any $h\in G_{1}$ such that $h\notin K_{1}$, $\frac{\sqrt{2\bigl(1-\Xi^{1/m}_{G_{1}}(h)\bigl)}}{\sqrt{2\bigl(1-\Xi^{1/m}_{G_{1}}(h)\bigl)}+3}$ is a Kazhdan constant for $((G\ltimes_\rho V,V),\{K_1,h\})$.

\end{sect15}

\end{document}